\documentclass{amsart}

\usepackage{amsmath}
\usepackage{amssymb}
\usepackage{amsthm}
\usepackage{amsfonts}
\usepackage{mathtools}
\usepackage{enumitem}
\usepackage{lscape}
\usepackage{caption}

\usepackage{wasysym}

\newcommand{\dominateseq}{\mathrel{\mathord{\unrhd}}}

\usepackage{etex}
\PassOptionsToPackage{usenames,dvipsnames}{xcolor}
\usepackage{xcolor}
\usepackage{tikz}
\usetikzlibrary{calc}

\usepackage{booktabs}
\usepackage{multirow}
\usepackage{multicol}
\usepackage{xspace}

\usepackage{hyperref}
\hypersetup{
    colorlinks=true,
    linkcolor=red,
}

\usepackage{shuffle}

\usepackage{soul}

\usepackage[OT1]{fontenc}
\usepackage[utf8]{inputenc}
\usepackage[english]{babel}
\usepackage{fancyhdr}
\usepackage{datetime}

\allowdisplaybreaks
\usepackage{aliascnt}
\makeatletter
\def\newaliasedtheorem#1[#2]#3{%
  \newaliascnt{#1@alt}{#2}
  \newtheorem{#1}[#1@alt]{#3}
  \expandafter\newcommand\csname #1@altname\endcsname{#3}
  \aliascntresetthe{#1@alt}
  \expandafter\newcommand\csname #1autorefname\endcsname{#3}
}
\makeatother

\theoremstyle{plain}
\newtheorem{Theorem}{Theorem}

\newaliasedtheorem{Corollary}[Theorem]{Corollary}
\newaliasedtheorem{Lemma}[Theorem]{Lemma}
\newaliasedtheorem{Proposition}[Theorem]{Proposition}

\theoremstyle{definition}
\newaliasedtheorem{Definition}[Theorem]{Definition}

\theoremstyle{remark}
\newaliasedtheorem{Example}[Theorem]{Example}
\newaliasedtheorem{Remark}[Theorem]{Remark}

\def\sectionautorefname~#1\null{\S#1\null}
\def\subsectionautorefname~#1\null{\S#1\null}
\def\subsubsectionautorefname~#1\null{\S#1\null}
\def\equationautorefname~#1\null{Equation~(#1)\null}
\def\itemautorefname~#1\null{#1\null}

\usepackage[labelfont=normalfont,labelformat=simple]{subcaption}

\usepackage{xargs}
\usepackage[colorinlistoftodos,prependcaption,textsize=tiny,linecolor=red,backgroundcolor=red!25,bordercolor=red]{todonotes}
\newcommandx{\franco}[2][1=]{\todo[linecolor=green,backgroundcolor=green!25,bordercolor=green,#1]{FS: #2}}
\newcommandx{\ton}[2][1=]{\todo[linecolor=yellow,backgroundcolor=yellow!25,bordercolor=yellow,#1]{TD: #2}}

\usepackage{comment}

\newcommand{\mb}{\mathbb}

\newcommand{\mf}{\mathfrak}

\newcommand{\inv}{^{-1}}
\def\binomial(#1,#2){{#1\choose #2}}
\newcommand{\mathid}{{\boldsymbol1}}
\newcommand{\symm}{\mf{S}}

\newcommand{\cf}{\textit{cf.} }
\newcommand{\ie}{\textit{i.e.}}

\newcommand\Rand{\operatorname{\mathsf{R2R}}}
\newcommand\RTop{\operatorname{\mathsf{R2T}}}
\newcommand\TopR{\operatorname{\mathsf{T2R}}}
\newcommand\Int{\operatorname{\mathsf{sh}}}

\newcommand\Ex{\operatorname{\partial}}

\newcommand\Repl{\operatorname{\Theta}}

\newcommand\proj{\operatorname{\mathsf{proj}}}
\newcommand\isoproj{\operatorname{\mathsf{isoproj}}}
\newcommand\projlift[1][a]{\operatorname{\mathsf{projlift}}_{#1}^{\lambda,\mu}}
\newcommand\projliftoper{\operatorname{\mathsf{projlift}}}

\newcommand\im{\operatorname{\mathsf{im}}}

\newcommand\ColStab{\operatorname{\mathsf{ColStab}}}
\newcommand\sign{\operatorname{\mathsf{sign}}}
\newcommand\Ind{\operatorname{Ind}}

\newcommand\Lift{\mathcal{L}}

\DeclareSymbolFont{ugrf@m}{U}{eur}{m}{n}
\DeclareMathSymbol{\updelta}{\mathord}{ugrf@m}{"0E}
\newcommand\KroneckerDelta{\operatorname{\updelta}}

\newcommand\shape{\operatorname{shape}}
\newcommand\size{\operatorname{size}}
\newcommand\diagonalindex{\operatorname{diag}}
\newcommand\evaluation{\operatorname{eval}}

\newcommand\spn{\operatorname{span}}
\newcommand\SSYT{\mathsf{SSYT}}
\newcommand\SYT{\mathsf{SYT}}

\def\mathgreen#1{\ensuremath{\boldsymbol{\color{ForestGreen}#1}}}
\def\mathred#1{\ensuremath{\boldsymbol{\color{BrickRed}#1}}}
\def\mathblue#1{\ensuremath{\boldsymbol{\color{RoyalBlue}#1}}}

\newcommand\ZZ{\mathbb{Z}}
\newcommand\NN{\mathbb{N}}
\newcommand\CC{\mathbb{C}}

\newcommand\MM[1][n]{M^{\langle #1 \rangle}}
\newcommand\mult{\operatorname{mult}}

\newcommand\word{\operatorname{word}}
\newcommand\row{\operatorname{row}}
\newcommand\w{\operatorname{\mathsf{w}}}

\let\Langle\langle
\let\Rangle\rangle
\def\langle{\left\Langle}
\def\rangle{\right\Rangle}

\usepackage{stmaryrd}
\newcommand{\decreasingrearrangement}[1]{\overset{{}_{\leftarrow}}{#1}}

\newcommand{\FreeAssociativeAlgebra}[1][A]{\mathbb{C}\langle A \rangle}

\usepackage[vcentermath]{youngtab}
\def\smallyoung(#1){\ensuremath{\footnotesize\young(#1)}}
\newcommand{\eig}{\operatorname{eig}}

\newcommand{\ndestab}{d}

\def\IWpartial{\partial}
\def\IWdelta{\delta}
\def\IWModule{\mathcal M}
\def\IWLaplacian{\Lambda}

\def\IWComplex{\mathcal K}
\def\IWPoset{P}

\def\unprotectedboldentry#1{\textcolor{Red}{\large{#1}}}
\def\boldentry{\protect\unprotectedboldentry}
\usepackage{tikz}
\usetikzlibrary{calc}
\usepackage{ifthen}
\newcommand{\tikztableauinternal}[1]{
    \def\newtableau{#1}
    \coordinate (x) at (-0.5,0.5);
    \coordinate (y) at (-0.5,0.5);
    \foreach \row in \newtableau {
        \coordinate (x) at ($(x)-(0,1)$);
        \coordinate (y) at (x);
        \foreach \entry in \row {
            \ifthenelse{\equal{\entry}{X}}
               {
                \node (y) at ($(y) + (1,0)$) {};
                \fill[color=gray!10] ($(y)-(0.5,0.5)$) rectangle +(1,1);
                \draw[color=gray, dotted] ($(y)-(0.5,0.5)$) rectangle +(1,1);
               }
               {
                \ifthenelse{\equal{\entry}{\boldentry X}}
                   {
                    \node (y) at ($(y) + (1,0)$) {};
                    \fill[color=gray] ($(y)-(0.5,0.5)$) rectangle +(1,1);
                    \draw ($(y)-(0.5,0.5)$) rectangle +(1,1);
                   }
                   {
                    \node (y) at ($(y) + (1,0)$) {\entry};
                    \draw ($(y)-(0.5,0.5)$) rectangle +(1,1);
                   }
               }
            }
        }
}

\newcommand{\tikztableau}[2][scale=0.6,every node/.style={font=\small}]{
    \begin{array}{c}
    \begin{tikzpicture}[#1]
        \tikztableauinternal{#2}
    \end{tikzpicture}
    \end{array}}
\newcommand{\tikztableausmall}[1]{\tikztableau[scale=0.45,every node/.style={font=\rm\small}]{#1}}
\newcommand{\tikztableauscript}[1]{\tikztableau[scale=0.30,every node/.style={font=\scriptsize}]{#1}}

\def\eigenvaluestableheader{
        $\lambda/\mu$ &
        $\ndestab^\mu$ &
        $f^\lambda$ &
        multiplicity &
        $\binom{|\lambda|+1}{2}$ &
        $\binom{|\mu|+1}{2}$ &
        $\diagonalindex(\lambda/\mu)$ &
        $\eig(\lambda/\mu)$
}

\begin{document}

\title[Spectral analysis of random-to-random Markov chains]%
    {Spectral analysis of \\random-to-random Markov chains}

\author{A. B. Dieker}
   \address{Industrial Engineering and Operations Research, Columbia University, 500 W 120th St, New York, NY 10027}
   \email{dieker@columbia.edu}

\author{F. V. Saliola}
   \address{Laboratoire de Combinatoire et d'Informatique Math\'ematique (LaCIM) \\
           Universit\'e du Qu\'ebec \`a Montr\'eal \\
            CP 8888, Succ. Centre-ville \\
           Montr\'eal (Qu\'ebec) H3C 3P8 \\
          Canada}
   \email{saliola.franco@uqam.ca}

\thanks{The work of the second author was supported by an NSERC Discovery
    Grant.}

\date{Compiled on {\today} at {\currenttime}}

\begin{abstract}
We compute the eigenvalues and eigenspaces of random-to-random Markov chains.
We use a family of maps which reveal a remarkable recursive structure of the eigenspaces,
yielding an explicit and effective construction of all eigenbases starting from bases of the kernels.
\end{abstract}

\maketitle

\setcounter{tocdepth}{2}
\tableofcontents

\newpage

\section{Introduction}

Random-to-random Markov chains, also known as \emph{random insertion}
\cite{SaloffCosteZuniga2008} and \emph{random-to-random insertion}
\cite{Subag2013}, describe the random evolution of $n$ (ordered) objects,
some of which may be identical,
if someone repeatedly removes an object at random
and puts it back at a random position.
One can think of the objects as being books on a shelf, entries in a database,
characters in a word, or cards in a deck of cards.
As such, the random-to-random Markov chain constitutes
a canonical card shuffling model, and can be thought of as sequentially applying
a random-to-top shuffle (choose a card at random and move it to the top)
and a top-to-random shuffle (move the top card to a random position)
at each step.

There are several fundamental questions concerning random-to-random shuffles
that have withstood analysis. This is quite striking considering that these
same questions have been answered for the related random-to-top and
top-to-random shuffles, which play an instrumental role in the theory of card
shuffling \cite{DiaconisBook} and the theory of random walks on groups and
semigroups \cite{DiaconisSaloffCoste1995, BHR1999, Brown2000, SaloffCoste2004,
MargolisSaliolaSteinbergJEMS, MargolisSaliolaSteinbergAMS}.
This paper provides answers to some of these questions and new tools to
investigate others.

\medskip
\emph{Our results.}
This paper studies random-to-random Markov chains through a spectral lens.
We introduce a family of maps that allow us to
compute the eigenvalues and eigenspaces of random-to-random Markov chains.
Our maps reveal a remarkable recursive structure of the eigenspaces, 
and yield an explicit and effective construction of all eigenbases starting from bases of the kernels.

\autoref{fig:overview} gives a high-level description of how our maps exploit the recursive
nature of random-to-random Markov chains to construct its eigenbases.
This paper introduces the maps $\mathcal L_i$ depicted in the figure,
explains which of them generate eigenspaces of the larger Markov chain, and describes
when two different eigenspaces of the smaller chain generate the same eigenspace of the larger chain.
Furthermore, this paper shows how to selectively apply our maps to generate all non-kernel eigenspaces 
{\em exactly} once from random-to-random kernels.

\begin{figure}
    \captionsetup{width=0.9\textwidth}
    \centering
    \includegraphics[width=0.85\textwidth]{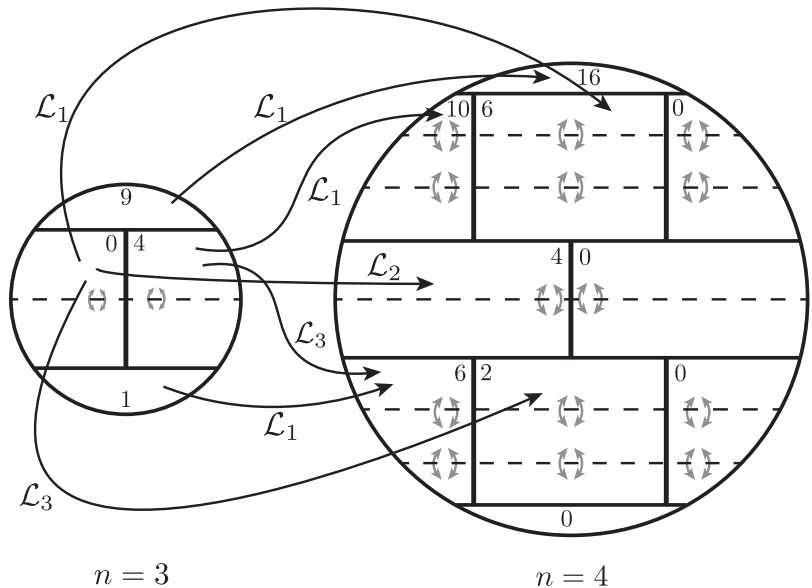}
\caption{An abstraction of the recursive structure of the eigenspaces for $n=3$ and $n=4$, 
highlighting the role of the maps $\mathcal L_i$ introduced in this paper.
Each cell represents a one-dimensional eigenspace of the random-to-random Markov chain.
Eigenspaces that share dashed borders are rotations of each other when exploiting well-understood permutation symmetries.
The number in a cell is the eigenvalue corresponding to its eigenspace.
The kernel is the only eigenspace that is not generated by these maps.}
\label{fig:overview}
\end{figure}

We also give two combinatorial descriptions of the eigenvalues of random-to-random Markov chains.
The first model indexes the eigenvalues by pairs of integer partitions called
\emph{horizontal strips}. The eigenvalues and their multiplicities are obtained
using a combinatorial statistic defined on the horizontal strips.
The second model directly outputs an eigenvalue for each arrangement of the
deck of cards. All eigenvalues (counting multiplicities) are obtained in this way by
considering all possible arrangements of the deck of cards, i.e., the states
of the random-to-random Markov chain.

\medskip
\emph{Implications of our results.}
Our results have several immediate implications, and may lead to further applications.

\begin{enumerate}[wide, label=\emph{\arabic*)}, itemsep=0.5em]
\item
We establish that all eigenvalues of random-to-random Markov chains are
integers, up to an explicit multiplicative scaling factor.
This has been previously conjectured in the Stanford University PhD thesis of
Jay-Calvin Uyemura-Reyes \cite[Section~5.2]{Reyes2002}.
Using Fourier analysis on symmetric groups, Uyemura-Reyes determines
eigenvalues and eigenvectors for various representations of the symmetric
groups (including the sign representation, the permutation representation, and
the standard representation
\cite[Proposition~5.3,~Theorems~5.4--5.5]{Reyes2002}), and makes several
explicit conjectures on the eigenvalues for other representations
\cite[Conjectures~5.6--5.11]{Reyes2002}. This leads him to conjecture that all
the eigenvalues are integers \cite[Conjecture~5.2]{Reyes2002}. Our results
settle all these conjectures.

\item
Our results have implications for the Laplacians associated with the complex of injective words.
Indeed, Hanlon and Hersh \cite{HanlonHersh2004} uncovered a close connection
between random-to-random Markov chains and these Laplacians.
They show that the aforementioned integrality property of the
eigenvalues of random-to-random Markov chains establishes that the
spectra of the Laplacians associated with the complex of injective words are also integral.
See \autoref{cor:laplacians}.

\item
Our results give a new perspective on the spectral analysis of symmetrized versions of
Bidigare--Hanlon--Rockmore random walks on the faces of a hyperplane
arrangement, which leads to intriguing questions.
This seminal theory was initiated by Bidigare, Hanlon and Rockmore (BHR)
\cite{BHR1999} and further developed by Brown and Diaconis
\cite{BrownDiaconis1998,Brown2000,Brown2004}. They showed
that several popular Markov chains are random walks on the faces of
a hyperplane arrangement. This includes the random-to-top shuffle,
which shows that the random-to-random shuffle can be obtained by
``symmetrizing'' a BHR random walk: sequentially apply the random-to-top
shuffle and its inverse shuffle (\ie, the top-to-random shuffle).

As first observed by Uyemura-Reyes \cite[Section~5.2.3]{Reyes2002}, and further
reinforced by the work of Reiner--Saliola--Welker \cite{RSW2014}, the
eigenvalues of other families of symmetrized BHR shuffles also seem to exhibit
nice properties. The techniques of the present paper can be used to analyse
the spectra of these shuffles.
For instance, Reiner--Saliola--Welker studied a family of symmetrized BHR
operators that commute with the random-to-random operator. Our results can be
used to give a second proof that these operators commute and construct their
eigenspaces (see \autoref{rem:commutativity-of-RSW-operators}). Possibly our
approach can be generalized for other symmetrized versions of BHR walks with
a symmetric distribution on the faces.

\item

Exploiting our results, Bernstein and Nestoridi~\cite{BernsteinNestoridi2017}
have settled an open problem on the so-called
\emph{mixing time} of random-to-random Markov chains, i.e., how many shuffles
are needed to approach the stationary distribution. Most of the existing
literature on random-to-random Markov chains focuses on this problem. The first
result of this kind is by Diaconis and Saloff-Coste
\cite{DiaconisSaloffCoste1995}, who give a bound of the correct asymptotic
order ($n\log(n)$ for a deck of $n$ cards) as the number of cards grows to
infinity. Subsequently, Uyemura-Reyes \cite[Section~5.1]{Reyes2002} established an upper
bound for the mixing time, which was later sharpened in work of Saloff-Coste
and Z{\'u}{\~n}iga \cite{SaloffCosteZuniga2008} and of Morris and
Qin \cite{MorrisQin2014}. Subag \cite{Subag2013} recently established a lower
bound. This bound was shown to be tight by Bernstein and
Nestoridi~\cite{BernsteinNestoridi2017} using a proof technique that
establishes mixing times from explicit combinatorial descriptions of
eigenvalues; see Diaconis and Shahshahani~\cite{DiaconisShahshahani1981}.

\item
Our results may shed new light on various generalizations of the random-to-random shuffle
that have been of recent interest.
Ayyer, Schilling and Thiéry \cite{AyyerSchillingThiery} introduced
Markov chains on the set of linear extensions of a finite poset; when the poset
is an antichain, they recover the random-to-random shuffle on permutations.
They conjecture that the second largest eigenvalue is bounded above by
$(1+1/n)(1-2/n)$ for posets with $n$ elements, with equality when the poset is
disconnected. For the case of the random-to-random shuffle,
it follows immediately from our results that this bound is achieved.
Representation theory also plays a prominental role in their analysis.

There has also been quite a bit of interest in systematic scan versions of the
random-to-random shuffle where instead of randomly selecting a card to remove,
the cards are removed systematically.
One such shuffle is the \emph{card-cyclic-to-random} shuffle:
remove the card labelled $1$ and randomly reinsert it back into the deck; then
remove the card labelled $2$ and randomly reinsert it; then
remove the card labelled $3$ and randomly reinsert it; and so on,
cycling through all the cards in the deck.
This shuffle was introduced by Pinsky \cite{Pinsky2015} and further studied by
Morris--Ning--Peres \cite{MorrisNingPeres2014}, who showed that
the mixing time is on the order of $n\log(n)$ for a deck of $n$ cards.
\end{enumerate}

\medskip
\emph{Outline of the article.}
Our results rest on two principal theorems.
The proofs of these two theorems are quite lengthy and make substantial use of
the representation theory of the symmetric group, so we delay them to
\autoref{sec:main-proofs}.
This allows us to present a detailed exposition in
\autoref{the-shuffling-processes} through \autoref{sec:eigenspaces}
that is unimpeded by technicalities and that does not require any knowledge of
representation theory.

\autoref{the-shuffling-processes} introduces notation and interprets shuffles
as operators acting on words. We use a word to model a deck of cards since this
allows for decks with repeated cards.

The eigenvalues are described in \autoref{sec:eigenvalues}.
We present two descriptions: one in terms of a combinatorial statistic on words
(\autoref{ssec:R2R-eigenvalues-RSK}); and a reformulation in terms of
horizontal strips (\autoref{ssec:R2R-eigenvalues-horizontal-strips}).
This makes use of the celebrated Robinson--Schensted--Knuth correspondence
between words and pairs of tableaux (defined in
\autoref{ssec:R2R-eigenvalues-RSK})
as well as the providential notion of desarrangement tableaux
(\autoref{sssec:desarrangementtableaux}).

The eigenspaces are described in \autoref{sec:eigenspaces}.
The central result is an inductive procedure to construct eigenvectors for the
random-to-random shuffle acting on words of length $n+1$ from eigenvectors for
the random-to-random shuffle acting on words of length $n$.
We present two concrete examples in \autoref{ssec:motivating-examples} to
illustrate the process; these are special cases of results in
\autoref{ssec:general-case}.
It turns out that all eigenvectors can be obtained in this way, except those
that lie in the kernel.
\autoref{constructing-eigenspaces-and-eigenbases} describes how this
leads an explicit recursive construction of eigenbases for random-to-random
shuffles starting from bases of their kernels.
In \autoref{ssec:frobenius-characteristics}, we briefly restrict our attention
to decks with no repeated cards (\ie, permutations). We compute the Frobenius
characteristic of the eigenspaces of the random-to-random operator viewed as
(left) modules of the symmetric groups.

\autoref{sec:main-proofs} is dedicated to proving the two principal theorems,
\autoref{thm:lifting-eigenvectors} and \autoref{thm:eigenspace-decomposition}.
A detailed outline of the section appears in \autoref{outline-of-the-proofs},
which also serves as an outline of the proofs. We refer the reader to
\autoref{outline-of-the-proofs} for more details.

\subsubsection*{Acknowledgements}

We wish to thank the following people for numerous enlightening conversations
concerning various aspects of this research project:
Nantel Bergeron,
Persi Diaconis,
Darij Grinberg,
Mathieu Guay-Paquet,
Nadia Lafrenière,
Vic Reiner,
Stéphanie Schanck,
Anne Schilling,
Nicolas M. Thiéry,
Volkmar Welker,
Mike Zabrocki.
We also thank Darij, Nadia, Stéphanie
and the anonymous referee
for their careful reading of the text and their numerous suggestions that
improved the exposition.

We thank Patricia Hersh for her encouragement and advice, and for pointing
out the application to the integrality of the Laplacians of the complex of
injective words.

We thank Alain Lascoux for explaining how to efficiently construct irreducible
matrix representations of the symmetric group. The resulting computations were
very helpful throughout the course of this project.

This research was facilitated by computer exploration using the open-source
mathematical software \texttt{Sage}~\cite{Sage} and its algebraic combinatorics
features developed by the \texttt{Sage-Combinat} community
\cite{Sage-Combinat}.

\section{The shuffling processes}
\label{the-shuffling-processes}

In this section, we introduce the card shuffling processes we will study,
describe their action on words,
and define their transition matrices.

\subsection{The shuffles}

The random-to-top shuffle, also known as the \emph{Tsetlin library} or the
\emph{move-to-front/end rule}, is the following method of shuffling a deck of
cards.
\begin{quote}
    \emph{Random-to-Top shuffle}:
        choose a card at random from a deck of cards and move it to the top
        of the deck.
\end{quote}
The top-to-random shuffle is the inverse process.
\begin{quote}
    \emph{Top-to-Random shuffle}:
        remove the top card from a deck of cards and insert it into the
        deck at a random position.
\end{quote}
The \emph{random-to-random shuffle} (which is sometimes called
\emph{random insertion} in the literature) is the ``symmetrization'' of the
random-to-top shuffle; that is, it is the composition of the random-to-top
shuffle with the top-to-random shuffle.
\begin{quote}
    \emph{Random-to-Random shuffle}: remove a card at random from
    a deck of cards and insert it into the deck at a random position.
\end{quote}

\subsection{Decks of cards as words}

These shuffling processes define random walks on all the possible ways of
arranging the cards in a deck.
If the deck of $n$ cards contains distinct cards, then we model these
arrangements as permutations of the set $\{1, 2, \dots, n\}$.
Otherwise, we model the deck of cards as a word (defined below) of length $n$.
The former is a special case of the latter, but it is often
instructive and advantageous to consider the two cases separately.

\smallskip

Let $A = \{a_1, a_2, a_3, \dots\}$ be an ordered finite set.
We say that $A$ is an \emph{alphabet} and that its elements are \emph{letters}.
The elements of $A$ play the role of the cards in the deck.
We implicitly assume throughout that $A$ has enough elements;
\ie, $A$ has at least as many letters as the deck has distinct cards.

A \emph{word} on $A$ is a finite sequence of elements of $A$.
A word then corresponds to an arrangement of the deck of cards.
The \emph{length} of a word $w$ is the number of letters that it contains
and is denoted by $\ell(w)$.
A word $w$ is said to have
\emph{evaluation}
$\nu = (\nu_1, \nu_2, \dots, \nu_r)$
if $w$ contains $\nu_1$ occurrences of the first letter of the alphabet,
$\nu_2$ occurrences of the second letter, and so on.
(Some authors call this the ``content'' of the word.)

We view permutations as words and denote them using
\emph{one-line notation} or \emph{word notation}:
if $\sigma$ is a permutation, then we write $\sigma = \sigma(1) \sigma(2)
\cdots \sigma(n)$. For instance, $\sigma = 213$ is the permutation that
interchanges $1$ and $2$, and fixes $3$.
Hence, permutations of size $n$ coincide with the words of evaluation
$(1, 1, \dots, 1)$ (with $n$ copies of $1$)
on the alphabet $\{1, 2, \dots, n\}$.

\subsection{Transition matrices}

The long-term behaviour of these random walks is governed by the properties
(such as eigenvalues and eigenvectors) of a certain matrix associated to the
walk called its transition matrix.

\begin{Definition}
    The \emph{transition matrix} of a shuffling process acting on words of
    evaluation $\nu$ is the matrix whose rows and columns are indexed by
    these words and whose $(w, u)$-entry is the probability of
    obtaining $u$ from $w$ with exactly one shuffle.

    The \emph{eigenvalues and eigenvectors of a shuffling process} are the
    eigenvalues and eigenvectors of the associated transition matrix.
\end{Definition}

We will see that the transition matrix of the random-to-random shuffle is $T
T^t$, where $T$ is the transition matrix of the random-to-top shuffle and $T^t$
is its transpose (note that $T^t$ is the transition matrix of the top-to-random
shuffle).

\begin{Example}
    \label{ex:transition-matrices}

    Our convention is to identify the ``top'' of the deck with the
    \emph{last position} of a word.
    With this convention, the random-to-top shuffle acting on a word $w$ moves
    a letter to the end of the word.
    For instance, starting with $2112$ one could obtain
    $1122$, $2121$ (in two ways), and $2112$.
    Hence, the transition matrix of the random-to-top shuffle
    acting on words of length $4$ and evaluation $(2,2)$ is
    \begin{gather*}
        \frac{1}{4} \times
        \bordermatrix{
                                & \text{\tiny $1122$} & \text{\tiny $1212$} & \text{\tiny $2112$} & \text{\tiny $1221$} & \text{\tiny $2121$} & \text{\tiny $2211$} \cr
            \text{\tiny $1122$} & 2                   & 0                   & 0                   & 2                   & 0                   & 0 \cr
            \text{\tiny $1212$} & 1                   & 1                   & 0                   & 1                   & 1                   & 0 \cr
            \text{\tiny $2112$} & 1                   & 0                   & 1                   & 0                   & 2                   & 0 \cr
            \text{\tiny $1221$} & 0                   & 2                   & 0                   & 1                   & 0                   & 1 \cr
            \text{\tiny $2121$} & 0                   & 1                   & 1                   & 0                   & 1                   & 1 \cr
            \text{\tiny $2211$} & 0                   & 0                   & 2                   & 0                   & 0                   & 2 \cr
        }
    \end{gather*}
    and the transition matrix of the random-to-random shuffle acting words of
    length $4$ and evaluation $(2,2)$ is
    \begin{gather*}
        \frac{1}{16} \times
        \bordermatrix{
                                & \text{\tiny $1122$} & \text{\tiny $1212$} & \text{\tiny $2112$} & \text{\tiny $1221$} & \text{\tiny $2121$} & \text{\tiny $2211$} \cr
            \text{\tiny $1122$} & 8                   & 4                   & 2                   & 2                   & 0                   & 0 \cr
            \text{\tiny $1212$} & 4                   & 4                   & 3                   & 3                   & 2                   & 0 \cr
            \text{\tiny $2112$} & 2                   & 3                   & 6                   & 0                   & 3                   & 2 \cr
            \text{\tiny $1221$} & 2                   & 3                   & 0                   & 6                   & 3                   & 2 \cr
            \text{\tiny $2121$} & 0                   & 2                   & 3                   & 3                   & 4                   & 4 \cr
            \text{\tiny $2211$} & 0                   & 0                   & 2                   & 2                   & 4                   & 8 \cr
        }.
    \end{gather*}
    (The words labelling the rows and columns of the matrix are ordered from
    right-to-left/bottom-to-top using reverse lexicographic ordering.)
\end{Example}

\section{Eigenvalues}
\label{sec:eigenvalues}

This section describes the eigenvalues of the random-to-random shuffle acting
on words. The eigenvalues can be indexed by certain combinatorial objects
called horizontal strips. We introduce these in
\autoref{ssec:R2R-eigenvalues-horizontal-strips} and describe a combinatorial
statistic that associates with each horizontal strip an eigenvalue of the
random-to-random shuffle.

In \autoref{some-applications}, we present two applications of these results.
We compute the second largest eigenvalue in
\autoref{sssec:second-largest-eigenvalue}.
In \autoref{sssec:laplacians} we establish the integrality of the spectra of
the Laplacians on the complex of injective words.

Since the random-to-random operator is acting on a vector space of dimension
the number of words, one can also index the eigenvalues using words.
\autoref{ssec:R2R-eigenvalues-RSK} reformulates the eigenvalue result to
associate with each word $w$ an eigenvalue of the random-to-random shuffle.

\subsection{Eigenvalues and horizontal strips}
\label{ssec:R2R-eigenvalues-horizontal-strips}

Our first description of the eigenvalues makes use of the
notion of the diagonal index of the cells of a (skew) partition.
We begin by reviewing the necessary definitions.

\subsubsection{Partitions and diagonal index}

An \emph{(integer) partition} is a sequence of positive integers
$\lambda = (\lambda_1, \lambda_2, \dots, \lambda_l)$ such that
$\lambda_1 \geq \lambda_2 \geq \cdots \geq \lambda_l > 0$.
Let $|\lambda| = \lambda_1 + \lambda_2 + \cdots + \lambda_l$.
If $|\lambda| = n$, then we say that $\lambda$ is a \emph{partition of $n$} and
denote this by $\lambda \vdash n$.

The \emph{diagram} of $\lambda$ is an array of left-justified boxes
with $\lambda_i$ boxes in row $i$.
(See \autoref{subfig:partition} for an example.)
The boxes are called the \emph{cells} of $\lambda$.

We use matrix coordinates to index the cells of $\lambda$: so $(i, j)$ refers
to the cell in row $i$ and column $j$ of the diagram of $\lambda$.
The \emph{diagonal index} of a cell $(i, j)$ is $j - i$;
see \autoref{subfig:diagonalindex} for examples.
(Some authors call this the ``content'' of the cell.)

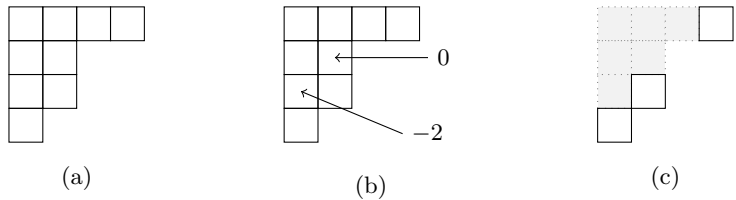
\begin{figure}[t!]
    \begin{subfigure}[t]{0.30\linewidth}
        \begin{displaymath}
            \begin{tikzpicture}[scale=0.45,every node/.style={font=\rm\small}]
                \tikztableauinternal{{\null,\null,\null,\null},{\null,\null},{\null,\null},{\null}}
            \end{tikzpicture}
        \end{displaymath}
        \subcaption{}
        \label{subfig:partition}
    \end{subfigure}
    \begin{subfigure}[t]{0.30\linewidth}
        \begin{displaymath}
            \begin{tikzpicture}[scale=0.45,every node/.style={font=\rm\small}]
                \tikztableauinternal{{\null,\null,\null,\null},{\null,\null},{\null,\null},{\null}}%
                    \draw[<-] (1.5, -1.5) -- (4.25, -1.5);%
                    \node at (4.25, -1.5) [anchor=west] {$0$};%
                    \draw[<-] (0.5, -2.5) -- (3.5, -3.75);%
                    \node at (3.5, -3.75) [anchor=west] {$-2$};%
            \end{tikzpicture}
        \end{displaymath}
        \subcaption{}
        \label{subfig:diagonalindex}
    \end{subfigure}
    \begin{subfigure}[t]{0.30\linewidth}
        \begin{displaymath}
            \begin{tikzpicture}[scale=0.45,every node/.style={font=\rm\small}]
                \tikztableauinternal{{X,X,X,\null},{X,X},{X,\null},{\null}}
            \end{tikzpicture}
        \end{displaymath}
        \subcaption{}
        \label{subfig:skewpartition}
    \end{subfigure}
    \caption{
        \ref{sub@subfig:partition}
            Diagram of the partition $(4,2,2,1)$.
        \ref{sub@subfig:diagonalindex}
            Diagonal index of the cells $(2,2)$ and $(3,1)$.
        \ref{sub@subfig:skewpartition}
            The horizontal strip $(4,2,2,1)/(3,2,1)$ consists of the 3 depicted
            white cells.
    }
    \label{fig:partitions}
\end{figure}

\subsubsection{Skew partitions and horizontal strips}
If $\lambda$ and $\mu$ are two partitions such that the diagram of $\lambda$
contains the diagram of $\mu$, then the \emph{skew partition} $\lambda/\mu$
consists of the cells of $\lambda$ that do not belong to $\mu$
(see \autoref{subfig:skewpartition}).
A \emph{horizontal strip} is a skew partition containing at most one cell in
each column.

We denote by $\diagonalindex(\lambda/\mu)$ the sum of the diagonal indices of
all the cells of $\lambda/\mu$:
\begin{equation*}
    \diagonalindex(\lambda/\mu) = \sum_{\text{cells $(i,j)$ of $\lambda/\mu$}} \big(j - i\big).
\end{equation*}
As an important special case occurs when $\mu$ is the empty partition.
In this case, the skew partition $\lambda/\emptyset$ is equal to $\lambda$,
and we denote $\diagonalindex(\lambda/\emptyset)$ by
$\diagonalindex(\lambda)$.

\subsubsection{Tableaux, ascents and desarrangement tableaux}
\label{sssec:desarrangementtableaux}

We refer the reader to \autoref{ex:desarrangementtableaux}
for examples of the notions introduced below.

Let $\lambda$ be a partition of $n$.
A \emph{tableau} of shape $\lambda$ is a filling of the cells of the diagram of
$\lambda$ with positive integers.
A tableau with $n$ cells is said to be \emph{standard} if its entries are
$\{1, 2, \dots, n\}$ and they are arranged in such as way that they strictly
increase along the rows (left to right) and down the columns (top to bottom).

An entry $i$ in a standard tableau $t$ of size $n$ is an \emph{ascent} of $t$
if either $i = n$ or if $i < n$ and $i+1$ appears weakly to the north and east
of $i$.
A standard tableau $t$ is said to be a \emph{desarrangement tableau}
if its smallest ascent is even.
Equivalently, $t$ is a desarrangement tableau
if there exists $r \in \NN$ such that
$1, 2, \dots, 2r$ are in the first column of $t$
and $2r+1$ is not in the first column of $t$.
Equivalently, $t$ is a desarrangement tableau if
the $(1,2)$-entry of $t$ is odd, or
there is no $(1,2)$-entry and $n$ is even.

\begin{Remark}
    Under the RSK correspondence between words and pairs of tableaux (recalled
    in \autoref{ssec:R2R-eigenvalues-RSK}), the ascents of a word coincide with
    the ascents of its recording tableau. Thus, desarrangement tableaux are
    precisely the recording tableaux of words whose first ascent is even.
    Such a word is called a \emph{desarrangement word}. They were first studied
    by Désarménien in \cite{Desarmenien1982}, who showed that desarrangement
    permutations are in bijection with derangements. Their properties were
    further developed by Désarménien and Wachs in \cite{DesarmenienWachs1988,
    DesarmenienWachs1993}. Desarrangement tableaux also appear in recent work
    of Reiner, the second author, and Welker \cite{RSW2014},
    and Hersh and Reiner \cite{HershReiner2015}.
\end{Remark}

\begin{Example}[Desarrangement tableaux]
    \label{ex:desarrangementtableaux}
    The unique tableau of size $0$ is a desarrangement tableau.
    There is one standard tableau of size $1$; it is not a desarrangement
    tableau.
    There are two standard tableaux of size $2$
    and four of size $3$.
    \begin{equation*}
        \tikztableausmall{{1,2}}
        \qquad
        \tikztableausmall{{1},{2}}
        \qquad
        \tikztableausmall{{1,2,3}}
        \qquad
        \tikztableausmall{{1,3},{2}}
        \qquad
        \tikztableausmall{{1,2},{3}}
        \qquad
        \tikztableausmall{{1},{2},{3}}
    \end{equation*}
    Only the second and fourth tableaux above are desarrangement tableaux.

    There are exactly $4$ desarrangement tableaux of size $4$:
    \begin{displaymath}
        \tikztableausmall{{1, 3, 4}, {2}}
        \qquad
        \tikztableausmall{{1, 3}, {2, 4}}
        \qquad
        \tikztableausmall{{1, 3}, {2}, {4}}
        \qquad
        \tikztableausmall{{1}, {2}, {3}, {4}}.
        \qedhere
    \end{displaymath}
\end{Example}

\subsubsection{Eigenvalues and horizontal strips}

The eigenvalues of the random-to-random shuffle acting on words of evaluation
$\nu \vdash n$ are indexed by horizontal strips $\lambda/\mu$ with $\lambda
\dominateseq \nu$. Here, $\dominateseq$ denotes the \emph{dominance order} on
partitions of $n$, defined by
\begin{center}
    $\lambda \dominateseq \nu$
    if and only if
    $\lambda_1 + \cdots + \lambda_i \geq \nu_1 + \cdots + \nu_i$ for all $i \geq 1$;
\end{center}
where it is understood that if a partition $\mu$ does not have a $j$-th part,
then $\mu_j = 0$.

\begin{Theorem}
    \label{thm:eigenvalues-of-R2R}
    Every eigenvalue of the random-to-random shuffle acting on words of
    evaluation $\nu \vdash n$ is of the form $\frac{1}{n^2} \eig(\lambda/\mu)$,
    where
    \begin{itemize}
        \item
            $\lambda$ is a partition of $n$ and
            $\lambda/\mu$ is a horizontal strip with $\lambda
            \dominateseq \nu$ such that $\mu \neq (n)$ and
            $\mu \neq (1, 1, \dots, 1)$ with an odd number of $1$s;

        \item
            and $\eig$ is the following combinatorial statistic defined on skew
            partitions:
            \begin{equation}
                \label{eq:eigenvalue-statistic-skew-partitions}
                \begin{aligned}
                    \eig\left(\lambda/\mu\right)
                    & =
                        \left[\binom{|\lambda|+1}{2} + \diagonalindex(\lambda)\right]
                        - \left[\binom{|\mu|+1}{2} + \diagonalindex(\mu)\right]
                        \\
                    & =
                    \binom{|\lambda|+1}{2} - \binom{|\mu|+1}{2}
                        + \diagonalindex\left(\lambda/\mu\right).
                \end{aligned}
            \end{equation}
    \end{itemize}
    The multiplicity of an eigenvalue $\frac{1}{n^2} \varepsilon$ is
    \begin{equation*}
        \sum_{
            \substack{
                \text{$\lambda/\mu$ is a horizontal strip,} \\
                \lambda \dominateseq \nu \text{~and~} \eig(\lambda/\mu) = \varepsilon}
        }
            K_{\lambda, \nu} \, \ndestab^\mu,
    \end{equation*}
    where $K_{\lambda, \nu}$ the number of semistandard tableaux of shape
    $\lambda$ and evaluation $\nu$,
    and
    $\ndestab^\mu$ is the number of desarrangement tableaux of shape $\mu$
    (\cf \autoref{sssec:desarrangementtableaux}).
\end{Theorem}

Examples illustrating \autoref{thm:eigenvalues-of-R2R} appear in
\autoref{fig:eigenvalues-of-R2R3} ($n = 3$)
and \autoref{appendix:figures} ($n \leq 5$).

\begin{proof}[Proof (Outline)]
    The result will follow from a decomposition of the vector space into
    subspaces on which the random-to-random operator acts as scalar
    multiplication. We provide an outline of the argument here.

    Properties of the representation theory of the symmetric group
    allow us to decompose the vector space $M^\nu$ spanned by words of evaluation $\nu$
    into subspaces $S^\lambda$ that are indexed by partitions $\lambda \vdash
    n$ with $\lambda \dominateseq \nu$.
    The number of copies of the space $S^\lambda$ in $M^\nu$ is $K_{\lambda, \nu}$.
    (This is explained in detail in
    \autoref{modules-of-words-and-reduction-to-specht-submodules} and
    \autoref{Mdecomposition} of \autoref{sss:Mdecomposition}.)

    \autoref{thm:eigenspace-decomposition} further decomposes these subspaces into
    subspaces that are indexed by horizontal strips $\lambda/\mu$, that are of
    dimension $\ndestab^\mu$, and on which the (normalized) random-to-random
    operator acts as scalar multiplication by $\eig(\lambda/\mu)$.
\end{proof}

\begin{figure}[hb]
    \begin{tabular}{c|ccc|cccc}
        \toprule
            $\lambda/\mu$ &
            $\ndestab^\mu$ &
            $K_{\lambda, 111}$ &
            multiplicity &
            $\binom{|\lambda|+1}{2}$ &
            $\binom{|\mu|+1}{2}$ &
            $\diagonalindex(\lambda/\mu)$ &
            $\eig(\lambda/\mu)$
        \\ \midrule
        $\tikztableauscript{{\null,\null,\null},}$ &   1 &   1 &   1 &   6 &   0 &   3 &   9 \\ \midrule
        $      \tikztableauscript{{X,\null},{X},}$ &   1 &   2 &   2 &   6 &   3 &   1 &   4 \\
        $          \tikztableauscript{{X,X},{X},}$ &   1 &   2 &   2 &   6 &   6 &   0 &   0 \\ \midrule
        $    \tikztableauscript{{X},{X},{\null},}$ &   1 &   1 &   1 &   6 &   3 &  -2 &   1 \\ \bottomrule
    \end{tabular}
    \caption{Eigenvalues (scaled by $n^2$) for random-to-random shuffles
        acting on permutations of size $n = 3$.
        (See \autoref{appendix:figures} for $2 \leq n \leq 6$.)}
    \label{fig:eigenvalues-of-R2R3}
\end{figure}

\subsection{Applications}
\label{some-applications}

Before continuing, we present two immediate implications of our results.
This section is not necessary for what follows.

\subsubsection{Second largest eigenvalue}
\label{sssec:second-largest-eigenvalue}

\autoref{thm:eigenvalues-of-R2R} allows us to compute the second largest
eigenvalue for the random-to-random shuffle.
It corresponds to the skew partition $\lambda/\mu$
with $\lambda = (n - 1, 1)$ and $\mu = (1, 1)$;
in this case
\begin{equation*}
    \begin{multlined}
        \eig\left(
                \tikztableausmall{{X, \null, \null, ${...}$, \null}, {X},}
            \right)
        = \binom{n+1}{2} - \binom{3}{2} + \Big(1 + \cdots + (n-2)\Big)
        \\
        = \frac{(n+1)n}{2} - 3 + \frac{(n-2)(n-1)}{2}
        = n^2 - n - 2 = (n - 2) (n + 1).
    \end{multlined}
\end{equation*}
It is of multiplicity
$K_{(n-1,1), \nu} \, \ndestab^{(1,1)} = \ell(\nu) - 1$,
since there is only one desarrangement tableau of shape $\mu = (1, 1)$ and
$\ell(\nu) - 1$ semistandard tableaux of shape $(n-1, 1)$ and evaluation $\nu$.
This establishes the following.

\begin{Corollary}[Second largest eigenvalue]
    \label{second-largest-eigenvalue}
    Let $\nu$ be a partition of $n$ that is not equal to $(n)$.
    Then the second largest eigenvalue of the random-to-random shuffle acting
    on words of evaluation $\nu$ is
    \begin{equation*}
        \frac{1}{n^2}\Big((n - 2) (n + 1)\Big) =
        \left(1 - \frac{2}{n}\right) \left(1 + \frac{1}{n}\right)
    \end{equation*}
    with multiplicity $\ell(\nu) - 1$.
\end{Corollary}
We refer to \cite{CaputoLiggettRichthammer2010,Dieker2010} for a related result
for a different shuffling process.

\subsubsection{Integrality of the spectra of the Laplacians for the complex of
    injective words}
\label{sssec:laplacians}

Hanlon and Hersh \cite{HanlonHersh2004} uncovered a close connection between
random-to-random Markov chains and the Laplacians of the complex of injective
words. Our results immediately imply that the spectra of these Laplacians are
also integral. To describe this connection, we first recall some definitions
and results from \cite{HanlonHersh2004,ReinerWebb2004}.

Fix $n$ and let $\IWPoset_n$ denote the set of words without repetition on the
alphabet $\{1, 2, \dots, n\}$. Words without repetition are called
\emph{injective words}. The set $\IWPoset_n$ is partially ordered using the
relation $w \leq u$ in $\IWPoset_n$ if and only if $w$ is a (not-necessarily-contiguous) subword of $u$.
With respect to this partial order, $\IWPoset_n$ is the poset of faces of
a $(n-1)$--dimensional regular CW-complex $\IWComplex_n$, called the
\emph{complex of injective words}.
This complex has been studied by Farmer \cite{Farmer1978} and Bjorner and Wachs
\cite{BjornerWachs1983}, whose results imply that its homology vanishes except
in the top degree, and by Reiner and Webb \cite{ReinerWebb2004}, who studied
the homology as modules for the symmetric groups.

Let $\IWModule_{r}$ denote the vector space spanned by the set of injective
words of length $r$ on the alphabet $\{1, 2, \dots, n\}$.
Let $\IWpartial_{r} : \IWModule_{r} \to \IWModule_{r-1}$ denote the
linear transformation defined on injective words of length $r$ by
\begin{gather*}
    \IWpartial_{r}(w_1w_2 \cdots w_r)
    = \sum_{j = 1}^r (-1)^j w_1 w_2 \cdots w_{j-1} w_{j+1} \cdots w_r;
\end{gather*}
and let $\IWdelta_{r} : \IWModule_{r} \to \IWModule_{r+1}$ denote the
adjoint of $\IWpartial_{r+1}$ (with respect to the inner product on
$\IWModule_{r}$ for which the injective words form an orthonormal basis).
The \emph{Laplacian} is the linear transformation
$\IWLaplacian_{r} : \IWModule_{r} \to \IWModule_{r}$
defined as
\begin{gather*}
    \IWLaplacian_{r}
    = \IWdelta_{r-1} \circ \IWpartial_{r} + \IWpartial_{r+1} \circ \IWdelta_{r}.
\end{gather*}

For $r = n$, Hanlon and Hersh proved that the Laplacian $\IWLaplacian_{n}$
coincides with the ``signed'' random-to-random operator $\Rand_n^\pm$, which is
defined for $1 \leq r \leq n$ as
\begin{align*}
    \Rand_r^\pm (w_1 \cdots w_r)
    &= r \cdot w_1 \cdots w_r \\
    &
      + \sum_{1\leq u < v \leq r} (-1)^{v - u}
      w_1 \cdots w_{u-1} \mathred{w_{u+1} \cdots w_{v}} \mathgreen{w_{u}} w_{v+1} \cdots w_r
    \\
    &
      + \sum_{1 \leq v < u \leq r} (-1)^{u - v}
      w_1 \cdots w_{v-1} \mathgreen{w_{u}} \mathred{w_{v} \cdots w_{u-1}} w_{u+1} \cdots w_r
\end{align*}
(see \cite[Theorem~3.1]{HanlonHersh2004}).
The random-to-random operator $\Rand_r$ is defined similarly,
ignoring the signs in the definition of $\Rand_r^\pm$:
\begin{align*}
    \Rand_r (w_1 \cdots w_r)
    &= r \cdot w_1 \cdots w_r \\
    &
      + \sum_{1\leq u < v \leq r}
      w_1 \cdots w_{u-1} \mathred{w_{u+1} \cdots w_{v}} \mathgreen{w_{u}} w_{v+1} \cdots w_r
    \\
    &
      + \sum_{1 \leq v < u \leq r}
      w_1 \cdots w_{v-1} \mathgreen{w_{u}} \mathred{w_{v} \cdots w_{u-1}} w_{u+1} \cdots w_r.
\end{align*}

It turns out that $\Rand_r^\pm$ and $\Rand_r$ are conjugate linear
transformations of $\IWModule_{r}$.
To see this, let $w = w_1 \cdots w_r$ denote an injective word of length $r$.
Since its letters are distinct, there is a unique permutation
$\sigma \in \symm_r$ satisfying
    $w_{\sigma_1} < w_{\sigma_2} < \cdots < w_{\sigma_r}$.
Let $\sign(w)$ denote the sign of this permutation,
and let $T: \IWModule_{r} \to \IWModule_{r}$ denote the linear transformation
obtained by extending $w \mapsto \sign(w) w$ linearly to all of
$\IWModule_{r}$.
Since for any permutation $\tau \in \symm_r$ we have
\begin{equation*}
    \sign(w_{\tau_1} w_{\tau_2} \cdots w_{\tau_r})
    =
    \sign(\tau) \sign(w_1 w_2 \cdots w_r),
\end{equation*}
it follows that $\Rand_r^\pm = T \circ \Rand_r \circ T^{-1}$
because
\begin{equation*}
    \begin{aligned}
        & \sign(w) \sign(w_1 \cdots w_{u-1} \mathred{w_{u+1} \cdots w_{v}} \mathgreen{w_{u}} w_{v+1} \cdots w_r)
        \\
        &\qquad =
        \sign((u, u+1, \ldots, v))
        =
        (-1)^{v - u}
    \end{aligned}
\end{equation*}
and similarly for the other summands.

Hanlon and Hersh prove that the integrality of the
spectrum of $\Rand_r^\pm$ implies the integrality of the spectrum of
$\IWLaplacian_r$ \cite[Theorem~3.3(2)]{HanlonHersh2004}.
Since the operators $\Rand_r^\pm$ and $\Rand_r$ are conjugate,
they have the same spectrum and so
the integrality of the spectrum of $\IWLaplacian_r$ follows from
\autoref{thm:eigenvalues-of-R2R}.

\begin{Corollary}
    \label{cor:laplacians}
    Fix $n \in \NN$ and let $\IWLaplacian_{r} : \IWModule_{r} \to
    \IWModule_{r}$, for $0 \leq r \leq n$, be the Laplacians associated with
    the complex of injective words on the alphabet $\{1, 2, \dots, n\}$.
    Then the spectrum of $\IWLaplacian_r$ is integral.
\end{Corollary}

\subsection{Eigenvalues and the RSK correspondence}
\label{ssec:R2R-eigenvalues-RSK}

Since the random-to-random operator is acting on a vector space whose dimension
is the number of words, it is natural to index the eigenvalues using these
words. Here, we reformulate \autoref{thm:eigenvalues-of-R2R} to associate with
each word an eigenvalue of the random-to-random shuffle.
This makes use of the celebrated Robinson--Schensted--Knuth correspondence
between words and pairs of tableaux.

\subsubsection{The RSK correspondence}

The Robinson–Schensted–Knuth (RSK) correspondence associates two tableaux,
denoted $P(w)$ and $Q(w)$ of the same shape, with a word $w$.
We review the correspondence here and refer the reader to
\cite{Sagan2001, StanleyEC2, Knuth1970} for more details.
The main step in this correspondence is the following.

\begin{Definition}[Row insertion]
    Let $t$ be a tableau and let $i \in \NN$.
    The \emph{row insertion} of $i$ into the $r$-th row of $t$ is the tableau
    defined recursively as follows.
    \begin{enumerate}
        \item
            If $i$ is greater than or equal to all entries of the $r$-th
            row of $t$ (including the case where the $r$-th row is empty),
            then append $i$ to the end of the row.
        \item
            Otherwise, let $j$ be the first occurrence of the smallest entry of
            the row that is greater than $i$. Replace $j$ with $i$ and insert
            the $j$ into the $(r+1)$-th row.
    \end{enumerate}
    If the row is not specified, then it is understood that $r=1$.
\end{Definition}

Iterating the row insertion procedure allows us to associate a tableau
$P(w)$ to any word $w$: start with the empty tableau (\ie, the
tableau with no rows and no columns); insert $w_1$ into row $1$; then insert
$w_2$ into row $1$; then insert $w_3$ into row $1$; and so on.
The tableau $P(w)$ is called the \emph{insertion tableau} of $w$.
Note that the entries of $P(w)$
weakly increase along the rows
and strictly increase down the columns.
Such tableaux are said to be \emph{semistandard}.

\begin{Example}
    \label{ex:insertion}
    Let $w = 234133134$.
    We insert $2$, then we insert $3$, then we insert
    $4$, and so on, obtaining the following sequence of tableaux.
    \begin{align*}
        \emptyset
        &
        \xrightarrow{\text{insert $w_1$}} \young(2)
        \xrightarrow{\text{insert $w_2$}} \young(23)
        \xrightarrow{\text{insert $w_3$}} \young(234)
        \xrightarrow{\text{insert $w_4$}} \young(134,2)
        \\[.5em]
        &
        \xrightarrow{\text{insert $w_5$}} \young(133,24)
        \xrightarrow{\text{insert $w_6$}} \young(1333,24)
        \xrightarrow{\text{insert $w_7$}} \young(1133,23,4)
        \\[.5em]
        &
        \xrightarrow{\text{insert $w_8$}} \young(11333,23,4)
        \xrightarrow{\text{insert $w_9$}} \young(113334,23,4)
        = P(w).
        \qedhere
    \end{align*}
\end{Example}

The above procedure constructs $P(w)$ one cell at a time.
By recording at which step the cells are created, we obtain
a second tableau $Q(w)$: a cell of $Q(w)$ contains $i$ if and only if
the corresponding cell of $P(w)$ is the $i$-th cell to be created.
The tableau $Q(w)$ is called the \emph{recording tableau} of $w$.
Note that the entries of $Q(w)$ are $\{1, 2, \dots, n\}$,
where $n$ is the length of $w$, and that they
strictly increase along the rows and down the columns of $Q(w)$.
Such tableaux are said to be \emph{standard}.

\begin{Example}
    \label{ex:recording}
    Continuing from \autoref{ex:insertion},
    \begin{equation*}
        \begin{aligned}
            P(w) = \young(113334,23,4)
            \qquad
            \text{and}
            \qquad
            Q(w) = \young(123689,45,7).
        \end{aligned}
        \qedhere
    \end{equation*}
\end{Example}

\subsubsection{Eigenvalues and the RSK correspondence}

Let $w = w_1 w_2 \cdots w_n$ be a word of length $n$ on an ordered alphabet.
A position $i \in \{1, 2, \dots, n\}$ is a \emph{(weak) ascent} of $w$ if
either $i = n$ or $i < n$ with $w_i \leq w_{i+1}$.
Recall that $w$ has evaluation
$\nu = (\nu_1, \nu_2, \dots, \nu_r)$
if $w$ contains $\nu_1$ occurrences of the first letter of the alphabet,
$\nu_2$ occurrences of the second letter, and so on.
If $Q$ is a tableau of shape $\lambda$, then we let
$\diagonalindex(Q) = \diagonalindex(\lambda)$,
which is the sum of the diagonal indices of the cells of $\lambda$.

The next result associates an eigenvalue with each
word $w$. Abusing notation, we denote this by $\eig(w)$.

\begin{Theorem}
    \label{thm:eigenvalues-of-R2R-RSK}
    The random-to-random shuffle acting on words of evaluation $\nu \vdash n$
    has an eigenvalue $\frac{1}{n^2} \eig(w)$
    for each word $w$ of evaluation $\nu$,
    where
    \begin{equation*}
        \label{eq:eigenvalue-statistic-permutations}
        \begin{multlined}
            \eig(w) =
            \left[ \binom{\ell(w)+1}{2} + \diagonalindex\Big( Q(w) \Big) \right]
            -
            \left[ \binom{\ell(w')+1}{2} + \diagonalindex\Big( Q(w') \Big) \right]
        \end{multlined}
    \end{equation*}
    and where $w'$ is the longest suffix of $w$ whose first ascent is even.

    The multiplicity of an eigenvalue $\varepsilon$ is the number of
    words $w$ of evaluation $\nu$ for which $\eig(w) = n^2 \, \varepsilon$.
\end{Theorem}

\begin{Example}
    \label{ex:eig-permutation}
    Let $w = 234133134$.
    Then $w' = 4133134$ is the longest suffix of $w$ whose first ascent
    is even (it occurs in position $2$ of $w'$).
    The recording tableau $Q(234133134)$ appears in \autoref{ex:recording}; and
    \begin{gather*}
        P(4133134) = \young(11334,3,4)
        \qquad
        Q(4133134) = \young(13467,2,5).
    \end{gather*}
    Since $\shape(Q(w)) = (6,2,1)$ and $\shape(Q(w')) = (5,1,1)$,
    we have
    \begin{align*}
        \eig(234133134)
        & =  \left[ \binom{10}{2} + \diagonalindex\big((6,2,1)\big) \right]
          - \left[ \binom{8}{2} + \diagonalindex\big((5,1,1)\big) \right]
        \\
        & = \left[45 + 12\right] - \left[28 + 7\right] = 22.
        \qedhere
    \end{align*}
\end{Example}

\begin{Example}
    \label{ex:R2R3-on-words}
    See \autoref{fig:R2R3-eigenvalues-RSK} for a computation of the eigenvalues for
    the random-to-random shuffle acting on words of length $3$ on the alphabet
    $\{1, 2, 3\}$. We list only the words whose evaluation is a partition.
    \begin{figure}[h]
        \begin{tabular}{ccccc}
            \toprule
            $\quad w \quad$ &
            $\quad w' \quad$ &
            $Q\left( w \right)$ &
            $Q\left( w' \right)$ &
            $\eig(w)$ \\
            \midrule
            $111$ & $\emptyset$ & $\tikztableauscript{{1,2,3},}$ & $\emptyset$ & $9$ \\
            \midrule
            $112$ & $\emptyset$ & $\tikztableauscript{{1,2,3},}$ & $\emptyset$ & $9$ \\
            $121$ & $21$ & $\tikztableauscript{{1,2},{3},}$ & $\tikztableauscript{{1},{2},}$ & $4$ \\
            $211$ & $211$ & $\tikztableauscript{{1,3},{2},}$ & $\tikztableauscript{{1,3},{2},}$ & $0$ \\
            \midrule
            $123$ & $\emptyset$ & $\tikztableauscript{{1,2,3},}$ & $\emptyset$ & $9$ \\
            $132$ & $32$ & $\tikztableauscript{{1,2},{3},}$ & $\tikztableauscript{{1},{2},}$ & $4$ \\
            $231$ & $31$ & $\tikztableauscript{{1,2},{3},}$ & $\tikztableauscript{{1},{2},}$ & $4$ \\
            $321$ & $21$ & $\tikztableauscript{{1},{2},{3},}$ & $\tikztableauscript{{1},{2},}$ & $1$ \\
            $213$ & $213$ & $\tikztableauscript{{1,3},{2},}$ & $\tikztableauscript{{1,3},{2},}$ & $0$ \\
            $312$ & $312$ & $\tikztableauscript{{1,3},{2},}$ & $\tikztableauscript{{1,3},{2},}$ & $0$ \\
            \bottomrule
        \end{tabular}
        \caption{Eigenvalues (scaled by $n^2$) for the random-to-random shuffle
            acting on words of length $n = 3$. (\cf \autoref{ex:R2R3-on-words}.)}
        \label{fig:R2R3-eigenvalues-RSK}
    \end{figure}
\end{Example}

\subsubsection{Equivalence of the eigenvalue descriptions}
\label{equivalence-of-the-eigenvalue-descriptions}

Next, we prove that the descriptions of the eigenvalues given in
\autoref{thm:eigenvalues-of-R2R} and \autoref{thm:eigenvalues-of-R2R-RSK}
are equivalent.

\begin{proof}[Proof that
    \autoref{thm:eigenvalues-of-R2R} and \autoref{thm:eigenvalues-of-R2R-RSK}
    are equivalent]
    The RSK correspondence is a bijection between words of evaluation
    $\nu$ and pairs of tableaux $(P, Q)$
    with $P$ a semi-standard tableau of evaluation $\nu$,
    and $Q$ a standard tableau such that $\shape(Q) = \shape(P)$ \cite{Knuth1970,Sagan2001}.
    Hence, if $w$ is a word of evaluation $\nu$, then
    \begin{equation*}
        \eig(w) = \eig\Big(\shape(P(w)) \big/ \shape(Q(\hat{w}))\Big),
    \end{equation*}
    where $\hat{w}$ is the longest suffix of $w$ whose first ascent is even,
    since $\shape(Q(w)) = \shape(P(w))$.

    Consequently, the result follows if, for each partition
    $\lambda$, there is a bijection between the set of standard tableaux of
    shape $\lambda$ and the set of desarrangement tableaux $\hat{Q}$ such that
    $\lambda/\shape(\hat{Q})$ is a horizontal strip. This is a result in
    \cite{RSW2014}; we briefly describe a bijection and refer to \cite{RSW2014}
    for details.

    Starting with $\hat{Q}$, perform ``reverse'' jeu-de-taquin slides \cite{Schutzenberger1977} on
    $\hat{Q}$ into the cells of $\lambda / \shape(\hat{Q})$ (starting with the
    leftmost empty cell and proceeding to the right).
    Then increment all the entries by $j = \size(\lambda / \shape(\hat{Q}))$.
    Properties of jeu-de-taquin \cite[Exercise 3.12.6]{Sagan2001} imply that
    we will create $j$ empty cells in the first row; filling these with
    $1, 2, \ldots, j$ results in a standard tableau of shape $\lambda$.
\end{proof}

\section{Eigenspaces}
\label{sec:eigenspaces}

This section presents a procedure to construct eigenvectors for the
random-to-random shuffle acting on words of length $n+1$ from eigenvectors for
the random-to-random shuffle acting on words of length $n$.
It turns out that all the eigenvectors can be obtained in this way,
except those that lie in the kernel.
This leads to an explicit construction of an eigenbasis starting from any basis
of the kernels of random-to-random (or equivalently, random-to-top) shuffles.

\subsection{Motivating examples}
\label{ssec:motivating-examples}

We begin with two special cases that serve to introduce the ideas underlying
our results.
In short, we will introduce a new letter into a word and consider how this
interacts with the random-to-random operator.
This turns out to admit a nice description when restricted to the Specht
submodules, which allows us to describe all the eigenvalues and eigenvectors of
the random-to-random shuffle. Indeed, the Specht submodules are subspaces that
are stable under the random-to-random shuffle and determining the eigenvalues
and eigenvectors reduces to determining them on each of the Specht submodules.
This reduction is sometimes called the \emph{Fourier transform} and will be
further discussed in \autoref{reduction-to-specht-modules}.

In \autoref{shuffles-as-linear-operators} we set up the environment and realize
the shuffling operators as linear transformations on the vector space of words.
In \autoref{sssec:specht-modules} we define the Specht submodules $S^\lambda$,
one for each partition $\lambda$.
In \autoref{sssec:level1lifting}, we define a linear transformation from
$S^\lambda$ to $S^{\lambda + \vec e_1}$, where $\lambda + \vec e_1$ is obtained
from $\lambda$ by adding $1$ to its first component, that maps eigenvectors of
the random-to-random shuffle acting on $S^\lambda$ to eigenvectors of the
random-to-random shuffle acting on $S^{\lambda+\vec e_1}$. We then use this to
construct an eigenbasis for $S^{(n-1, 1)}$.
In \autoref{sssec:level2lifting}, we modify the previous linear transformation to
construct eigenvectors in $S^{(2,2)}$ from eigenvectors in $S^{(2,1)}$.

\subsubsection{Shuffles as linear operators}
\label{shuffles-as-linear-operators}

We will interpret shuffles as linear transformations on the vector space
whose basis is the set of all words on a finite ordered alphabet $A$.
To illustrate the idea, consider the effect of performing a top-to-random
shuffle on a word $w_1 w_2 w_3 w_4$ (where $w_i \in A$); there are four
possible outcomes:
\begin{gather*}
    w_1w_2w_3\mathred{w_4}
    \qquad w_1w_2\mathred{w_4}w_3
    \qquad w_1\mathred{w_4}w_2w_3
    \qquad \mathred{w_4}w_1w_2w_3.
\end{gather*}
(Recall that our convention is that the last letter is the ``top'' card.)
The associated linear transformation is defined on words of length $4$ as
\begin{align}
    \label{ex:t2r-on-words}
    \TopR_4(w_1w_2w_3\mathred{w_4}) = w_1w_2w_3\mathred{w_4} + w_1w_2\mathred{w_4}w_3 + w_1\mathred{w_4}w_2w_3 + \mathred{w_4}w_1w_2w_3,
\end{align}
and is defined on arbitrary linear combinations of words of length $4$ by
linearity: $\TopR_4(\sum_w c_w w) = \sum_w c_w \TopR_4(w)$.
The matrix of this linear transformation is a rescaling (by a factor of $n$) of
the transition matrix of the top-to-random shuffle.
Similarly, one defines $\RTop_n$ and $\Rand_n$.
(For the explicit definitions, see \autoref{sec:operators-in-group-algebra}.)

For a word $w$ on the ordered alphabet $A = \{a_1, a_2, a_3, \dots\}$,
let $\Int_i(w)$ denote the sum of all words obtained from $w$ by inserting the
$i$-th letter of $A$; explicitly,
\begin{gather*}
    \Int_i(w_1\cdots w_n)
    = \sum_{0\leq j \leq n} w_1 \cdots w_{j} \cdot \mathred{a_i} \cdot w_{j+1} \cdots w_n.
\end{gather*}
By extending $\Int_i$ by linearity, we obtain a linear transformation on the
vector space of words, which we also denote by $\Int_i$.
It turns out that the composition $\Rand_{n+1} \circ \Int_i$ admits a nice
description when we restrict to certain subspaces, which we describe next.

\subsubsection{Specht modules}
\label{sssec:specht-modules}

Let $\lambda$ be a partition of $n$ and let $t$ be
a tableau of shape $\lambda$ whose entries are $1, 2, \ldots, n$ (each
appearing exactly once).
Let $\word(t)$ denote the word over the alphabet $A = \{a_1, a_2, a_3, \dots\}$
in which the $i$-th letter is $a_r$, where $r$ is the number of the row of $t$
that contains $i$.
Note that the evaluation of $\word(t)$ is $\lambda$ because it contains
$\lambda_1$ occurrences of $a_1$, $\lambda_2$ occurrences of $a_2$, and so on.
(We are implicitly assuming that $A$ has at least as many letters as $t$ has rows.)

Next, let $M^\lambda$ denote the vector space spanned by words of evaluation
$\lambda$ and
\begin{gather*}
    \w_t = \sum_{\sigma \in \ColStab(t)} \sign(\sigma) \word(\sigma(t)) \in M^\lambda,
\end{gather*}
where the sum ranges over all the permutations $\sigma$ of size $n$ that
permute the entries in the columns of $t$ in all possible ways.
The \emph{Specht module} $S^\lambda$ is the subspace
\begin{gather*}
    S^\lambda = \spn\left\{ \w_t : t \in \SYT_\lambda\right\} \subseteq M^\lambda,
\end{gather*}
where $\SYT_\lambda$ denotes the set of standard (Young) tableaux of shape
$\lambda$. (\autoref{remark:specht-modules-via-polytabloids} explains the
relationship with the construction of Specht modules via polytabloids).

\begin{Example}
    \label{ex:specht-modules}
    For our examples,
    we will work with the alphabet $A = \{a,b,c,\dots\}$.
    \begin{align*}
        S^{(1,1)}
        &= \spn\left\{
            \hspace{-0.5em}
            \begin{array}{c}
                \w_{\!\!\!\tikztableauscript{{1},{2}}\!\!\!}
            \end{array}
            \hspace{-0.5em}
            \right\}
        = \spn\left\{
            ab - ba
            \right\}.
        \\
        S^{(2,1)}
        &= \spn\left\{
            \hspace{-0.5em}
            \begin{array}{cc}
                \w_{\!\!\!\tikztableauscript{{1,2},{3}}\!\!\!},
                &
                \w_{\!\!\!\tikztableauscript{{1,3},{2}}\!\!\!}
            \end{array}
            \hspace{-0.5em}
            \right\}
        = \spn\left\{
            aab - baa, \,
            aba - baa
        \right\}.
        \\
        S^{(3,1)}
        &= \spn\left\{
            \hspace{-0.5em}
            \begin{array}{ccc}
                \w_{\!\!\!\tikztableauscript{{1,2,3},{4}}\!\!\!},
                &
                \w_{\!\!\!\tikztableauscript{{1,2,4},{3}}\!\!\!},
                &
                \w_{\!\!\!\tikztableauscript{{1,3,4},{2}}\!\!\!}
            \end{array}
            \hspace{-0.5em}
        \right\}
        = \spn\left\{
            \begin{array}{c}
            aaab - baaa, \\
            aaba - baaa, \\
            abaa - baaa
            \end{array}
        \right\}.
        \\
        S^{(2,2)}
        &= \spn\left\{
            \hspace{-0.5em}
            \begin{array}{cc}
                \w_{\!\!\!\tikztableauscript{{1,2},{3,4}}\!\!\!},
                &
                \w_{\!\!\!\tikztableauscript{{1,3},{2,4}}\!\!\!}
            \end{array}
            \hspace{-0.5em}
        \right\}
        = \spn\left\{
            \begin{array}{c}
                aabb - abba - baab + bbaa, \\
                abab - abba - baab + baba
            \end{array}
        \right\}.
        \qedhere
    \end{align*}
\end{Example}

\subsubsection{Level 1 lifting}
\label{sssec:level1lifting}

It turns out that $\Int_1$ is a linear transformation from $S^\lambda$
to $S^{\lambda + \vec e_1}$, where $\lambda + \vec e_1$ is obtained from
$\lambda$ by adding $1$ to its first component, that maps eigenvectors of
$\Rand_n$ to eigenvectors of $\Rand_{n+1}$.

\begin{Proposition}[Special case of \autoref{thm:lifting-eigenvectors}]
    \label{prop:lifting-eigenvectors-1}
    Let $\lambda = (\lambda_1, \lambda_2, \dots, \lambda_r) \vdash n$.

    If $v \in S^\lambda$ is an eigenvector of $\Rand_n$ with eigenvalue
    $\varepsilon$, then
    \begin{equation*}
        \Int_1(v) \in S^{(\lambda_1 + 1, \lambda_2, \dots, \lambda_r)}
    \end{equation*}
    and
    \begin{equation*}
        \Rand_{n+1}\left( \Int_1(v) \right)
        =
        \left(\varepsilon + (n+1) + \lambda_1\right) \, \Int_1(v).
    \end{equation*}
    Thus, $\Int_1(v)$ is an eigenvector of $\Rand_{n+1}$ with eigenvalue
    $\varepsilon + (n+1) + \lambda_1$.
\end{Proposition}

\begin{Example}
    \label{ex:eigenbases-for-hook}
    Note that $ab - ba \in S^{(1,1)}$ belongs to the kernel of $\Rand_2$.
    Hence,
    \begin{gather*}
        \Int_1(ab - ba) = 2 \left(aab - baa\right)
    \end{gather*}
    belongs to $S^{(2,1)}$ and it is an eigenvector $\Rand_3$ with eigenvalue
    $0 + (2+1) + 1 = 4$.
    Since $S^{(2,1)}$ is $2$-dimensional and we have one eigenvector, we obtain
    a second eigenvector by taking any vector in $S^{(2,1)}$ that is orthogonal
    to the first. Hence,
    \begin{align*}
        aab - 2 aba + baa
    \end{align*}
    is an eigenvector; it belongs to the kernel of $\Rand_3$.

    Applying $\Int_1$ to the above two eigenvectors in $S^{(2,1)}$, we obtain
    \begin{align*}
        \Int_1(aab - baa)
        &= 3 aaab + aaba - abaa - 3 baaa, \\
        \Int_1(aab - 2 aba + baa)
        &= 3 \left( aaab - aaba - abaa + baaa \right),
    \end{align*}
    which are eigenvectors of $\Rand_4$ with eigenvalues $10$ and $6$,
    respectively. Moreover, they belong to $S^{(3,1)}$.
    The orthogonal complement of the subspace spanned by these two vectors,
    with respect to the inner product in which words form an orthonormal basis, is
    $1$-dimensional and yields an element of the kernel of $\Rand_4$:
    \begin{gather*}
        aaab - 3 aaba + 3 abaa - baaa.
        \qedhere
    \end{gather*}
\end{Example}

\begin{Remark}[Eigenbasis for $S^{(n-1,1)}$]
    By iterating $\Int_1$, as in the above example, we obtain a method to
    recursively generate an eigenbasis for $\Rand_n$ in $S^{(n-1,1)}$.
    \begin{itemize}
        \item
            Apply $\Int_1$ to an eigenbasis for $S^{(n-2,1)}$;
            this gives $n-2$ eigenvectors with distinct positive integer eigenvalues.

        \item
            The last eigenvector belongs to the kernel of $\Rand_n$; it is
            \begin{gather*}
                \sum_{j = 1}^{n} (-1)^{j-1} \binom{n - 1}{j - 1}
                \,
                \overbrace{aa \cdots a}^{j-1} \, b \, \overbrace{aa \cdots a}^{n-j}.
            \end{gather*}
    \end{itemize}
    To see that this lies in the kernel of $\Rand_n$, first observe
    that it is sufficient to show that this lies in the kernel of $\RTop_n$.
    Next, consider a term $a^{j-1} b a^{n-j}$ in the image of the above
    element under $\RTop_n$:
    \begin{itemize}
        \item
            it can arise from moving an occurrence of $a$ from the factor
            $a^{n-j}$ of $a^{j-1} b a^{n-j}$ to the end of the word--there
            are $n-j$ ways to do this;
        \item
            it can arise from moving an occurrence of $a$ from the factor
            $a^{j}$ of $a^{j} b a^{n-j-1}$ to the end of the word---there
            are $j$ ways to do this;
        \item
            so the coefficient of this term is
            $(n - j) (-1)^{j-1} \binom{n-1}{j-1} + j (-1)^{j} \binom{n-1}{j} = 0$.
    \end{itemize}
\end{Remark}

Consequently, we recover the following result of Uyemura-Reyes
\cite[Theorem~5.4]{Reyes2002} on the eigenvalues of the random-to-random
Markov chain acting on $S^{(n-1,1)}$. Recall that the eigenvalues coincide
with the eigenvalues of $\frac{1}{n^2} \Rand_n$.

\begin{Corollary}[Uyemura-Reyes]
    The random-to-random Markov chain acting on $S^{(n-1,1)}$ has
    a simple eigenvalue $\frac{1}{n^2}\left(n(n-1) - j(j-1)\right)$
    for each $j \in \{2, 3, \dots, n\}$.
\end{Corollary}
Uyemura-Reyes also constructed eigenvectors for $\Rand_n$ acting on
$M^{(n-1,1)} \cong S^{n} \oplus S^{(n-1,1)}$ in which the entries of the
eigenvectors are obtained by evaluating the discrete Chebyshev and Legendre
orthogonal polynomials \cite[Lemma~5.15]{Reyes2002}.

\subsubsection{Level 2 lifting}
\label{sssec:level2lifting}

We need a modified version of the linear transformation $\Int_1$ to
construct eigenvectors in $S^{(2,2)}$ from eigenvectors in $S^{(2,1)}$.

In \autoref{ex:eigenbases-for-hook} we determined a basis of $S^{(2,1)}$ consisting
of eigenvectors for $\Rand_3$.
Naively, we apply $\Int_2$ to one of these eigenvectors and obtain
\begin{align*}
    \Int_2(aab - 2 aba + baa) =
    2 aabb - abab - 4 abba + 2 baab - baba + 2 bbaa.
\end{align*}
This is not an eigenvector of $\Rand_4$;
moreover, it does not even belong to $S^{(2,2)}$.
However, we find that the above can be written as
\begin{equation}
    \label{eqn:second-shuffle}
    \begin{aligned}
        &\Int_2(aab - 2 aba + baa) =
        \\
        & \qquad \qquad \quad
        \big( 2 aabb - abab - abba - baab - baba + 2 bbaa \big)
        - 3 \big( abba - baab \big).
    \end{aligned}
\end{equation}

The first summand of \autoref{eqn:second-shuffle} belongs to $S^{(2,2)}$,
since by comparing with \autoref{ex:specht-modules} we see that it is equal to
\begin{gather*}
    2 \, \w_{\!\!\!\tikztableauscript{{1,2},{3,4}}}
    - \w_{\!\!\!\tikztableauscript{{1,3},{2,4}}}.
\end{gather*}
Moreover, this is an eigenvector for $\Rand_4$ with eigenvalue $4$.

The second summand of \autoref{eqn:second-shuffle} is also an eigenvector for
$\Rand_4$ (eigenvalue $6$). Moreover, it can be obtained directly from
$aab - 2aba + baa$: start with
\begin{gather*}
    \Int_1(aab - 2aba + baa) = 3 \left(aaab - aaba - abaa + baaa\right),
\end{gather*}
and then substitute each word by the sum of all words obtained from it by
replacing an occurrence of $a$ (the first letter of the alphabet)
by $b$ (the second letter),
\begin{gather*}
    \Repl_{1,2}\big(\Int_1(aab - 2aba + baa)\big)
    = -6 \left( abba - baab \right).
\end{gather*}

To summarize, from the eigenvector $aab - 2aba + baa$ (eigenvalue $0$), we obtain
an eigenvector with eigenvalue $4$ that belongs to $S^{(2,2)}$ by applying the
operator
\begin{equation}
    \label{ex:lifting-formula-for-21-to-22}
    \Int_2 - \frac{1}{2} \Repl_{1,2} \circ \Int_1.
\end{equation}

\begin{Proposition}[Special case of \autoref{thm:lifting-eigenvectors}]
    \label{prop:lifting-eigenvectors-2}
    Let $\lambda = (\lambda_1, \lambda_2, \dots, \lambda_r)$ be a partition
    of $n$ such that
    $\lambda + \vec e_2 = (\lambda_1, \lambda_2 + 1, \dots, \lambda_r)$
    is a partition of $n+1$.

    If $v \in S^\lambda$ is an eigenvector of $\Rand_n$ with eigenvalue
    $\varepsilon$, then
    \begin{equation*}
        \Lift_2^\lambda(v)
        = \Int_2(v) + \frac{1}{(\lambda_2 - 2) - (\lambda_1 - 1)} \Repl_{1, 2}\big(\Int_1(v)\big)
    \end{equation*}
    belongs to the Specht module $S^{\lambda + \vec e_2}$ and
    \begin{equation*}
        \Rand_{n+1}\left( \Lift_2^\lambda(v) \right)
        =
        \left(\varepsilon + (n+1) + (\lambda_2 + 1) - 2\right) \, \Lift_2^\lambda(v).
    \end{equation*}
    Thus, $\Lift_2^\lambda(v)$, if nonzero, is an eigenvector of $\Rand_{n+1}$
    with eigenvalue $$\varepsilon + (n+1) + (\lambda_2 + 1) - 2.$$
\end{Proposition}

\subsection{Lifting eigenvectors}
\label{ssec:general-case}

We introduce linear transformations
$\Lift_i^\lambda$, from $S^\lambda$ to $S^{\lambda + \vec e_i}$,
for $i \geq 1$, that map eigenvectors of $\Rand_n$ to eigenvectors of
$\Rand_{n+1}$.
When $i = 1$, this operator coincides with the operator $\Int_1$ from
\autoref{sssec:level1lifting}.
For $i > 1$, they are modified versions of the $\Int_i$ operators
(\cf \autoref{sssec:level2lifting}).

\subsubsection{Lifting operators}
\label{sssec:lifting operators}

Let $A = \{a_1, a_2, a_3, \dots\}$ be an ordered alphabet.
Define the \emph{replacement operator} $\Repl_{i,j}$ to be the linear
transformation that maps a word $w = w_1 w_2 \cdots w_n$ to the sum of all the
words that can be obtained from $w$ by replacing exactly one occurrence of
$a_i$ in $w$ by $a_j$; explicitly,
\begin{align*}
    \Repl_{i, j}(w) = \sum_{\substack{1\leq k\leq n \\ w_k = \mathred{a_i}}}
                w_1 \cdots w_{k-1} \cdot \mathblue{a_j} \cdot w_{k+1} \cdots w_n.
\end{align*}

\begin{Example}
    For the usual alphabet $A = \{a, b, c, \dots\}$, we have
    \begin{gather*}
        \Repl_{\mathred{1}, \mathblue 2}(aaab) = {\mathblue b}aab + a{\mathblue b}ab + aa{\mathblue b}b.
        \qedhere
    \end{gather*}
\end{Example}

Recall that $M^\lambda$ is the vector space spanned by words of evaluation
$\lambda$ and that the Specht module $S^\lambda$ is a subspace of $M^\lambda$
(\cf \autoref{sssec:specht-modules}).
Hence, $M^\lambda = S^\lambda \oplus (S^\lambda)^\perp$,
where the orthogonal complement is computed with respect to the inner product
on $M^\lambda$ in which words form an orthonormal basis.
Let $\proj:M^\lambda \to S^\lambda$ be the corresponding projection onto
$S^\lambda$.
(See \autoref{sssec:isotypic-projectors} for a reformulation
in terms of ``isotypic projectors''.)



\begin{Theorem}
    \label{thm:lifting-eigenvectors}
    Let $\lambda = (\lambda_1, \lambda_2, \dots, \lambda_l)$
    be a partition of $n$ and
    let $1 \leq i \leq l + 1$ be such that
    $\lambda + \vec e_i = (\lambda_1, \lambda_2, \dots,
        \lambda_{i-1}, \lambda_i + 1, \lambda_{i+1}, \dots, \lambda_l)$
    is a partition of $n + 1$.

    Consider the composition of $\Int_i|_{S^\lambda}$, $1 \leq i \leq l+1$, with the
    projection onto the submodule
    $S^{\lambda + \vec e_i}$ of $M^{\lambda + \vec e_i}$:
    \begin{equation*}
        \Lift_i^\lambda :
        S^{\lambda} \xrightarrow{\Int_i} M^{\lambda + \vec e_i}
                    \xrightarrow{\proj} S^{\lambda + \vec e_i}.
    \end{equation*}
    \begin{enumerate}[label=(\alph*)]
        \item\label{thm:lifting-eigenvectors-part-a}
            $\Lift_i^\lambda$ maps eigenvectors of $\Rand_n$ in $S^\lambda$ to
            eigenvectors of $\Rand_{n+1}$ in $S^{\lambda + \vec e_i}$;
            explicitly,
            if $v \in S^{\lambda}$ is an eigenvector of $\Rand_n$ of eigenvalue
            $\varepsilon$, then
            $\Lift_i^\lambda(v)$, if nonzero, is an eigenvector of
            $\Rand_{n+1}$ with eigenvalue
            $\varepsilon + (n+1) + (\lambda_i + 1) - i$.

        \item\label{thm:lifting-eigenvectors-part-b}
            $\Lift_i^\lambda$ admits the following description:
            \begin{equation}
                \label{eq:general-lifting}
                \Lift_i^\lambda =
                \sum_{1 \leq b_1 < \cdots < b_t < b_{t+1} = i}
                \left(
                    \prod_{j=1}^{t}
                        \frac{1}{(\lambda_i - i) - (\lambda_{b_j} - {b_j})} \Repl_{b_j, b_{j+1}}
                \right)
                \Int_{b_1}.
            \end{equation}
    \end{enumerate}
\end{Theorem}

When $\lambda$ is clear from the context, we will write $\Lift_i$ instead of
$\Lift_i^\lambda$.

\begin{proof}
    \autoref{sec:main-proofs} is dedicated to proving this result:
    \autoref{thm:lifting-eigenvectors-part-a} is a special case of
    \autoref{liftingeigenvectors2};
    and \autoref{thm:lifting-eigenvectors-part-b} follows from
    \autoref{reformulation-of-morphisms}.
    See the beginning of \autoref{sec:main-proofs}
    for a detailed outline of the argument.
\end{proof}

The next two examples show that the operators from previous sections are
special cases of \autoref{eq:general-lifting}.

\begin{Example}
    \autoref{prop:lifting-eigenvectors-1} is a special case of
    \autoref{thm:lifting-eigenvectors}. Indeed, if $i = 1$, then there is only one
    sequence satisfying the conditions of the summation in
    \autoref{eq:general-lifting}, namely $(b_{1}) = (1)$.
    Thus, \autoref{eq:general-lifting} reduces to $\Lift_1^\lambda = \Int_1$.
\end{Example}

\begin{Example}

    Let $\lambda = (2,1)$ and $i = 2$. There are two sequences satisfying the
    conditions of the summation in \autoref{eq:general-lifting}, namely $(b_{1})
    = (2)$ and $(b_1, b_2) = (1,2)$:
    \begin{align*}
        \Lift_2^{(2,1)} =
        \Int_2 + \frac{1}{(\lambda_2 - 2) - (\lambda_1 - 1)} \Repl_{1,2} \circ \Int_1
        =
        \Int_2 - \frac{1}{2} \Repl_{1,2} \circ \Int_1,
    \end{align*}
    which coincides with the linear transformation in
    \autoref{ex:lifting-formula-for-21-to-22}.
    We saw that this maps the eigenvector $aab - 2aba + baa \in S^{(2,1)}$ with
    eigenvalue $0$ to an eigenvector in $S^{(2,2)}$ with eigenvalue
    $0 + (n + 1) + (\lambda_2 + 1) - 2 = 4$.
\end{Example}

\subsubsection{Graphical description}
Here is a graphical description of the operator $\Lift_i^\lambda$.
Consider the following directed graph (refer to \autoref{fig:composition-graphs}
for an example).
\begin{itemize}
    \item
        The vertices of the directed graph consist of $\lambda$ and
        $\lambda + \vec e_j$ for $1 \leq j \leq i$;
    \item
        for each $j$,
        there is an arrow $\lambda \xrightarrow{} \lambda + \vec e_j$ labelled
        $\Int_j$; and
    \item
        for $j < k$, there is an arrow
        $\lambda + \vec e_j \xrightarrow{} \lambda + \vec e_k$
        labelled
        $\frac{1}{(\lambda_i - i) - (\lambda_j - j)} \Repl_{j, k}.$
\end{itemize}
Then $\Lift_i^\lambda$ has one summand for each path in the directed
graph that begins at $\lambda$ and ends at $\lambda + \vec e_i$;
the corresponding summand is the composition of the linear transformations
labelling the edges.

\begin{figure}[b]
    \centering
    \begin{subfigure}[c]{0.35\linewidth}
        \begin{tikzpicture}[scale=0.60, >=latex,line join=bevel,]
        \node (4+1+1) at (90.999bp,126bp) [draw,draw=none] {${\def\lr#1{\multicolumn{1}{|@{\hspace{.6ex}}c@{\hspace{.6ex}}|}{\raisebox{-.3ex}{$#1$}}}\raisebox{-.6ex}{$\begin{array}[b]{*{4}c}\cline{1-4}\lr{\phantom{x}}&\lr{\phantom{x}}&\lr{\phantom{x}}&\lr{\phantom{x}}\\\cline{1-4}\lr{\phantom{x}}\\\cline{1-1}\lr{\phantom{x}}\\\cline{1-1}\end{array}$}}$};
        \node (3+2+1) at (24.999bp,22bp) [draw,draw=none] {${\def\lr#1{\multicolumn{1}{|@{\hspace{.6ex}}c@{\hspace{.6ex}}|}{\raisebox{-.3ex}{$#1$}}}\raisebox{-.6ex}{$\begin{array}[b]{*{3}c}\cline{1-3}\lr{\phantom{x}}&\lr{\phantom{x}}&\lr{\phantom{x}}\\\cline{1-3}\lr{\phantom{x}}&\lr{\phantom{x}}\\\cline{1-2}\lr{\phantom{x}}\\\cline{1-1}\end{array}$}}$};
        \node (3+1+1) at (24.999bp,230bp) [draw,draw=none] {${\def\lr#1{\multicolumn{1}{|@{\hspace{.6ex}}c@{\hspace{.6ex}}|}{\raisebox{-.3ex}{$#1$}}}\raisebox{-.6ex}{$\begin{array}[b]{*{3}c}\cline{1-3}\lr{\phantom{x}}&\lr{\phantom{x}}&\lr{\phantom{x}}\\\cline{1-3}\lr{\phantom{x}}\\\cline{1-1}\lr{\phantom{x}}\\\cline{1-1}\end{array}$}}$};
        \draw [BrickRed,->] (4+1+1) ..controls (67.767bp,89.391bp) and (54.87bp,69.07bp)  .. (3+2+1);
        \definecolor{strokecol}{rgb}{0.0,0.0,0.0};
        \pgfsetstrokecolor{strokecol}
        \draw (89.499bp,74bp) node {$-\frac{1}{3}\Repl_{1,2}$};
        \draw [RoyalBlue,->] (3+1+1) ..controls (48.232bp,193.39bp) and (61.128bp,173.07bp)  .. (4+1+1);
        \draw (80.499bp,178bp) node {$\Int_1$};
        \draw [RoyalBlue,->] (3+1+1) ..controls (11.72bp,192.08bp) and (4.8811bp,168.95bp)  .. (1.9991bp,148bp) .. controls (-0.66638bp,128.63bp) and (-0.66638bp,123.37bp)  .. (1.9991bp,104bp) .. controls (4.3407bp,86.981bp) and (9.2945bp,68.52bp)  .. (3+2+1);
        \draw (19.499bp,126bp) node {$\Int_2$};
        \end{tikzpicture}
    \end{subfigure}
    \begin{subfigure}[c]{0.4\linewidth}
        \begin{tikzpicture}[scale=0.60, >=latex,line join=bevel,]
        \node (4+2) at (96bp,212bp) [draw,draw=none] {${\def\lr#1{\multicolumn{1}{|@{\hspace{.6ex}}c@{\hspace{.6ex}}|}{\raisebox{-.3ex}{$#1$}}}\raisebox{-.6ex}{$\begin{array}[b]{*{4}c}\cline{1-4}\lr{\phantom{x}}&\lr{\phantom{x}}&\lr{\phantom{x}}&\lr{\phantom{x}}\\\cline{1-4}\lr{\phantom{x}}&\lr{\phantom{x}}\\\cline{1-2}\end{array}$}}$};
        \node (3+2+1) at (108bp,22bp) [draw,draw=none] {${\def\lr#1{\multicolumn{1}{|@{\hspace{.6ex}}c@{\hspace{.6ex}}|}{\raisebox{-.3ex}{$#1$}}}\raisebox{-.6ex}{$\begin{array}[b]{*{3}c}\cline{1-3}\lr{\phantom{x}}&\lr{\phantom{x}}&\lr{\phantom{x}}\\\cline{1-3}\lr{\phantom{x}}&\lr{\phantom{x}}\\\cline{1-2}\lr{\phantom{x}}\\\cline{1-1}\end{array}$}}$};
        \node (3+2) at (96bp,304bp) [draw,draw=none] {${\def\lr#1{\multicolumn{1}{|@{\hspace{.6ex}}c@{\hspace{.6ex}}|}{\raisebox{-.3ex}{$#1$}}}\raisebox{-.6ex}{$\begin{array}[b]{*{3}c}\cline{1-3}\lr{\phantom{x}}&\lr{\phantom{x}}&\lr{\phantom{x}}\\\cline{1-3}\lr{\phantom{x}}&\lr{\phantom{x}}\\\cline{1-2}\end{array}$}}$};
        \node (3+3) at (20bp,120bp) [draw,draw=none] {${\def\lr#1{\multicolumn{1}{|@{\hspace{.6ex}}c@{\hspace{.6ex}}|}{\raisebox{-.3ex}{$#1$}}}\raisebox{-.6ex}{$\begin{array}[b]{*{3}c}\cline{1-3}\lr{\phantom{x}}&\lr{\phantom{x}}&\lr{\phantom{x}}\\\cline{1-3}\lr{\phantom{x}}&\lr{\phantom{x}}&\lr{\phantom{x}}\\\cline{1-3}\end{array}$}}$};
        \draw [BrickRed,->] (4+2) ..controls (98.944bp,165.39bp) and (103.38bp,95.104bp)  .. (3+2+1);
        \definecolor{strokecol}{rgb}{0.0,0.0,0.0};
        \pgfsetstrokecolor{strokecol}
        \draw (132.5bp,120bp) node {$-\frac{1}{5}\Repl_{1,3}$};
        \draw [BrickRed,->] (3+3) ..controls (13.789bp,91.915bp) and (12.85bp,74.479bp)  .. (21bp,62bp) .. controls (33.618bp,42.68bp) and (58.129bp,32.552bp)  .. (3+2+1);
        \draw (45.5bp,74bp) node {$-\frac{1}{3}\Repl_{2,3}$};
        \draw [BrickRed,->] (4+2) ..controls (51.295bp,196.96bp) and (26.848bp,186.9bp)  .. (21bp,178bp) .. controls (14.863bp,168.66bp) and (13.859bp,156.51bp)  .. (3+3);
        \draw (48.5bp,166bp) node {$-\frac{1}{5}\Repl_{1,2}$};
        \draw [RoyalBlue,->] (3+2) ..controls (96bp,274.65bp) and (96bp,254.35bp)  .. (4+2);
        \draw (116.5bp,258bp) node {$\Int_1$};
        \draw [RoyalBlue,->] (3+2) ..controls (54.352bp,281.49bp) and (21.122bp,258.63bp)  .. (7bp,228bp) .. controls (-5.137bp,201.68bp) and (1.9536bp,168.09bp)  .. (3+3);
        \draw (22.5bp,212bp) node {$\Int_2$};
        \draw [RoyalBlue,->] (3+2) ..controls (129.16bp,292.15bp) and (145.57bp,283.31bp)  .. (155bp,270bp) .. controls (190.01bp,220.56bp) and (174.42bp,196.33bp)  .. (180bp,136bp) .. controls (181.31bp,121.84bp) and (185.01bp,117.31bp)  .. (180bp,104bp) .. controls (171.16bp,80.537bp) and (152.22bp,59.418bp)  .. (3+2+1);
        \draw (198.5bp,166bp) node {$\Int_3$};
        \end{tikzpicture}
    \end{subfigure}
    \caption{Graphical depictions of $\Lift_2^{(3,1,1)}$ and $\Lift_3^{(3,2)}$.}
    \label{fig:composition-graphs}
\end{figure}

\begin{Example}
    \label{ex:liftings-for-321}
    Referring to the directed graphs in \autoref{fig:composition-graphs},
    \begin{gather*}
        \Lift_2^{(3,1,1)} = \Int_2 - \frac{1}{3}\Repl_{1,2} \circ \Int_1
    \end{gather*}
    and
    \begin{gather*}
        \Lift_3^{(3,2)} =
        \Int_3 - \frac{1}{5} \Repl_{1,3} \circ \Int_1
               - \frac{1}{3} \Repl_{2,3} \circ \Int_2
               + \frac{1}{15}\Repl_{2,3} \circ \Repl_{1,2} \Int_1.
    \end{gather*}
    These two linear transformations, together with
    \begin{gather*}
        \Lift_1^{(2,2,1)} = \Int_1 : S^{(2,2,1)} \longrightarrow S^{(3,2,1)},
    \end{gather*}
    are all the ways to obtain eigenvectors for $\Rand_6$ in $S^{(3,2,1)}$
    using \autoref{thm:lifting-eigenvectors}.
    Moreover, all eigenvectors for $\Rand_6$ in $S^{(3,2,1)}$ can be obtained
    in this way.
\end{Example}

\subsection{Constructing eigenspaces and eigenbases}
\label{constructing-eigenspaces-and-eigenbases}

\begin{figure}[b]
    \captionsetup{width=0.9\textwidth}
    \centering
    \includegraphics[width=\textwidth]{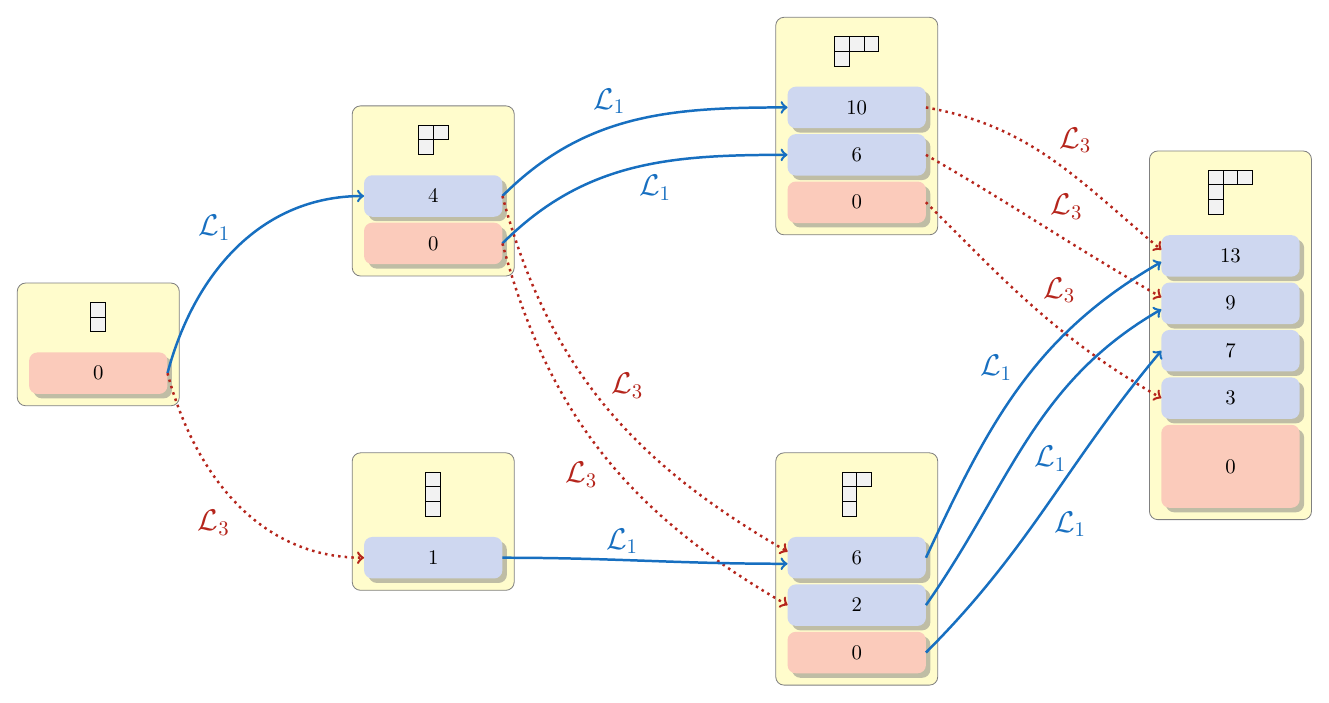}
    \caption{
        A depiction of \autoref{thm:eigenspace-decomposition},
        which describes the eigenspace decomposition for
        $\Rand_n$ acting on $S^\lambda$.
        Each cell represents an eigenspace and the number in a cell
        is the corresponding eigenvalue.
        Here, we see how each of the non-kernel eigenspaces in $S^{(3,1,1)}$ is
        the image of some $\ker\left(\Rand^\mu\right)$ under $\Lift^{\lambda/\mu}$.
    }
    \label{fig:eigenspace-decomposition-311}
\end{figure}

In this section, we describe how to construct an explicit eigenbasis for
$\Rand_n$ acting on $S^\lambda$ starting with bases for the kernels of
$\Rand_m$ acting on $S^\mu$ for all partitions $\mu$ with $|\mu| = m \leq n$.

As we will see, all the eigenvectors of $\Rand_{n}$ acting on $S^{\lambda}$,
except those that lie in its kernel, can be obtained using the linear
transformations $\Lift_i$.

More precisely, if $v \in S^\lambda$ is an eigenvector of $\Rand_n$ that does
not belong to $\ker(\Rand_n)$, then we can write $v$ in the form
\begin{equation}
    \label{eq:v-as-Lift}
    v = \Lift_{i_1}\Big(v^{(1)}\Big)
      + \Lift_{i_2}\Big(v^{(2)}\Big)
      + \cdots
      + \Lift_{i_r}\Big(v^{(r)}\Big),
\end{equation}
where $v^{(j)}$ is an eigenvector of $\Rand_{n-1}$ in some Specht
module $S^\mu$ with $\lambda = \mu + \vec e_{i_j}$.
If $v^{(j)} \notin \ker(\Rand_{n-1})$,
then we can apply the same reasoning to write $v^{(j)}$ in the form
given by \autoref{eq:v-as-Lift}.
Continuing in this way, we can express $v$ as a sum of images of elements in
the kernel of some $\Rand_{n-k}$ under maps of the following form.

\begin{Definition}
    For any skew partition $\lambda/\mu$, let $r_1 \leq r_2 \leq \dots \leq r_k$
    be such that $\lambda$ is obtained from $\mu$ by adding cells in rows $r_1, \dots,
    r_k$ (counting multiplicities). Define
    \begin{gather*}
        \Lift^{\lambda/\mu} = \Lift_{r_k} \circ \cdots \circ \Lift_{r_2} \circ \Lift_{r_1}.
    \end{gather*}
\end{Definition}

By the above, every eigenvector of $\Rand_n$ in $S^\lambda$ is a sum of elements of the
form $\Lift^{\lambda/\mu}(u)$ for partitions $\mu \subseteq \lambda$ and
some element $u \in S^\mu$ in the kernel of $\Rand_{|\mu|}$.
It will follow from \autoref{image-of-composition-of-shuffles}
that $\Lift^{\lambda/\mu} = 0$ if $\lambda/\mu$ is not a horizontal
strip (\ie, if $\lambda/\mu$ contains two cells in the same column).

For a partition $\mu$ of $m$, let $\Rand^\mu$ denote the restriction of
$\Rand_m$ to $S^\mu$.

\begin{Theorem}
    \label{thm:eigenspace-decomposition}
    Let $\lambda$ be a partition of $n$. Then, as vector spaces,
    \begin{gather*}
        S^\lambda
        = \bigoplus_{
            \substack{
                \mu \colon \lambda/\mu \text{~is a} \\
                \text{horizontal strip}
            }
        }
        \Lift^{\lambda/\mu}\Big(\ker \Rand^{\mu}\Big).
    \end{gather*}
    Furthermore, the $\varepsilon$-eigenspace of $\Rand_n$ acting on
    $S^\lambda$ is
    \begin{gather*}
        \bigoplus_{
            \substack{
                \mu \colon \lambda/\mu \text{~is a} \\
                \text{horizontal strip} \\
                \text{such that~} \eig(\lambda/\mu) = \varepsilon
            }
        }
        \Lift^{\lambda/\mu}\Big(\ker \Rand^\mu\Big).
    \end{gather*}
    Moreover, if $\mathcal K_{\mu}$ denotes a basis of the kernel of
    $\Rand$ acting on $S^\mu$, then
    \begin{gather*}
        \bigcup_{
            \substack{
                \mu \colon \lambda/\mu \text{~is a} \\
                \text{horizontal strip}
            }
        }
        \Lift^{\lambda/\mu}\left(\mathcal K_{\mu}\right)
    \end{gather*}
    is an eigenbasis for the action of $\Rand_n$ on $S^\lambda$.
\end{Theorem}

\begin{proof}
    \autoref{sec:main-proofs} is dedicated to proving this result:
    this is \autoref{directsumdecomposition} combined with
    \autoref{suffices-to-project-at-end} and \autoref{thm:lifting-eigenvectors}.
\end{proof}

\autoref{fig:eigenspace-decomposition-311} and
\autoref{fig:eigenspace-decomposition-32} illustrate
\autoref{thm:eigenspace-decomposition}
by showing how the decomposition of $S^{(3,1,1)}$ and $S^{(3,2)}$ into eigenspaces for $\Rand_5$
is obtained from the kernels of $\Rand_m$ ($1 \leq m \leq n$) acting on
$S^\mu$, as $\mu$ ranges over all partitions such that $\lambda/\mu$ is
a horizontal strip.

\begin{Example}[Eigenbasis for $S^{(3,2)}$]
    \label{ex:eigenbasis-32}
    Let $\lambda = (3,2)$. We use \autoref{thm:eigenspace-decomposition} to
    construct an eigenbasis for $S^{\lambda}$ (\cf
    \autoref{fig:eigenspace-decomposition-32}.)

    \begin{figure}[b]
        \captionsetup{width=0.9\textwidth}
        \centering
        \includegraphics[width=\textwidth]{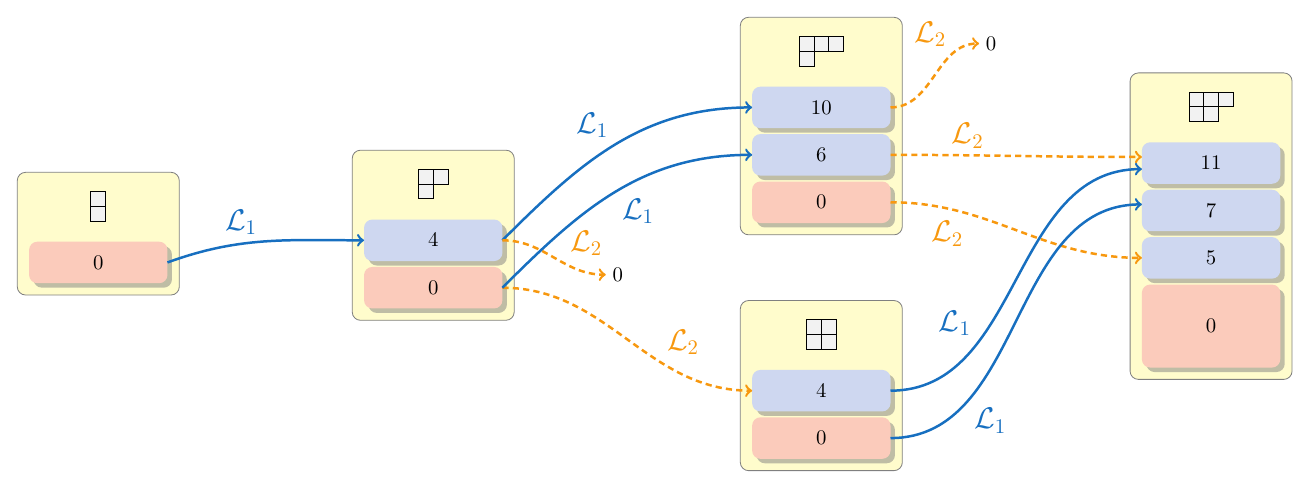}
        \caption{
            A depiction of \autoref{thm:eigenspace-decomposition},
            which describes the eigenspace decomposition for
            $\Rand_n$ acting on $S^\lambda$.
            Each cell represents an eigenspace and the number in a cell
            is the corresponding eigenvalue.
            Here, we see how the eigenspaces in $S^{(3,2)}$ arise as images of the
            kernels in $S^{(2,1)}$, $S^{(3,1)}$, $S^{(2,2)}$ and $S^{(3,2)}$
            (\cf \autoref{ex:eigenbasis-32}).
            Note that the restriction of $\Lift_2$ to the $10$-eigenspace of
            $S^{(3,1)}$ is $0$: the $10$-eigenspace is the image of the kernel
            in $S^{(1,1)}$ under $\Lift_1 \circ \Lift_1$; and the composition
            $\Lift_2 \circ \Lift_1 \circ \Lift_1 = \Lift^{(3,2)/(1,1)}$ is
            $0$ because $(3,2)/(1,1)$ is not a horizontal strip
            (\cf \autoref{image-of-composition-of-shuffles}).
        }
        \label{fig:eigenspace-decomposition-32}
    \end{figure}

    Since $S^\lambda$ is $5$-dimensional, we need $5$ linearly independent
    eigenvectors. We get these by
    applying $\Lift^{\lambda/\mu}$ to a
    basis $\mathcal K_\mu$ of $\ker(\Rand^{\mu})$
    for each of the following five horizontal strips $\lambda/\mu$.
    \begin{gather*}
        \begin{array}{cccccc}
        \tikztableausmall{{    X,    X,\null},{\null,\null}}
        &
        \tikztableausmall{{    X,    X,    X},{\null,\null}}
        &
        \tikztableausmall{{    X,    X,\null},{    X,\null}}
        &
        \tikztableausmall{{    X,    X,\null},{    X,    X}}
        &
        \tikztableausmall{{    X,    X,    X},{    X,\null}}
        &
        \tikztableausmall{{    X,    X,    X},{    X,    X}}
        \\
        \text{\small $\mu = (2)$}
        &
        \text{\small $\mu = (3)$}
        &
        \text{\small $\mu = (2,1)$}
        &
        \text{\small $\mu = (2,2)$}
        &
        \text{\small $\mu = (3,1)$}
        &
        \text{\small $\mu = (3,2)$}
        \end{array}
    \end{gather*}

    For $\mu = (2)$, we do not obtain any eigenvectors since $\ker(\Rand^{(2)}) = 0$
    (because there are no desarrangement tableaux of shape $(2)$).

    For $\mu = (3)$, we do not obtain any eigenvectors since $\ker(\Rand^{(3)}) = 0$
    (because there are no desarrangement tableaux of shape $(3)$).

    For $\mu = (2, 1)$, we obtain
    an eigenvector associated with the eigenvalue $11$:
    \begin{equation*}
        \begin{aligned}
            \mathcal K_{(2,1)}
            &= \left\{
                    aab - 2 aba + baa
                \right\}
             = \left\{
                \hspace{-0.4em}
                \begin{array}{c}
                    \w_{\!\!\!\tikztableauscript{{1,2},{3}}}
                    - 2 \w_{\!\!\!\tikztableauscript{{1,3},{2}}}
                \end{array}
                \hspace{-0.9em}
                \right\};
            \\
            \Lift^{(3,2)/(2,1)}\left(\mathcal K_{(2,1)}\right)
            &
            = \left(
                \left(\Int_2 - \frac{1}{2} \Repl_{1, 2} \circ \Int_1\right)
                \circ \Int_1
              \right)
              \left(\mathcal K_{(2,1)}\right)
            \\
            & = \left\{
                4 aaabb - 2 abaab - 2 ababa - 2 baaab - 2 baaba + 4 bbaaa
                \right\}
            \\
            &
            = \left\{
                \hspace{-0.4em}
                \begin{array}{c}
                    4 \w_{\!\!\!\tikztableauscript{{1,2,3},{4,5}}}
                    - 2 \w_{\!\!\!\tikztableauscript{{1,3,4},{2,5}}}
                    + 2 \w_{\!\!\!\tikztableauscript{{1,3,5},{2,4}}}
                \end{array}
                \hspace{-0.9em}
            \right\}.
        \end{aligned}
    \end{equation*}
    For $\mu = (2, 2)$, we obtain
    an eigenvector associated with the eigenvalue $7$:
    \begin{equation*}
        \begin{aligned}
            \mathcal K_{(2, 2)}
            &= \left\{
                    abab - abba - baab + baba
                \right\}
             = \left\{
                \hspace{-0.4em}
                \begin{array}{c}
                    \w_{\!\!\!\tikztableauscript{{1, 3}, {2, 4}}}
                \end{array}
                \hspace{-0.9em}
                \right\};
            \\
            \Lift^{(3,2)/(2, 2)}\left(\mathcal K_{(2, 2)}\right)
            & = \Int_1\left(\mathcal K_{(2, 2)}\right)
            \\
            & = \left\{
                \begin{array}{l}
                    2 aabab - 2 aabba + abaab + ababa \\
                    \qquad - 2 abbaa - 3 baaab + baaba + 2 babaa
                \end{array}
                \right\}
            \\
            &
            = \left\{
                \hspace{-0.4em}
                \begin{array}{c}
                    2 \w_{\!\!\!\tikztableauscript{{1, 2, 4}, {3, 5}}}
                    - 2 \w_{\!\!\!\tikztableauscript{{1, 2, 5}, {3, 4}}}
                    +   \w_{\!\!\!\tikztableauscript{{1, 3, 4}, {2, 5}}}
                    +   \w_{\!\!\!\tikztableauscript{{1, 3, 5}, {2, 4}}}
                \end{array}
                \hspace{-0.9em}
            \right\}.
        \end{aligned}
    \end{equation*}
    For $\mu = (3, 1)$, we obtain
    an eigenvector associated with the eigenvalue $5$:
    \begin{equation*}
        \begin{aligned}
            \mathcal K_{(3,1)}
            &= \left\{
                    aaab - 3 aaba + 3 abaa - baaa
                \right\}
                \\
            &= \left\{
                \hspace{-0.4em}
                \begin{array}{c}
                        \w_{\!\!\!\tikztableauscript{{1, 2, 3}, {4}}}
                    - 3 \w_{\!\!\!\tikztableauscript{{1, 2, 4}, {3}}}
                    + 3 \w_{\!\!\!\tikztableauscript{{1, 3, 4}, {2}}}
                \end{array}
                \hspace{-0.9em}
                \right\};
            \\
            \Lift^{(3,2)/(3,1)}\left(\mathcal K_{(3,1)}\right)
            & =
            \left(\Int_2 - \frac{1}{3} \Repl_{1, 2} \circ \Int_1\right)
            \left(\mathcal K_{(3,1)}\right)
            \\
            & = \left\{
                    \displaystyle \frac{10}{3}
                    \left(
                        aaabb - aabab - aabba + abbaa + babaa - bbaaa
                    \right)
                \right\}
            \\
            &
            = \left\{
                \frac{10}{3}
                \left(
                    \hspace{-0.4em}
                    \begin{array}{c}
                            \w_{\!\!\!\tikztableauscript{{1, 2, 3}, {4, 5}}}
                        - \w_{\!\!\!\tikztableauscript{{1, 2, 4}, {3, 5}}}
                        - \w_{\!\!\!\tikztableauscript{{1, 2, 5}, {3, 4}}}
                        + \w_{\!\!\!\tikztableauscript{{1, 3, 5}, {2, 4}}}
                    \end{array}
                    \hspace{-0.9em}
                \right)
            \right\}.
        \end{aligned}
    \end{equation*}
    For $\mu = (3, 2)$, we have
    two eigenvectors in the kernel:
    \begin{equation*}
        \begin{aligned}
          \Lift^{(3,2)/(3,2)} \left(\mathcal K_{(3,2)}\right)
        &= \mathcal K_{(3,2)}
        \\
        &= \left\{
            \begin{array}{c}
                aabab - aabba - abaab + abbaa + baaba - babaa, \\
                abaab - 2 ababa + abbaa - baaab + 2 baaba - babaa
            \end{array}
            \right\}
            \\
        &= \left\{
            \hspace{-0.4em}
            \begin{array}{l}
                  \w_{\!\!\!\tikztableauscript{{1, 2, 4}, {3, 5}}}
                - \w_{\!\!\!\tikztableauscript{{1, 2, 5}, {3, 4}}}
                - \w_{\!\!\!\tikztableauscript{{1, 3, 4}, {2, 5}}}
                \\
                    \w_{\!\!\!\tikztableauscript{{1, 3, 4}, {2, 5}}}
                - 2 \w_{\!\!\!\tikztableauscript{{1, 3, 5}, {2, 4}}}
            \end{array}
            \hspace{-0.9em}
            \right\}.
    \end{aligned}
    \end{equation*}
\end{Example}

\subsection{Frobenius characteristics of the eigenspaces}
\label{ssec:frobenius-characteristics}

We briefly restrict our attention to the set of words of evaluation
$\nu = (1, 1, \dots, 1)$, which we abbreviate as $\nu = 1^n$.
Since these words do not contain any repeated letters,
we can think of them as permutations.
This allows us to define an action of a permutation $\sigma$ on a word $w$ of
evaluation $1^n$ as their composition $\sigma \circ w$ as
permutations.

Note that $M^{1^n}$ is isomorphic to the two-sided regular
representation of $\symm_n$ (\ie, the group algebra of $\symm_n$).
Indeed, the left $\symm_n$--action defined above is the composition of
permutations, and the right $\symm_n$--action, which corresponds to
permutations acting on the positions of a word, is also the composition of
permutations if the word has evaluation $1^n$.

The eigenspaces of the random-to-random operator acting on $M^{1^n}$
are invariant subspaces for this (left) action of the symmetric group.
Indeed, if $\sigma \in S_n$ and $v \in M^{1^n}$ is an eigenvector of $\Rand_n$
with eigenvalue $\varepsilon$, then $\sigma v$ is also an eigenvector with the
same eigenvalue, because
\begin{equation*}
    (\sigma v) \cdot \Rand_n
    =
    \sigma (v \cdot \Rand_n)
    =
    \varepsilon (\sigma v).
\end{equation*}
Hence, one can ask for the irreducible decompositions of these left
$\symm_n$--modules, or equivalently, their Frobenius characteristics.

We will obtain these decompositions by starting with a known decomposition of
$M^{1^n}$ into a direct sum of irreducible two-sided $\symm_n$--modules.
The irreducible two-sided $\symm_n$--modules are all of the form
$(S^\mu)^\ast \otimes S^\lambda$,
where $(S^\mu)^\ast$ is isomorphic to the left $\symm_n$--module indexed by
$\mu \vdash n$.
Note that this expression in terms of tensor products neatly separates the
left and right actions:
the left action of $\symm_n$ on $(S^\mu)^\ast \otimes S^\lambda$
corresponds to acting by $\symm_n$ on the tensor factor $(S^\mu)^\ast$,
and the right action corresponds to acting on the right tensor factor
$S^\lambda$.
The only modules that occur in $M^{1^n}$ are of the form
$(S^\lambda)^\ast \otimes S^\lambda$, and we have
\begin{equation}
    \label{reg-rep-decomp}
    M^{1^n} \cong
    \bigoplus_{\lambda \vdash n} \left(S^\lambda\right)^\ast \otimes S^\lambda.
\end{equation}

\begin{Lemma}
    \label{lemma:kernel-dimensions}
    For every partition $\lambda \vdash n$,
    the nullity of $\Rand_n$ acting on $S^\lambda$ is
    \begin{equation*}
        \dim \ker(\Rand^\lambda) = \ndestab^\lambda,
    \end{equation*}
    where $\ndestab^\lambda$ is the number of desarrangement tableaux of shape
    $\lambda$.
\end{Lemma}

\begin{proof}


    Since the right action of $\symm_n$ on $M^{1^n}$ corresponds to acting on
    the right tensor factors $S^\lambda$ in \autoref{reg-rep-decomp},
    it follows that
    \begin{equation*}
        \ker(\Rand_n)
        \cong
        \bigoplus_{\lambda \vdash n}
        \left(S^\lambda\right)^\ast \otimes \ker(\Rand^\lambda).
    \end{equation*}
    Thus,
    $\dim\ker(\Rand^\lambda)$
    is equal to the multiplicity of
    $\left(S^\lambda\right)^\ast$ in $\ker(\Rand_n)$.

    Since the random-to-random shuffle is the symmetrization of the
    random-to-top shuffle, we have that $\ker(\Rand_n) = \ker(\RTop_n)$.
    The irreducible decomposition of $\ker(\RTop_n)$ as a left
    $\symm_n$--module was computed in \cite{ReinerWachs2002},
    \cite[Proposition~2.3]{ReinerWebb2004} and
    \cite[Proposition~VI.9.5]{RSW2014}
    (\cf \autoref{frob-char-for-r2t}):
    \begin{equation*}
        \ker(\RTop_n) \cong
        \bigoplus_{\substack{\text{desarrangement} \\ \text{tableaux $Q$ of size $n$}}}
        \left(S^{\shape(Q)}\right)^\ast.
    \end{equation*}
    Thus, the multiplicity of the irreducible $\symm_n$--module
    $\left(S^\lambda\right)^\ast$ in $\ker(\RTop_n)$ is $\ndestab^\lambda$,
    the number of desarrangement tableaux of shape $\lambda$.
\end{proof}

\begin{Theorem}
    \label{thm:frobenius-characteristic-R2R}
    The eigenspaces of the random-to-random shuffle acting on permutations of
    size $n$ are left modules for the symmetric group $\symm_n$.
    The Frobenius characteristic of the $\frac{1}{n^2} \varepsilon$-eigenspace is
    \begin{equation}
        \label{eq:frobenius-characteristic-R2R}
        \smashoperator{\sum_{
            \substack{
                \text{horizontal strips $\lambda/\mu$} \\
                \text{with } \eig(\lambda/\mu) = \varepsilon
            }}}
        \ \ndestab^\mu \, s_\lambda,
    \end{equation}
    where $s_\lambda$ is the Schur function indexed by $\lambda$
    and $\ndestab^\mu$ denotes the number of desarrangement tableaux of shape $\mu$.
\end{Theorem}

\begin{proof}
    Since we are interested in the left module structure only,
    we decompose the right tensor factors in \autoref{reg-rep-decomp}
    as vector spaces using
    \autoref{thm:eigenspace-decomposition}.
    Hence, we have the following isomorphism as left $\symm_n$--modules:
    \begin{gather*}
        \left(S^\lambda\right)^\ast \otimes S^\lambda
        \cong \bigoplus_{
                \text{horizontal strips~} \lambda/\mu
            }
            \left(S^\lambda\right)^\ast
                \otimes \Lift^{\lambda/\mu}\left(\ker \Rand^\mu\right).
    \end{gather*}
    By lumping together the horizontal strips $\lambda/\mu$ with the same value
    of $\eig(\lambda/\mu)$, the
    $\frac{1}{n^2}\varepsilon$--eigenspace is equal to the following direct sum
    of left $\symm_n$--modules:
    \begin{gather*}
        \bigoplus_{
          \substack{
              \lambda \vdash n \text{~and} \\
              \text{horizontal strips } \lambda/\mu \\
              \text{with~} \eig(\lambda/\mu) = \varepsilon
          }}
          \left(S^\lambda\right)^\ast \otimes
                \Lift^{\lambda/\mu}\left(\ker \Rand^\mu\right).
    \end{gather*}
    The above contains one copy of $\left(S^\lambda\right)^\ast$ for each
    vector in a basis of $\Lift^{\lambda/\mu}\left(\ker \Rand^\mu\right)$,
    so its Frobenius characteristic is,
    by \autoref{lemma:kernel-dimensions},
    \begin{gather*}
        \sum_{
          \substack{
              \lambda \vdash n \text{~and} \\
              \text{horizontal strips~} \lambda/\mu \\
              \text{with~} \eig(\lambda/\mu) = \varepsilon
          }}
            s_\lambda
            \times
            \dim\left(\ker \Rand^\mu\right)
        =
        \sum_{
          \substack{
              \lambda \vdash n \text{~and} \\
              \text{horizontal strips~} \lambda/\mu \\
              \text{with~} \eig(\lambda/\mu) = \varepsilon
          }}
          s_\lambda
          \times
          \ndestab^\mu.
          \qedhere
    \end{gather*}
\end{proof}

\begin{Example}
    With the help of {\tt Sage} \cite{Sage},
    the Frobenius characteristics of all the eigenspaces for $n \leq 8$ were
    explicitly computed and appear in
    \cite[Figures 4--6 of Chapter VI; Figures 1--6 of Appendix A]{RSW2014}.
    For instance, from this data, one has that the Frobenius characteristic of the
    $\frac{9}{36}$-eigenspace for the random-to-random operator acting on
    permutations of size $6$ is
    \begin{equation*}
        s_{(3,2,1)} + 2s_{(4,1,1)} + 2s_{(4,2)},
    \end{equation*}
    which corresponds to the following horizontal strips $\lambda/\mu$
    for which $\eig(\lambda/\mu) = 9$.
    \begin{equation*}
        \tikztableauscript{{X,X,X},{X,\null},{\null},}
        \quad
        \tikztableauscript{{X,X,X,\null},{X},{X},}
        \quad
        \tikztableauscript{{X,X,X,\null},{X,X},}
    \end{equation*}
\end{Example}

\begin{Remark}[Frobenius characteristic for random-to-top eigenspaces]
    \label{frob-char-for-r2t}
    Since the random-to-random shuffle is the symmetrization of the
    random-to-top shuffle, the kernels of the associated transition
    matrices coincide.
    The Frobenius characteristic of this eigenspace
    appears in \cite[Proposition~2.3]{ReinerWebb2004}; it is
    \begin{equation*}
        \smashoperator{\sum_{\substack{\text{desarrangement}\\\text{tableaux $T$ of size $n$}}}} \ s_{\shape(T)}
        \quad = \quad
        \sum_{\lambda \vdash n} \ \ndestab^\lambda \, s_\lambda.
    \end{equation*}

    Futhermore, Reiner and Wachs in unpublished work \cite{ReinerWachs2002}
    computed the Frobenius
    characteristic of \emph{all} the eigenspaces for the random-to-top shuffle:
    the Frobenius characteristic of the $\frac{j}{n}$-eigenspace is obtained by
    summing \autoref{eq:frobenius-characteristic-R2R} over partitions $\mu$ of size
    $j$:
    \begin{equation*}
        \smashoperator{\sum_{
            \substack{
                \text{horizontal strips} \\
                \lambda/\mu \text{ with } |\mu| = j
            }}} \ \ndestab^\mu s_\lambda.
    \end{equation*}
    A proof appears in \cite[Theorem~9.5]{RSW2014}.
\end{Remark}

\section{Proofs of \autoref{thm:lifting-eigenvectors} and \autoref{thm:eigenspace-decomposition}}
\label{sec:main-proofs}

The goal of this section is to prove \autoref{thm:lifting-eigenvectors} on lifting
eigenvectors of $\Rand_n$ to eigenvectors of $\Rand_{n+1}$; as well as
\autoref{thm:eigenspace-decomposition} on the decomposition of the Specht module
$S^\lambda$ into eigenspaces for $\Rand_n$.

Although the theory that we develop is valid for any ordered alphabet $A$, for
simplicity of notation and without loss of generality we will suppose that the
alphabet is $[n] = \{1, 2, 3, \dots,n \}$. However, in the examples we will
continue to use the usual lexicographically ordered alphabet $A = \{a, b, c,
\dots\}$ in order to distinguish words from permutations and coefficients.

\subsection{Outline of the proofs}
\label{outline-of-the-proofs}

We begin with a detailed outline of the proofs. The argument is divided into
several parts with each part in its own subsection.

\begin{enumerate}[wide, itemsep=0.5em]
    \item[\autoref{modules-of-words-and-reduction-to-specht-submodules}.]
        We begin by establishing the setting for the argument.
        We work in the algebra of words and interpret the shuffling
        operators as linear endomorphisms of this space.
        More specifically,
        we identify the random-to-random operator with an element of the group
        algebra of $\symm_n$ acting on the vector space $\MM[n]$ whose basis is
        the set of words of length $n$.

        We then reduce the problem of determining the spectrum and eigenspaces of
        the random-to-random shuffle acting on $\MM[n]$ to determining these on
        indecomposable modules for the symmetric group $\symm_n$.
        This reduction technique,
        sometimes called the \emph{Fourier transform},
        is constructive: there are explicit decompositions of $\MM[n]$ into
        indecomposable $\symm_n$--modules.
        Consequently, it suffices to determine the spectrum and eigenspaces of
        the random-to-random operator on all the Specht modules, since these
        are all the indecomposable $\symm_n$--modules up to isomorphism.

    \item[\autoref{an-identity-in-the-algebra-of-words}.]
        Next, we prove the following identity relating the random-to-random
        operators $\Rand_n$, the shuffling operators $\Int_a$, and the
        replacement operators $\Repl_{b,a}$:
        \begin{align*}
            \Rand_{n+1} \circ \Int_a - \Int_a \circ \Rand_n
            = (n+1) \Int_a + \sum_{1 \leq b \leq n} \Int_b \circ \Repl_{b,a}.
        \end{align*}
        These are linear transformations from $\MM[n]$ to $\MM[n+1]$.
        See \autoref{BracketRandInt}.

    \item[\autoref{action-of-the-symmetric-group}.]
        We restrict the above identity to the Specht submodules $S^\lambda$.
        The operators $\Repl_{b,a}$ are morphisms of $\symm_n$--modules
        that vanish on $S^\lambda$ for $b > a$ (\autoref{vanishingRab}),
        leading to the following identity
        (\autoref{SpechtRestritionOfBracketRandInt}):
        \begin{align*}
            \left(\Rand_{n+1} \circ \Int_a
                    - \Int_a \circ \Rand_n\right)\Big|_{S^\lambda}
            = (n+1) \Int_a\Big|_{S^\lambda} + \sum_{1 \leq b \leq a}
                \left(\Int_b \circ \Repl_{b,a}\right)\Big|_{S^\lambda}.
        \end{align*}

    \item[\autoref{projection-to-specht-module}.]
        The linear transformation above maps $S^\lambda$ to $M^{\lambda + \vec e_a}$.
        In \autoref{liftingeigenvectors2},
        we compose this with the projection $\isoproj_{\lambda + \vec e_r}$ to
        the submodule of $M^{\lambda + \vec e_a}$ that is isomorphic to
        a direct sum of copies of $S^{\lambda + \vec e_r}$. We will
        see that $\Rand_{n+1}$ commutes with $\isoproj{\lambda + \vec e_r}$
        (\cf the discussion preceding \autoref{eqn1}), and so we obtain
        \begin{align*}
            &
            \Rand_{n+1} \circ
                \left(\isoproj_{\lambda + \vec e_r} \circ \Int_a\right)\Big|_{S^\lambda}
                - \left(\isoproj_{\lambda + \vec e_r} \circ \Int_a\right) \circ \Rand_n \Big|_{S^\lambda}
            \\ & \hspace{1.75in}
            =
            \left((n+1) + (\lambda_r + 1) - r\right) \,
            \left(\isoproj_{\lambda + \vec e_r} \circ \Int_a\right) \Big|_{S^\lambda}.
        \end{align*}
        In particular, this proves that if $v \in S^\lambda$ is an eigenvector
        of $\Rand_n$ with eigenvalue $\varepsilon$, then
        $\isoproj_{\lambda + \vec e_r}(\Int_a(v))$ is either $0$
        or it is an eigenvector of $\Rand_{n+1}$ with eigenvalue
        $\varepsilon + (n+1) + (\lambda_r + 1) - r$.
        This proves \autoref{thm:lifting-eigenvectors}\autoref{thm:lifting-eigenvectors-part-a}.

    \item[\autoref{reformulation-of-the-projection-morphism}.]
        In \autoref{reformulation-of-morphisms} we prove that
        $\isoproj_{\lambda + \vec e_a} \circ \Int_a$ coincides
        with the linear transformations $\Lift_a^\lambda$ defined in
        \autoref{thm:lifting-eigenvectors}\autoref{thm:lifting-eigenvectors-part-b}:
        \begin{equation*}
            \Lift_a^\lambda =
            \sum_{1 \leq b_1 < \cdots < b_t < b_{t+1} = a}
            \left(
                \prod_{j=1}^{t}
                    \frac{1}{(\lambda_a - a) - (\lambda_{b_j} - {b_j})} \Repl_{b_j, b_{j+1}}
            \right)
            \Int_{b_1}.
        \end{equation*}
        This finishes the proof of \autoref{thm:lifting-eigenvectors}.

    \item[\autoref{image-of-r2r}.]
        Next, we prove that every eigenvector of
        $\Rand|_{S^{\lambda}}$, except those that lie in the kernel,
        is a linear combination of vectors in the image of some
        $\Rand|_{S^{\lambda - \vec e_a}}$
        under $\isoproj_{\lambda} \circ \Int_a$
        (\autoref{imageR2R}).

    \item[\autoref{construction-of-eigenspaces-from-kernels}.]
        Repeated application of the previous result implies that the
        eigenvectors of $\Rand|_{S^\lambda}$ are linear combinations of
        elements belonging to subspaces of the form
        \begin{align*}
            \Big(\isoproj_{\nu^{(s)}} \circ \Int_{b_s}\Big)
            \circ
            \cdots
            \circ
            \Big(\isoproj_{\nu^{(2)}} \circ \Int_{b_2}\Big)
            \circ
            \Big(\isoproj_{\nu^{(1)}} \circ \Int_{b_1}\Big)
            \left(\ker \Rand|_{S^\mu}\right).
        \end{align*}
        We study the above composition of maps
        (\autoref{suffices-to-project-at-end}) and show that $S^\lambda$ is the
        direct sum of these subspaces, one for each horizontal strip
        $\lambda/\mu$ (\autoref{directsumdecomposition}).
        This proves \autoref{thm:eigenspace-decomposition}.
\end{enumerate}

\subsection{Fourier transform reduction}
\label{modules-of-words-and-reduction-to-specht-submodules}

It turns out that in order to understand the spectrum of the random-to-random
shuffle, it suffices to understand the behaviour on certain subspaces.
This is sometimes called the \emph{Fourier transform reduction}.

In this section, we establish the setting for our argument.
We will work in the algebra of words and interpret the shuffling
operators as certain linear operators acting on this space.
We then explain why it suffices to reduce the study of these
operators to the study of their action on the Specht submodules.

\subsubsection{Algebra of words}
\label{sec:algebra-of-words}
Let $A$ be an alphabet; that is, $A = \{a_1, a_2, a_3, \dots\}$ is an ordered
finite set and its elements are called \emph{letters}.
View the letters of the alphabet $A$ as \emph{non-commuting} indeterminates and
let $\CC\langle A \rangle$ denote the algebra of polynomials in the variables
$A$ with coefficients in $\CC$.
Hence, $\CC\langle A \rangle$ is the $\CC$--vector space with basis the set of
finite words over $A$.
The sum and product of two polynomials is computed in the usual way, with the
product of two words $u$ and $v$ defined as the concatenation $uv$.
We call $\CC\langle A \rangle$ the \emph{algebra of words over $A$} (or
\emph{free associative algebra generated by $A$}).

There is a second operation on $\CC \langle A \rangle$ that we will use
frequently. It is called the \emph{shuffle product}.
Informally, it is the sum of all the ways of interleaving two words.
It can be defined recursively as follows: for words $w$ and $u$, and letters
$a$ and $b$,
\begin{equation*}
    w a \shuffle u b = (w \shuffle u b)a + (w a \shuffle u)b,
\end{equation*}
where the shuffle product of $w$ with the empty word is $w$.
This operation is commutative and associative.
Since $\Int_a(w)$ is the sum of all words obtained from $w$ by inserting $a$,
it follows that
\begin{equation*}
    \Int_a(w) = w \shuffle a.
\end{equation*}

\subsubsection{Shuffling operators as elements of the group algebra of the
    symmetric group}
\label{sec:operators-in-group-algebra}

The next observation is a straightforward re-interpretation of the shuffling
operators $\TopR_n$, $\RTop_n$ and $\Rand_n$ in terms of the action of
permutations on words. Specifically, it identifies these operators with linear
transformations associated with certain elements of the group algebra of the
symmetric group $\symm_n$.

Let $\MM[n]$ denote the subspace of $\CC\langle A \rangle$ spanned by words of
length $n$.
The symmetric group $\symm_n$ of degree $n$ acts on $\MM$ by permuting the
positions of the letters of a word. Explicitly, for a word $w_1 w_2 \cdots w_n$
and a permutation $\sigma \in \symm_n$,
\begin{gather*}
    w_1 w_2 \cdots w_n \cdot \sigma
    = w_{\sigma(1)} w_{\sigma(2)} \cdots w_{\sigma(n)}.
\end{gather*}

\begin{Example}
    \label{ex:permutation-action-on-words}
    If $w = cabcbaaa$ and $\sigma = 71842563, \tau = 26347158 \in \symm_8$,
    then
    \begin{align*}
        w \cdot \sigma = w_7 w_1 w_8 w_4 w_2 w_5 w_6 w_3
                       &= a   c   a   c   a   b   a   b
        \\
        (w \cdot \sigma) \cdot \tau = 
                                       a   c   a   c   a   b   a   b
                                 \cdot 2   6   3   4   7   1   5   8
                                   &=  c   b   a   c   a   a   a   b
        \\
        w \cdot (\sigma \tau) =
              c   a   b   c   b   a   a   a
        \cdot 1   5   8   4   6   7   2   3
          &=  c   b   a   c   a   a   a   b
    \end{align*}
    In particular, this example exemplifies that this is indeed a \emph{right}
    $\symm_n$--action.
\end{Example}

\begin{Definition}
    \label{random-to-random-in-group-algebra}
    Let $w$ be a word of length $n$.
    \begin{enumerate}
        \item
            $\RTop_n(w) = w \cdot \rho_n$, where
            $\displaystyle \rho_n = \mathid + \sum_{1\leq i < n} (i, i+1, \ldots, n)$.

        \item
            $\TopR_n(w) = w \cdot \tau_n$, where
            $\displaystyle \tau_n = \mathid + \sum_{1\leq i < n} (n, n-1, \ldots, i)$.

        \item
            $\Rand_n(w) = w \cdot \xi_n$, where
            \begin{align*}
                \xi_n
                    = \rho_n \tau_n
                    &=
                    n \cdot \mathid
                    \quad + \sum_{1\leq u < v \leq n} (u, u+1, \ldots, v-1, v) \\
                    & \hspace{1in} + \sum_{1 \leq v < u \leq n} (u, u-1, \ldots, v+1, v).
            \end{align*}
    \end{enumerate}
\end{Definition}

\subsubsection{Reduction to Specht modules}
\label{reduction-to-specht-modules}

Since $\TopR_n$, $\RTop_n$ and $\Rand_n$ are given by the action of an element
of the group algebra, any submodule of an $\symm_n$--module is stable for the
action of these operators.
Concretely, this means that understanding the action of these operators on
a $\symm_n$--module reduces to understanding their action on certain
submodules.

More precisely, suppose that an $\symm_n$--module $M$ decomposes into a direct
sum of submodules $M = N_1 \oplus \cdots \oplus N_r$.
Furthermore, suppose that each $N_i$ is \emph{indecomposable}:
that is, $N_i$ cannot be expressed as a direct sum of non-trivial submodules.

Let $\xi$ be an element of the group algebra of $\symm_n$.
Since each submodule $N_i$ is stable for the action of $\xi$, picking a basis
of $N_1$, a basis of $N_2$, and so on, gives a basis of $M$ in which the
matrices expressing the action of $\xi$ are all block diagonal: the blocks on
the diagonal are the matrices expressing the action of $\xi$ on the submodules
$N_i$.

Accordingly, the spectrum and eigenspaces of $\xi$ acting on $M$ are
completely determined by the spectra and eigenspaces of $\xi$ acting on the
indecomposable modules for $\symm_n$. These are the ``Specht modules''.
(The definition of the Specht modules together with an explicit decomposition
of the modules $\MM$ will be given in \autoref{action-of-the-symmetric-group}.)

\subsection{An identity in the algebra of words}
\label{an-identity-in-the-algebra-of-words}

The goal of this section is to prove the following identity that relates the
random-to-random operators $\Rand$, the shuffling operators $\Int$, and the
replacement operators $\Repl$:
\begin{align}
    \label{eq:main-identity}
    \Rand_{n+1} \circ \Int_a - \Int_a \circ \Rand_n
    = (n+1) \Int_a + \sum_{1 \leq b \leq n} \Int_b \circ \Repl_{b,a}.
\end{align}
This is an identity of linear transformations from the subspace spanned by
words of length $n$ to the subspace spanned by words of length $n+1$.

Let $\MM[n]$ denote the subspace of $\ZZ\langle A \rangle$ spanned by words of
length $n$.
Recall that $\Int_a^n: \MM \to \MM[n+1]$ denotes the linear operator that maps
a word to the sum of all words obtained by inserting the letter $a$;
explicitly,
\begin{gather*}
    \Int_a^{n}(w_1\cdots w_n)
    = \sum_{0\leq j \leq n} w_1 \cdots w_{j} \cdot \mathred{a} \cdot w_{j+1} \cdots w_n.
\end{gather*}
Let $\Ex_a^n : \MM \to \MM[n-1]$ denote the linear operator that maps a word to
the sum of all words obtained by removing exactly one occurrence of the letter
$a$:
\begin{align*}
    \Ex_a^{n}(w_1 \cdots w_n)
    = \sum_{\substack{1 \leq j \leq n \\[.5ex] \mathred{w_j = a}}} w_1 \cdots w_{j-1} w_{j+1} \cdots w_n.
\end{align*}
To simplify notation, we write $\Int_a$ and $\Ex_a$ instead of $\Int_a^n$ and
$\Ex_a^n$ when $n$ is determined by the context.
\begin{Example}
    \begin{align*}
        \Int_a(aaba) &= 3 \, aaaba + 2 \, aabaa, \\
        \Ex_a(aaba) &= 2 \, aba + aab, \\
        \Ex_b(aaba) &=      aaa. \qedhere
    \end{align*}
\end{Example}

The next result shows that top-to-random, random-to-top and random-to-random
can be described in terms of $\Int$ and $\Ex$.

\begin{Proposition}
    \label{r2r=t2r.r2t}
    \
    \begin{enumerate}
        \item
            \label{item1}
            For any word $w = w_1 \cdots w_n$ of length $n$,
            \begin{align*}
                \TopR_n(w) &= \Int_{w_n}^{n-1}(w_1 \cdots w_{n-1}).
            \end{align*}

        \item
            \label{item2}
            For any word $w$ of length $n$,
            \begin{align*}
                \RTop_n(w) &= \sum_{a \in A} \Ex_a^{n}(w) \cdot a,
            \end{align*}
            where $\cdot$ denotes concatenation of words.

        \item
            \label{item3}
            The random-to-random operator is the composition of the
            random-to-top operator with the top-to-random operator:
            \begin{gather*}
                \Rand_n = \sum_{a \in A} \Int_a^{n-1} \circ \Ex_a^{n} = \TopR_n \circ \RTop_n.
            \end{gather*}
    \end{enumerate}
\end{Proposition}

\begin{proof}
    The shuffling process underlying the $\TopR_n$ operator pops the last
    letter of $w$ and inserts it into the word $w_1 \cdots w_{n-1}$.
    This proves \eqref{item1}, and \eqref{item2} is proved similarly.

    The shuffling process underlying the $\Rand_n$ operator removes a letter
    $i$ from a word $w$, which coincides with $\Ex_i^{n}$, and then inserts
    $i$ into the word, which coincides with $\Int_i^{n-1}$. This proves the
    first equality of \eqref{item3}. For the second equality, we compute the composition:
    \begin{gather*}
        \left( \TopR_n \circ \RTop_n \right) (w)
        = \sum_{a \in A} \TopR_n \left( \Ex_a^{n}(w) \cdot a \right)
        = \sum_{a \in A} (\Int_a^{n-1} \circ \Ex_a^{n})(w)
        = \Rand_n(w). \qedhere
    \end{gather*}
\end{proof}

In light of \autoref{r2r=t2r.r2t},
it is opportune to understand the relationships between $\Int_i$ and $\Ex_j$.
It is straightforward to verify that $\Int_i$ and $\Ex_i$ are adjoint
operators with respect to the pairing $\langle \cdot, \cdot \rangle$ on $\MM$
that makes the basis of words an orthonormal basis; explicitly,
for $w$ a word of length $n$ and $u$ a word of length $n+1$,
\begin{gather}
    \label{Int-and-Ex-are-adjoint}
    \langle \Int_i^n(w), u \rangle = \langle w, \Ex_i^{n+1}(u) \rangle.
\end{gather}

Although these operators commute amongst themselves, that is,
\begin{align}
    \label{lemma:IntercalationsCommute}
    \Int_a^{n+1} \circ \Int_b^{n} &= \Int_b^{n+1} \circ \Int_a^{n} \\
    \Ex_b^{n} \circ \Ex_a^{n+1} &= \Ex_a^{n} \circ \Ex_b^{n+1},
\end{align}
it is not true that $\Int_a$ and $\Ex_b$ commute. However, we can describe
what happens using the replacement operators defined in \autoref{sssec:lifting
operators}; recall that
\begin{align*}
    \Repl_{i, j}(w) = \sum_{\substack{1\leq k\leq n \\ w_k = \mathred{i}}}
                w_1 \cdots w_{k-1} \cdot \mathblue{j} \cdot w_{k+1} \cdots w_n.
\end{align*}

\begin{Lemma}
    \label{InExDiff}
    \label{ReplacementIntCommutation}
    For all $a, b \in [n]$,
    \begin{gather}
        \Ex_{b}^{n+1} \circ \Int_{a}^{n}
        - \Int_{a}^{n-1} \circ \Ex_{b}^{n}
        = \Repl_{b,a}^n + \KroneckerDelta_{a,b} (n+1) \mathid
        \label{InExDiff1}
        \\
        \Repl_{a,b}^{n+1} \circ \Int_a^{n} - \Int_a^{n} \circ \Repl_{a,b}^{n} = \Int_b^{n}
        \label{InExDiff2}
        \\
        \Ex_a^{n+1} \circ \Repl_{b,a}^{n+1} - \Repl_{b,a}^n \circ \Ex_a^{n+1} = \Ex_b^{n+1}
        \label{InExDiff3}
    \end{gather}
    where $\KroneckerDelta_{a,b}$ denotes the Kronecker delta function,
    which is $1$ if $a = b$ and $0$ if $a \neq b$.
\end{Lemma}

\begin{proof}
    \autoref{InExDiff1}.
    We construct a bijection between the words appearing in
    $\left(\Ex_{b}^{n+1} \circ \Int_{a}^{n}\right)(w)$
    and the words appearing in
    $(\Int_{a}^{n-1} \circ \Ex_{b}^{n})(w) + \Repl_{b,a}^n(w) + \KroneckerDelta_{a,b} (n+1) w$.

    Every word $u$ appearing in
    $\left(\Ex_{b}^{n+1} \circ \Int_{a}^{n}\right)(w)$
    corresponds to a pair $(i,j)$ such that $u$ is obtained from $w$ by
    inserting $a$ at position $i$
    and then removing the $j$-th letter provided that it is $b$;
    whereas every word $u$ appearing in
    $\left(\Int_{a}^{n-1} \circ \Ex_{b}^n\right)(w)$
    corresponds to a pair $(i',j')$ such that $u$ is obtained from $w$ by
    removing $w_{j'}$, provided that $w_{j'} = b$,
    and then inserting $a$ at position $i'$.

    The case $i = j$ can only occur if $a = b$ since it corresponds to
    inserting $a$ at position $i$ only to then remove it.
    So when $a = b$, we get $n+1$ copies of $w$.
    This accounts for the terms $\KroneckerDelta_{a,b} (n+1) w$.

    If $i < j$, then
    \begin{gather*}
        u = w_1 \cdots w_{i-1} \,{a}\, w_{i} \cdots w_{j-2} \, w_{j} \cdots w_n
        \qquad\text{and}\qquad
        w_{j-1} = b.
    \end{gather*}
    This word appears in
    $\left(\Int_{a}^{n-1} \circ \Ex_{b}^n\right)(w)$
    since it can be obtained from $w$ by removing $w_{j-1} = b$
    and inserting $a$ at position $i$.
    This corresponds to the pair $(i', j') = (i, j-1)$ of
    $\left(\Int_{a}^{n-1} \circ \Ex_{b}^n\right)(w)$.
    Hence, the pairs $(i,j)$ with $i < j$
    correspond to the pairs $(i', j')$ with $i' \leq j'$.

    If $i > j$, then
    \begin{gather*}
        u = w_1 \cdots w_{j-1} \, w_{j+1} \cdots w_{i-1} \,{a}\, w_{i} \cdots w_n
        \qquad\text{and}\qquad
        w_j = b.
    \end{gather*}
    This word appears in
    $\left(\Int_{a}^{n-1} \circ \Ex_{b}^n\right)(w)$
    since it can be obtained from $w$ by removing $w_j = b$
    and inserting $a$ at position $i-1$.
    This corresponds to the pair $(i', j') = (i-1, j)$ of
    $\left(\Int_{a}^{n-1} \circ \Ex_{b}^n\right)(w)$.
    Hence, the pairs $(i,j)$ with $i > j$
    correspond to the pairs $(i', j')$ with $i' \geq j'$.

    Consequently, we obtain all the pairs $(i', j')$ corresponding to the
    terms appearing in $\left(\Int_{a}^{n-1} \circ \Ex_{b}^n\right)(w)$
    as well as a second copy of the pairs $(i', j')$ with $i' = j'$.
    The pairs $(i', i')$ correspond to removing
    $w_{i'} = b$ and inserting $a$ into position $i'$;
    that is, they correspond to the terms of $\Repl_{b,a}(w)$.

    \autoref{InExDiff2}.
    If $b = a$, then $\Repl_{a,b}^{n}(w) = \mult_a(w)\, w$,
    where $\mult_a(w)$ denotes the number of occurrences of $a$ in $w$.
    Hence,
    \begin{equation*}
        \begin{multlined}
            \left(\Repl_{a,a}^{n+1} \circ \Int_a^{n} - \Int_a^{n} \circ \Repl_{a,a}^{n}\right)(w)
            \\
            = (\mult_a(w) + 1) \Int_a^{n}(w) - \mult_a(w) \Int_a^{n}(w) 
            = \Int_a^{n}(w).
        \end{multlined}
    \end{equation*}

    If $b \neq a$, then use \autoref{InExDiff1} 
    and \autoref{lemma:IntercalationsCommute} to write:
    \begin{align*}
        \Repl_{a,b}^{n+1} \circ \Int_a^{n}
        &=
        \Ex_{a}^{n+2} \circ \Int_{b}^{n+1} \circ \Int_a^{n}
        - \Int_{b}^{n} \circ \Ex_{a}^{n+1} \circ \Int_a^{n}
        \\
        &=
        \Ex_{a}^{n+2} \circ \Int_{a}^{n+1} \circ \Int_b^{n}
        - \Int_{b}^{n} \circ \Ex_{a}^{n+1} \circ \Int_a^{n},
        \\
        \Int_a^{n} \circ \Repl_{a,b}^{n}
        &=
        \Int_a^{n} \circ \Ex_{a}^{n+1} \circ \Int_{b}^{n}
        - \Int_a^{n} \circ \Int_{b}^{n-1} \circ \Ex_{a}^{n}
        \\
        &=
        \Int_a^{n} \circ \Ex_{a}^{n+1} \circ \Int_{b}^{n}
        - \Int_b^{n} \circ \Int_a^{n-1} \circ \Ex_{a}^{n}.
    \end{align*}
    The difference is then
    \begin{align*}
        &
        \Repl_{a,b}^{n+1} \circ \Int_a^{n}
        -
        \Int_a^{n} \circ \Repl_{a,b}^{n}
        \\
        &=
        \left(\Ex_{a}^{n+2} \circ \Int_{a}^{n+1}
        - \Int_a^{n} \circ \Ex_{a}^{n+1} \right) \circ \Int_{b}^{n}
        -
        \Int_{b}^{n} \circ \left(\Ex_{a}^{n+1} \circ \Int_a^{n}
        - \Int_a^{n-1} \circ \Ex_{a}^{n}\right)
        \\
        &=
        \left(\Repl_{a,a}^{n+1} + (n+2) \mathid\right) \circ \Int_b^{n}
        - \Int_b^{n} \circ \left(\Repl_{a,a}^{n} + (n+1) \mathid\right)
        \\
        &=
        \Int_b^{n} -
        \left(\Repl_{a,a}^{n+1} \circ \Int_b^{n}
        - \Int_b^{n} \circ \Repl_{a,a}^{n} \right)
        =
        \Int_b^{n}
    \end{align*}
    where the last equality follows from
    the fact that $\Repl_{a,a}(w) = \mult_a(w)\, w$:
    \begin{equation*}
        \begin{multlined}
            \left(\Repl_{a,a}^{n+1} \circ \Int_b^{n}
            - \Int_b^{n} \circ \Repl_{a,a}^{n} \right)(w)
            \\
            = \mult_a(w) \Int_b^{n}(w) - \Int_b^{n} (\mult_a(w)\, w) = 0.
        \end{multlined}
    \end{equation*}

    \autoref{InExDiff3}.
    Use \autoref{InExDiff2} and the fact that
    the adjoint of $\Int_a^n$ is $\Ex_a^{n+1}$ and
    the adjoint of $\Repl_{a,b}^{n}$ is $\Repl_{b,a}^{n}$.
\end{proof}

\begin{Remark}
    \label{rem:commutativity-of-RSW-operators}
    In \cite{RSW2014}, Reiner, Saliola and Welker introduced and studied
    a family of operators $\nu_{(n-k,1^k)}$, for $1 \leq k < n$, that
    includes the random-to-random operator.
    Using the notation of this paper, they admit the following description:
    \begin{gather*}
        \nu_{(n-k, 1^k)}
        =
        \sum_{i_1 < i_2 < \dots < i_k}
        \Int_{i_1} \circ \Int_{i_2} \circ \cdots \circ \Int_{i_k}
        \circ
        \Ex_{i_k} \circ \cdots \circ \Ex_{i_2} \circ \Ex_{i_1}.
    \end{gather*}
    Setting $k=1$ recovers the random-to-random operator (\autoref{r2r=t2r.r2t}).

    In \cite[Theorem~1.1]{RSW2014}, it is proved that the operators from this
    family pairwise commute.
    The results in \autoref{InExDiff}, which describe the relationship between
    $\Int_a \circ \Ex_b$ and $\Ex_b \circ \Int_a$, can be used to give another
    proof of this fact. However, this proof is rather involved and not very
    conceptual, so we leave the details to the interested reader.
\end{Remark}

Equipped with these results, we can now prove \autoref{eq:main-identity}
as well as a useful reformulation.

\begin{Theorem}
    \label{BracketRandInt}
    For all $a \in [n]$,
    \begin{align}
        \label{eq:BracketRandInt1}
        \Rand_{n+1} \circ \Int_a^{n} - \Int_a^{n} \circ \Rand_n
        &= (n+1) \Int_a^{n} + \sum_{1 \leq b \leq n} \Int_b^{n} \circ \Repl_{b,a}^{n}
        \\
        \label{eq:BracketRandInt2}
        &= \Int_a^{n}
        + \sum_{1 \leq b \leq n} \Repl_{b,a}^{n+1} \circ \Int_b^{n}.
    \end{align}
\end{Theorem}

\begin{proof}
    Since $\Rand_n = \sum_{b \leq n} \Int_b \circ \Ex_b$,
    and $\Ex_{n+1} \circ \Int_a = 0$ (because $a \neq n + 1$),
    \begin{equation*}
        \begin{aligned}
        & \Rand_{n+1} \circ \Int_a^{n} - \Int_a^{n} \circ \Rand_{n}
        \\
        & = \left( \sum_{1 \leq b \leq n+1}
            \Int_b^{n} \circ \Ex_b^{n+1} \right) \circ \Int_a^{n}
            - \Int_a^{n} \circ 
            \left(\sum_{1 \leq b \leq n} \Int_b^{n-1} \circ \Ex_b^{n}\right)
        \\
        & =
            \sum_{1 \leq b \leq n}
             \left(
                \Int_b^{n} \circ \Ex_b^{n+1} \circ \Int_a^{n} -
                \Int_a^{n} \circ \Int_b^{n-1} \circ \Ex_b^{n}
            \right).
        \end{aligned}
    \end{equation*}
    From \autoref{InExDiff} and
    $\Int_b^{n} \circ \Int_a^{n-1} = \Int_a^{n} \circ \Int_b^{n-1}$
    (\autoref{lemma:IntercalationsCommute}), it follows that
    \begin{align*}
        \Int_b^{n} \circ \Ex_b^{n+1} \circ \Int_a^{n}
        - \Int_a^{n} \circ \Int_b^{n-1} \circ \Ex_b^{n}
        = \Int_b^{n} \circ \Repl_{b,a}^{n} + \KroneckerDelta_{a,b} (n+1) \Int_b^{n}.
    \end{align*}
    This establishes the first equality.
    For the second equality:
    \begin{align*}
        \Rand_{n+1} \circ \Int_a^{n} - \Int_a^{n} \circ \Rand_n
        &= (n+1) \Int_a^{n} + \sum_{1 \leq b \leq n} \Int_b^{n} \circ \Repl_{b,a}^{n}
        \\
        &= (n+1) \Int_a^{n}
        + \sum_{1 \leq b \leq n}
        \left(
            \Repl_{b,a}^{n+1} \circ \Int_b^{n} - \Int_a^{n}
        \right)
        \\
        &= \Int_a^{n}
        + \sum_{1 \leq b \leq n} \Repl_{b,a}^{n+1} \circ \Int_b^{n}.
        \qedhere
    \end{align*}
\end{proof}

\subsection{Restriction to Specht modules $S^\lambda$}
\label{action-of-the-symmetric-group}

\autoref{BracketRandInt} is an identity of linear transformations from $M^\lambda$
to $M^{\lambda + \vec e_i}$.
In this section, we restrict this identity to the Specht submodule $S^\lambda$
of $M^\lambda$.
We begin by recalling the definition of $S^\lambda$ as well as the
decomposition of $M^\lambda$ into Specht modules, which
leads to a useful simplification of the identity in \autoref{BracketRandInt}.

\subsubsection{Specht modules}

For a sequence $\alpha$ of non-negative integers that sum to $n$, denote by
$M^\alpha$ the subspace of $\MM$ spanned by the words of evaluation $\alpha$.
Each subspace $M^\alpha$ is invariant under the $\symm_n$--action, and
so it is an $\symm_n$--module. Note that $M^\alpha \cong M^\beta$ as
$\symm_n$--modules if and only if $\alpha$ is a rearrangement of $\beta$.

Each $M^\alpha$ contains a distinguished submodule $S^\alpha$
called the \emph{Specht submodule}. We recall its construction for the case of
a partition $\lambda$; for the general case, use that $M^\alpha \cong
M^{\decreasingrearrangement{\alpha}}$, where $\decreasingrearrangement{\alpha}$
is the rearrangement of $\alpha$ into a weakly decreasing sequence.

Let $\lambda$ be a partition of $n$ and let $t$ be
a tableau of shape $\lambda$.
Let $\word(t)$ denote the word in which the $i$-th letter is $r$, where $r$ is
the number of the row of $t$ that contains $i$.
Note that the evaluation of $\word(t)$ is $\lambda$.

Next, let $N_t$ denote the following element of the group algebra
of $\symm_n$
\begin{gather*}
    N_t = \sum_{\sigma \in \ColStab(t)} \sign(\sigma) \sigma,
\end{gather*}
and let
\begin{gather*}
    \w_t = \word(t) \cdot N_t
        = \sum_{\sigma \in \ColStab(t)} \sign(\sigma) \word(\sigma(t)),
\end{gather*}
where the sums range over all the permutations $\sigma \in \symm_n$ that
permute the entries in the columns of $t$ in all possible ways.

The \emph{Specht submodule} $S^\lambda$ of $M^\lambda$ is the subspace
\begin{gather*}
    S^\lambda = \spn\left\{ \w_t : t \in \SYT_\lambda\right\},
\end{gather*}
where $\SYT_\lambda$ denotes the set of standard (Young) tableaux of shape
$\lambda$.

\begin{Example}
    \label{ex:tableau2word}
    If $t$ and $s$ are the following tableaux, respectively,
    \begin{gather*}
        t = \tikztableausmall{{2,8,6,7},{5,3},{1,4}}
        \qquad
        \qquad
        s = \tikztableausmall{{5,3,7,1},{6,8},{2,4}}
    \end{gather*}
    then $\word(t) = 31232111$ and $\word(s) = 13131212$.
    These are the words $w$ and $w \cdot \sigma$ of
    \autoref{ex:permutation-action-on-words} (with $a=1$, $b=2$, $c=3$),
    which illustrates the identity
    $\word(\sigma\inv(t)) = \word(t) \cdot \sigma$ since $t = \sigma(s)$.

    If $\symm_{I}$ denotes the subgroup of permutations of the elements of the
    set $I$, then
    \begin{align*}
        \ColStab(t) &= \symm_{\{2,5,1\}} \times \symm_{\{8,3,4\}} \times \symm_{\{6\}} \times \symm_{\{7\}},
        \\
        \ColStab(s) &= \symm_{\{5,6,2\}} \times \symm_{\{3,8,4\}} \times \symm_{\{7\}} \times \symm_{\{1\}}.
    \qedhere
    \end{align*}
\end{Example}


\begin{Remark}
    \label{remark:specht-modules-via-polytabloids}
    A common way to define the Specht modules is as the span of
    \emph{polytabloids} $e_t = N_t \cdot \{t\}$, where $\{t\}$ is the
    \emph{tabloid} corresponding to the tableau $t$ \cite{James1978,Sagan2001}.
    Our definition encodes a tabloid $\{t\}$ as $\word(t)$. This encoding is an
    isomorphism of $\symm_n$--modules as the $\symm_n$--action on words
    corresponds to the usual $\symm_n$--action on tabloids:
    $\word(t) \cdot \sigma = \word(\sigma\inv(t))$
    (\cf \autoref{ex:tableau2word}).
\end{Remark}

\subsubsection{Decomposition of $M^\lambda$}
\label{sss:Mdecomposition}

Up to isomorphism, the Specht modules $S^\lambda$, for $\lambda \vdash n$, are
all the irreducible $\symm_n$--modules.
Consequently, $M^\lambda$ decomposes into a direct sum of submodules isomorphic
to a Specht module.
The decomposition of the modules $M^\lambda$ into irreducible
$\symm_n$--modules is known as Young's Rule \cite[Theorem~2.11.2]{Sagan2001}, which
we describe next.

For each tableau $T$ of size $n$,
there is an $\symm_n$--module morphism
$$\Theta_T : M^{\shape(T)} \to M^{\evaluation(T)}$$
that embeds the Specht module $S^{\shape(T)}$ into $M^{\evaluation(T)}$.
Explicitly, $\Theta_T$ is the $\symm_n$--module morphism defined
on the increasing word with evaluation equal to $\shape(T)$
as the sum of all the words that can be obtained from $T$ by permuting the
entries in the rows and then concatenating the rows.
\begin{Example}
    Consider the morphism
    \begin{gather*}
        \Theta_{\!\!\!\tikztableauscript{{2,1,1},{3,2},}} : M^{(3,2)} \to M^{(2,2,1)}.
    \end{gather*}
    There are six words that can be obtained by permuting the elements in the
    rows of the tableau, thus
    \begin{gather*}
        \Theta_{\!\!\!\tikztableauscript{
                {$\mathred2$,$\mathred1$,$\mathred1$},
                {$\mathblue3$,$\mathblue2$},}}\!\!
                (11122) \quad=\quad
        \mathred{211}\mathblue{32} + \mathred{121}\mathblue{32} +
        \mathred{112}\mathblue{32} + \mathred{211}\mathblue{23} +
        \mathred{121}\mathblue{23} + \mathred{112}\mathblue{23}.
    \end{gather*}
    Since this is a $\symm_5$--morphism,
    the image of any other word $11122 \cdot \sigma$ in $M^{(3,2)}$
    is
    \begin{gather*}
        \Theta_{\!\!\!\tikztableauscript{
                {$\mathred2$,$\mathred1$,$\mathred1$},
                {$\mathblue3$,$\mathblue2$},}}\!\!
                (11122 \cdot \sigma)
        \quad=\quad
        \Theta_{\!\!\!\tikztableauscript{
                {$\mathred2$,$\mathred1$,$\mathred1$},
                {$\mathblue3$,$\mathblue2$},}}\!\!
                (11122)
                \cdot \sigma.
        \qedhere
    \end{gather*}
\end{Example}

\begin{Proposition}[{\cite[Theorem~2.11.2]{Sagan2001}}]
    \label{Mdecomposition}
    Let $\SSYT_n$ denote the set of semistandard tableaux of $n$
    and $K_{\lambda, \mu}$ the set of semistandard tableaux of shape
    $\lambda$ and evaluation $\mu$. Then
    \begin{gather*}
        M^\mu = \bigoplus_{\substack{T \in \SSYT_n \\ \evaluation(T)=\mu}} \Theta_T\left(S^{\shape(T)}\right)
        \, \cong \,\,  \bigoplus_{\lambda \dominateseq \mu} K_{\lambda,\mu} \, S^\lambda
    \end{gather*}
    where $\dominateseq$ denotes dominance order on partitions.
    In particular, each $M^\alpha$ contains exactly one copy of $S^\alpha$.
\end{Proposition}

\subsubsection{Restriction of \autoref{eq:BracketRandInt1} to the Specht module $S^\lambda$}

Let $\lambda$ be a partition of $n$. The linear transformations $\Repl_{b,a}$
are $\symm_n$--module morphisms. Specifically, they are special cases of the
morphisms $\Theta_T$ defined above.
\begin{Example}
    \begin{gather*}
        \Repl_{1,2}(1112) =
        2112 + 1212 + 1122 =
            \Theta_{\!\!\!\tikztableauscript{
                    {$\mathred1$,$\mathred1$,$\mathred2$},
                    {$\mathblue2$},}}\!\!(1112).
            \qedhere
    \end{gather*}
\end{Example}
Consequently, the restriction of $\Repl_{b,a}$ to $S^\lambda$ will either be
the trivial morphism or an isomorphism.

\begin{Lemma}
    \label{vanishingRab}
    Let $\lambda \vdash n$ and $\alpha = \lambda - \vec e_b + \vec e_a$.
    If $\Repl_{b,a}^{n}$ is nonzero on $S^\lambda$, then
    $\lambda$ dominates the non-increasing rearrangement of $\alpha$.
    Consequently, if $a < b$, then $\Repl_{b,a}^{n}(S^\lambda) = 0$.
\end{Lemma}

\begin{proof}
    Suppose $\Repl_{b,a}^n$ is nonzero on $S^\lambda$.
    The map $\Repl_{b,a}^n$ is an $\symm_n$--module morphism from $M^\lambda$ to
    $M^{\alpha}$, and so the restriction of $\Repl_{b,a}^n$ to $S^\lambda$ is an
    $\symm_n$--module morphism that maps $S^\lambda$ into $M^{\alpha}$.
    Since $\Repl_{b,a}^n|_{S^\lambda} \neq 0$, by Schur's lemma it is an
    isomorphism onto its image, and so $M^{\alpha}$ contains a submodule isomorphic to
    $S^\lambda$.
    \begin{description}
        \item[Case 1]
            Suppose $\alpha$ is a partition.
            Then, by Young's rule, $M^{\alpha}$ contains a submodule
            isomorphic to $S^\lambda$ if and only if $\lambda$ dominates
            $\alpha = \lambda - \vec e_b + \vec e_a$,
            which happens if and only if $b \leq a$.

        \item[Case 2]
            Suppose $\alpha$ is not a partition.
            By rearranging the entries of $\alpha$
            we obtain a partition $\mu$.
            As $\symm_n$--modules, $M^\mu \cong M^\alpha$
            and so by Young's rule $\lambda$ dominates $\mu$.
            If $a < b$, then
            \begin{align*}
                \begin{array}{cccccccc}
                    \lambda = (\lambda_1 & \lambda_2 & \cdots & \lambda_{a-1} & \lambda_a & \cdots & \lambda_b & \cdots) \\
                    \alpha  = (\lambda_1 & \lambda_2 & \cdots & \lambda_{a-1} & \lambda_a + 1 & \cdots & \lambda_b - 1 & \cdots)
                \end{array}
            \end{align*}
            which would imply that $\lambda$ does not dominate $\mu$, a contradiction.
            \qedhere
    \end{description}
\end{proof}

\begin{Corollary}
    \label{SpechtRestritionOfBracketRandInt}
    Let $\lambda \vdash n$ and $a \in [n]$. Then,
    \begin{align*}
        \left(\Rand_{n+1} \circ \Int_a
                - \Int_a \circ \Rand_n\right)\Big|_{S^\lambda}
        = (n+1) \Int_a\Big|_{S^\lambda} + \sum_{1 \leq b \leq a}
            \left(\Int_b \circ \Repl_{b,a}\right)\Big|_{S^\lambda}.
    \end{align*}
\end{Corollary}

\subsection{Projection to the Specht module $S^{\lambda + \vec e_a}$}
\label{projection-to-specht-module}

The restriction to $S^\lambda$ of the operator
$\Rand_{n+1} \circ \Int_a - \Int_a \circ \Rand_n$
defines a linear transformation from $S^\lambda$
to $M^{\lambda + \vec e_a}$.
We compose this with the canonical
projection onto $S^{\lambda + \vec e_a}$.
We begin by reviewing the definition and properties
of these projectors and culminate with a proof
of \autoref{thm:lifting-eigenvectors}\autoref{thm:lifting-eigenvectors-part-a}
(see \autoref{liftingeigenvectors2}).

\subsubsection{Isotypic projectors}
\label{sssec:isotypic-projectors}

Next we want to consider the projection of
\begin{gather*}
    (\Rand_{n+1} \circ \Int_a - \Int_a \circ \Rand_n)(S^\lambda)
\end{gather*}
onto the $S^\mu$--isotypic component of $M^\mu$, where $\mu = \lambda + \vec
e_a$.
We begin by recalling the definition of the isotypic projector associated with
a simple module.

Let $G$ be any finite group.
For any simple $G$--module $W$ with character denoted by $\chi_W$,
let $\mathfrak p_W$ denote the following element of the group algebra $\mb C G$
of $G$:
\begin{gather*}
    \mathfrak p_W = \frac{\dim(W)}{|G|} \sum_{g \in G} \overline{\chi_W(g)} \, g.
\end{gather*}
This element defines a projection from any $G$--module $V$ onto its
$W$--isotypic component:
\begin{eqnarray*}
    \isoproj_W : V & \longrightarrow & V \\
    v & \longmapsto & v \cdot \mathfrak p_W.
\end{eqnarray*}
To see that $\isoproj_W$ is a morphism of $G$--modules,
first note that $\mathfrak p_W$ lies in the center of the
group algebra $\mb CG$:
for any $h \in G$,
\begin{align}
    \label{isoprojIsCentral}
    \mathfrak p_W \, h
    = \frac{\dim(W)}{|G|} \sum_{g \in G} \overline{\chi_W(g)} \, gh
    = \frac{\dim(W)}{|G|} \sum_{h g' h\inv \in G} \overline{\chi_W(h g' h\inv)} \, hg'
    = h \, \mathfrak p_W(v),
\end{align}
where the last equality follows from the fact that $\chi_W$ is constant on
conjugacy classes.
Then $\isoproj_W$ is a morphism of $G$--modules, since for any $g \in G$,
\begin{equation}
    \label{eq:ProjectionsAreMorphisms}
    \isoproj_W(v \cdot g)
    = v \cdot (g \mathfrak p_W)
    = v \cdot (\mathfrak p_W g)
    = \isoproj_W(v) \cdot g.
\end{equation}
Moreover, $\isoproj_W$ commutes with any $G$--module endomorphism $\varphi$ of
$V$, since
\begin{equation}
    \label{eq:ProjectionsCommuteWithMorphisms}
    \isoproj_W(\varphi(v))
    = \varphi(v) \cdot \mathfrak p_W
    = \varphi(v \cdot \mathfrak p_W)
    = \varphi(\isoproj_W(v)).
\end{equation}

\subsubsection{Projection onto $S^{\lambda + \vec e_a}$}

Although $\Int_a^{n}(S^\lambda)$ is not a submodule of
$M^{\lambda + \vec e_a}$,
we can show that it is contained in a multiplicity-free
submodule of $M^{\lambda + \vec e_a}$.

\begin{Lemma}
    \label{l:SubmoduleGeneratedByLifts}
    The subspace $\Int_a^{n}(S^\lambda)$ is contained in
    a $\symm_{n+1}$--submodule of $M^{\lambda + \vec e_a}$ that is isomorphic
    to $\bigoplus_\mu S^\mu$, where $\mu$ ranges over the partitions obtained
    from $\lambda$ by adding a cell in row $r$ with $r \leq a$.
\end{Lemma}

\begin{proof}
    Let $w$ be a word of length $n$. Then
    \begin{gather*}
        \Int_a^{n}(w) = w \shuffle a,
    \end{gather*}
    where $\shuffle$ denotes the \emph{shuffle product} of words
    defined in \autoref{sec:algebra-of-words}.
    Moreover, if $b$ is a letter that does not occur in $w$, then
    \begin{gather*}
        \Int_a^{n}(w) = \Repl_{b,a}(w \shuffle b).
    \end{gather*}

    Let $b = \ell(\lambda) + 1$ (so that $b$ is a letter not occuring in any
    word $w \in M^\lambda$).
    Then
    \begin{gather*}
        S^\lambda \shuffle b = \left\{ x \shuffle b : x \in S^\lambda \right\}
    \end{gather*}
    is contained in the $\symm_{n+1}$--submodule of $M^{\lambda + \vec e_b}$
    generated by $x \cdot b$ with $x \in S^\lambda$:
    \begin{gather*}
        N = \langle \left\{ x \cdot b : x \in S^\lambda \right\} \rangle.
    \end{gather*}
    The submodule $N$ is isomorphic to the induced module
    $\Ind_{\symm_n \times \symm_1}^{\symm_{n+1}}(S^\lambda \otimes S^1)$
    (see \cite{James1978}).
    Then, by the branching rule for representations of the symmetric group, $N$
    decomposes as the multiplicity-free direct sum of the Specht modules
    $S^\mu$, where $\mu$ is a partition of $n+1$ that contains $\lambda$.

    Since $\Int_a^n(w) = \Repl_{b,a}(w \shuffle b)$,
    it follows that
    \begin{gather*}
        \Int_a^n(S^\lambda)
        \subseteq
        \Repl_{b,a}\left( N \right)
        \cong
        \bigoplus_{\mu \supseteq \lambda} \Repl_{b,a}\left( S^\mu \right).
    \end{gather*}
    Since $\Repl_{b,a}$ is an $\symm_{n+1}$--module morphism,
    $\Repl_{b,a}(S^\mu)$ is either $0$ or isomorphic to $S^\mu$.
    Since these are submodules of
    $\Repl_{b,a}(M^{\lambda + \vec e_b}) \subseteq M^{\lambda + \vec e_a}$,
    it follows that $\mu$ dominates $\lambda + \vec e_a$.
    So $\mu$ is obtained from $\lambda$ by adding a cell in row $r$ with $r
    \leq a$.
\end{proof}

\begin{Definition}
    \label{d:projlift}
    For partitions $\lambda \vdash n$, $\mu \vdash n+1$
    and $a \in \{1, 2, \dots, \ell(\lambda) + 1\}$,
    let $\projlift[a]$ denote the restriction of
    $\isoproj_\mu \circ \Int_a : M^\lambda \to M^{\lambda + \vec e_a}$ to $S^\lambda$:
    \begin{align}
        \label{e:projlift}
        \projlift[a] = \left. \left(\isoproj_\mu \circ \Int_a\right) \right|_{S^\lambda}
        : S^\lambda \to M^{\lambda + \vec e_a}.
    \end{align}
\end{Definition}

Note that the image of $\projlift[a]$ is contained in the $S^\mu$--isotypic
component of $M^{\lambda + \vec e_a}$.
Moreover, by \autoref{l:SubmoduleGeneratedByLifts},
the subspace $\Int_a(S^\lambda)$ is contained in a submodule of
$M^{\lambda + \vec e_a}$ isomorphic to a direct sum of
Specht modules $S^\mu$, where $\mu$ ranges over the partitions obtained from
$\lambda$ by adding a cell in row $r$ with $r \leq a$.
Hence, if $\mu \neq \lambda + \vec e_r$ for all $1 \leq r \leq a$,
then $\projlift[a] = 0$.

%

\begin{Lemma}
    \label{MainLemma}
    For $\lambda = (\lambda_1, \lambda_2, \dots, \lambda_l) \vdash n$,
    $a \in \{1, 2, \dots, l + 1\}$, and
    $\mu = \lambda + \vec e_r$ with $1 \leq r \leq a$,
    \begin{equation}
        \label{eq:MainLemma}
        \begin{multlined}
            \Rand_{n+1} \circ \projlift - \projlift \circ \Rand_n
            \\
            \shoveleft[.5in] {
                = \left((n + 1) + (\lambda_a + 1) - a\right) \projlift
                + \sum_{r \leq b < a} \Repl_{b,a} \circ \projlift[b].
            }
        \end{multlined}
    \end{equation}
\end{Lemma}

\begin{proof}
    By \autoref{SpechtRestritionOfBracketRandInt},
    \begin{align*}
        \left(\Rand_{n+1} \circ \Int_a
                - \Int_a \circ \Rand_n\right)\Big|_{S^\lambda}
        = (n+1) \Int_a\Big|_{S^\lambda} + \sum_{1 \leq b \leq a}
            \left(\Int_b \circ \Repl_{b,a}\right)\Big|_{S^\lambda}.
    \end{align*}
    Since $\Rand_{n+1}$ is given by the action of an element of the group
    algebra of $\symm_{n+1}$ (\autoref{random-to-random-in-group-algebra})
    and $\isoproj_\mu$ is an $\symm_{n+1}$--module morphism,
    $\isoproj_\mu$ and $\Rand_{n+1}$ commute, and so
    \begin{equation}
        \label{eqn1}
        \begin{multlined}
            \Rand_{n+1} \circ \projlift
            - \projlift \circ \Rand_n
            \\
            \shoveleft[.5in]
            = \left(n + 1\right) \projlift
            +
            \sum_{1 \leq b \leq a} \left(\isoproj_\mu \circ \Int_b
                    \circ \Repl_{b,a}\right) \Big|_{S^\lambda}.
        \end{multlined}
    \end{equation}
    By \autoref{ReplacementIntCommutation},
    $\Int_b \circ \Repl_{b,a} = \Repl_{b,a} \circ \Int_b - \Int_a$,
    and so
    \begin{align*}
        \isoproj_\mu \circ \Int_b \circ \Repl_{b,a}
        &= \isoproj_\mu \circ \Repl_{b,a} \circ \Int_b - \isoproj_\mu \circ \Int_a \\
        &= \Repl_{b,a} \circ \isoproj_\mu \circ \Int_b - \isoproj_\mu \circ \Int_a,
    \end{align*}
    where the last equality follows from the fact that
    $\isoproj_\mu$ commutes with $\symm_{n+1}$--module morphisms
    (see \autoref{eq:ProjectionsCommuteWithMorphisms}).
    Thus, the right hand side of \autoref{eqn1} becomes
    \begin{equation}
        \label{eqn2}
        \begin{multlined}
            \left((n + 1) - a\right) \projlift
            +
            \sum_{1 \leq b \leq a} \Repl_{b,a} \circ \projlift[b].
        \end{multlined}
    \end{equation}

    Next, we prove that $\Repl_{b,a} \circ \projlift[b] = 0$ for $b < r$.
    By \autoref{l:SubmoduleGeneratedByLifts},
    $\Int_b(S^\lambda)$ is contained
    in a submodule of $M^{\lambda + \vec e_b}$
    that is isomorphic to a direct sum of the modules
    $S^{\lambda + \vec e_1}$,
    $S^{\lambda + \vec e_2}$, \dots, $S^{\lambda + \vec e_b}$.
    Hence, $\projlift[b] = 0$ if $\mu = \lambda + \vec e_r$
    with $r > b$. Thus, \autoref{eqn2} becomes
    \begin{align}
        \label{eqn3}
        \left((n + 1) - a\right) \projlift
        +
        \sum_{r \leq b \leq a} \Repl_{b,a} \circ \projlift[b].
    \end{align}

    Finally, consider the case $b = a$.
    If $w$ is a word that contains $\lambda_a + 1$ occurrences
    of $a$, as is the case for all words in $M^{\lambda + \vec e_a}$,
    then $\Repl_{a,a}(w) = (\lambda_a + 1) w$, and so
    \autoref{eqn3} becomes
    \begin{gather*}
        \left((n + 1) + (\lambda_a + 1) - a\right) \projlift
        +
        \sum_{r \leq b < a} \Repl_{b,a} \circ \projlift[b].
        \qedhere
    \end{gather*}
\end{proof}

\begin{Theorem}
    \label{liftingeigenvectors2}
    For $\lambda = (\lambda_1, \lambda_2, \dots, \lambda_l) \vdash n$,
    $a \in \{1, 2, \dots, l + 1\}$, and
    $\mu = \lambda + \vec e_r$ with $1 \leq r \leq a$,
    \begin{equation*}
        \begin{multlined}
            \Rand_{n+1} \circ \projlift[a]
            - \projlift[a] \circ \Rand_n
            \\
            \shoveleft[.5in]
            =
            \left((n+1) + (\lambda_r + 1) - r\right) \,
            \projlift[a].
        \end{multlined}
    \end{equation*}
    In particular, if $v \in S^\lambda$ is an eigenvector of $\Rand_n$
    with eigenvalue $\varepsilon$, then
    either $\projlift[a](v) = 0$
    or $\projlift[a](v)$ is an eigenvector
    of $\Rand_{n+1}$ with eigenvalue
    \begin{gather}
        \label{eigenvalue-shift-formula}
        \varepsilon + (n+1) + (\lambda_r + 1) - r.
    \end{gather}
\end{Theorem}

Note that the coefficient of $\projlift[a]$ above is equal to
the size of $\mu$ plus the diagonal index of the unique cell of $\mu/\lambda$.
In particular, the eigenvalue of $v$ is ``shifted'' by a positive integer.

\begin{proof}
    The case for $r = a$ follows immediately from \autoref{MainLemma},
    since the summation in \autoref{eq:MainLemma} is $0$.
    We will deduce the general case from the case $r = a$.

    For $\mu = \lambda + \vec e_r$, we have
    \begin{equation}
        \label{inductionhypothesis}
        \begin{multlined}
            \Rand_{n+1} \circ \projlift[r] - \projlift[r] \circ \Rand_n
            \\ \shoveleft[.5in]
            = \big((n+1) + (\lambda_r + 1) - r\big) \, \projlift[r].
        \end{multlined}
    \end{equation}
    Apply the $\symm_{n+1}$--morphism $\Repl_{r,a}$ to
    \autoref{inductionhypothesis}:
    \begin{equation}
        \label{RabEqn}
        \begin{multlined}
            \Rand_{n+1} \circ \Repl_{r,a} \circ \projlift[r]
            - \Repl_{r,a} \circ \projlift[r] \circ \Rand_n
            \\ \shoveleft[.5in]
            =
            \big((n+1) + (\lambda_r + 1) - r\big) \, \Repl_{r,a} \circ \projlift[r].
        \end{multlined}
    \end{equation}
    First consider the left-hand side of \autoref{RabEqn}.
    The operators $\isoproj_\mu$ commute with
    $\symm_{n+1}$--morphisms (see \autoref{eq:ProjectionsCommuteWithMorphisms}),
    so the left-hand side of \autoref{RabEqn} is equal to
    the restriction to $S^\lambda$ of
    \begin{equation*}
        \begin{multlined}
            \Rand_{n+1} \circ \Repl_{r,a} \circ \isoproj_\mu \circ \Int_r
            - \Repl_{r,a} \circ \isoproj_\mu \circ \Int_r \circ \Rand_n
            \\ \shoveleft[.5in]
            =
            \Rand_{n+1} \circ \isoproj_\mu \circ \Repl_{r,a} \circ \Int_r
            - \isoproj_\mu \circ \Repl_{r,a} \circ \Int_r \circ \Rand_n.
        \end{multlined}
    \end{equation*}
    By the identity
    $\Repl_{r,a} \circ \Int_r = \Int_a + \Int_r \circ \Repl_{r,a}$
    (\autoref{ReplacementIntCommutation}), the above is equal to
    \begin{equation*}
        \begin{multlined}
            \left(
                \Rand_{n+1} \circ \isoproj_\mu \circ \Int_a
                + \Rand_{n+1} \circ \isoproj_\mu \circ \Int_r \circ \Repl_{r,a}
            \right)
            \\ \shoveleft[.5in]
            - \left(
                \isoproj_\mu \circ \Int_a \circ \Rand_n
                + \isoproj_\mu \circ \Int_r \circ \Repl_{r,a} \circ \Rand_n
            \right).
        \end{multlined}
    \end{equation*}
    Since $\Repl_{r,a}$ is a $\symm_{n}$--module morphism,
    it commutes with $\Rand_n$, and so this becomes
    \begin{equation*}
        \begin{multlined}
            \left(
                \Rand_{n+1} \circ \isoproj_\mu \circ \Int_a
                - \isoproj_\mu \circ \Int_a \circ \Rand_n
            \right)
            \\ \shoveleft[.5in]
            + \left(
                \Rand_{n+1} \circ \isoproj_\mu \circ \Int_r
                - \isoproj_\mu \circ \Int_r \circ \Rand_n
            \right) \circ \Repl_{r,a}.
        \end{multlined}
    \end{equation*}
    Restricting the above to $S^\lambda$ and using the identity in
    \autoref{inductionhypothesis},
    the left-hand side of \autoref{RabEqn} is equal to
    \begin{equation*}
        \begin{multlined}
            \left(
                \Rand_{n+1} \circ \projlift[a] - \projlift[a] \circ \Rand_n
            \right)
            \\ \shoveleft[.5in]
            + \big( (n+1) + (\lambda_r + 1) - r \big) \projlift[r] \circ \Repl_{r,a}.
        \end{multlined}
    \end{equation*}

    Next consider the right-hand side of \autoref{RabEqn}.
    Combining \autoref{eq:ProjectionsAreMorphisms} with
    the fact that $\Repl_{r,a}$ is an $\symm_{n+1}$--module morphism
    and using \autoref{ReplacementIntCommutation},
    \begin{equation*}
        \begin{multlined}
            \big((n+1) + (\lambda_r + 1) - r\big) \, \Repl_{r,a} \circ \projlift[r]
            \\ \shoveleft[.5in] =
            \big((n+1) + (\lambda_r + 1) - r\big) \, \Repl_{r,a} \circ \isoproj_\mu \circ \Int_r \Big|_{S^\lambda}
            \\ \shoveleft[.5in] =
            \big((n+1) + (\lambda_r + 1) - r\big) \, \isoproj_\mu \circ \Repl_{r,a} \circ \Int_r \Big|_{S^\lambda}
            \\  \shoveleft[.5in] =
            \big((n+1) + (\lambda_r + 1) - r\big) \, \isoproj_\mu \circ \Int_a \Big|_{S^\lambda}
            \\ \shoveleft[.75in]
            + \big((n+1) + (\lambda_r + 1) - r\big) \, \isoproj \circ \Int_r \circ \Repl_{r,a} \Big|_{S^\lambda}
            \\ \shoveleft[.5in] =
            \big((n+1) + (\lambda_r + 1) - r\big) \, \projlift[a]
            \\ \shoveleft[.75in]
            + \big((n+1) + (\lambda_r + 1) - r\big) \, \projlift[r] \circ \Repl_{r,a}.
        \end{multlined}
    \end{equation*}
    The result then follows by cancelling the common term.
\end{proof}

\subsection{Reformulation of the projection morphism}
\label{reformulation-of-the-projection-morphism}

This section describes a simple method to compute
$\projliftoper_{i}^{\lambda, \lambda + \vec e_i}$
that uses only the morphisms $\Repl_{a,b}$ and $\Int_i$.
In particular, this method does not require the explicit computation of any
projections.

The starting point is the observation that the image of $S^\lambda$ under
$\Int_{i}$ is contained in a submodule of $M^{\lambda + \vec e_i}$ that is
isomorphic to
    $S^{\lambda + \vec e_1}
        \oplus S^{\lambda + \vec e_2}
        \oplus \cdots
        \oplus S^{\lambda + \vec e_i}$
(\autoref{l:SubmoduleGeneratedByLifts}).
Hence,
$\isoproj_{\lambda + \vec e_1}
+ \isoproj_{\lambda + \vec e_2}
+ \cdots + \isoproj_{\lambda + \vec e_i}$
is the identity on this submodule. Consequently,
if $\Int_i^\lambda = \Int_i|_{S^\lambda}$, then
\begin{align*}
    \Int_{i}^\lambda
    & =
    \left(
        \isoproj_{\lambda + \vec e_1}
        +
        \isoproj_{\lambda + \vec e_2}
        +
        \cdots
        +
        \isoproj_{\lambda + \vec e_i}
    \right) \circ \Int_{i}^\lambda
    \\
    & =
    \projliftoper_{i}^{\lambda, \lambda + \vec e_1}
    +
    \projliftoper_{i}^{\lambda, \lambda + \vec e_2}
    +
    \cdots
    +
    \projliftoper_{i}^{\lambda, \lambda + \vec e_i}
\end{align*}
so that
\begin{equation*}
    \projliftoper_{i}^{\lambda, \lambda + \vec e_i}
    =
    \Int_{i}^\lambda
    -
    \sum_{1 \leq j < i}
        \projliftoper_{i}^{\lambda, \lambda + \vec e_j}.
\end{equation*}

\begin{Lemma}
    \label{projlift-identity}
    Let $\lambda = (\lambda_1, \lambda_2, \dots, \lambda_l) \vdash n$
    and $1 \leq r < a \leq l + 1$. Define
    \begin{equation*}
        \gamma_{r, a} = (\lambda_r - r) - (\lambda_a - a).
    \end{equation*}
    Then
    \begin{align}
        & \projliftoper_{a}^{\lambda, \lambda + \vec e_r}
        \nonumber
        \\
        & \quad =
            \sum_{r \leq b < a}
            \frac{\Repl_{b,a}}{\gamma_{r,a}}
            \circ \projliftoper_{b}^{\lambda, \lambda + \vec e_r}
        \label{reformulation1}
        \\
        & \quad =
            \sum_{r = b_0 < b_1 < \cdots < b_t < b_{t+1} = a}
            \frac{\Repl_{b_{t}, b_{t+1}}}{\gamma_{b_0,b_{t+1}}}
            \circ
            \frac{\Repl_{b_{t-1}, b_{t}}}{\gamma_{b_0,b_{t}}}
            \circ
            \cdots
            \circ
            \frac{\Repl_{b_{0}, b_{1}}}{\gamma_{b_0,b_{1}}}
            \circ
            \projliftoper_{r}^{\lambda, \lambda + \vec e_{r}}.
        \label{reformulation2}
    \end{align}
\end{Lemma}

\begin{proof}
    \autoref{reformulation1}
    is a direct consequence of \autoref{MainLemma} and \autoref{liftingeigenvectors2}:
    \begin{align*}
        & \sum_{r \leq b < a} \Repl_{b,a} \circ \projliftoper_{b}^{\lambda, \lambda + \vec e_r}
        \\
        & =
            \left(
                \Rand_{n+1} \circ \projliftoper_{a}^{\lambda, \lambda + \vec e_r}
                - \projliftoper_{a}^{\lambda, \lambda + \vec e_r} \circ \Rand_n
            \right)
        & \text{(\autoref{MainLemma})}
        \\
        & \qquad \qquad
            - \Big((n + 1) + (\lambda_a + 1) - a\Big) \projliftoper_{a}^{\lambda, \lambda + \vec e_r}
        \\
        & =
        \Big((n + 1) + (\lambda_r + 1) - r\Big) \projliftoper_{a}^{\lambda, \lambda + \vec e_r}
        & \text{(\autoref{liftingeigenvectors2})}
        \\
        & \qquad \qquad
            - \Big((n + 1) + (\lambda_a + 1) - a\Big) \projliftoper_{a}^{\lambda, \lambda + \vec e_r}
        \\
        & = \Big((\lambda_r - r) - (\lambda_a - a)\Big) \projliftoper_{a}^{\lambda, \lambda + \vec e_r}.
        &
    \end{align*}
    \autoref{reformulation2} follows from
    repeated application of \autoref{reformulation1}.
\end{proof}

\begin{Proposition}
    \label{reformulation-of-morphisms}
    Let $\lambda = (\lambda_1, \lambda_2, \dots, \lambda_l)$ be a partition of
    $n$ and $1 \leq i \leq l+1$.
    \begin{equation*}
        \projliftoper_{i}^{\lambda, \lambda + \vec e_i}
        =
        \sum_{1 \leq b_1 < \cdots < b_t < b_{t+1} = i}
        \frac{\Repl_{b_{t}, b_{t+1}}}{\gamma_{i,b_{t}}}
        \circ
        \frac{\Repl_{b_{t-1}, b_{t}}}{\gamma_{i,b_{t-1}}}
        \circ
        \cdots
        \circ
        \frac{\Repl_{b_{1}, b_{2}}}{\gamma_{i,b_{1}}}
        \circ
        \Int_{b_1}^\lambda.
    \end{equation*}
\end{Proposition}

\begin{proof}
    We proceed by induction on $i$.

    \emph{Base cases: $i=1$ and $i=2$.}
    If $i = 1$, then
        $\projliftoper_{1}^{\lambda, \lambda + \vec e_1}
        = \Int_{1}^\lambda$.
    If $i = 2$, then
    \begin{equation*}
        \projliftoper_{2}^{\lambda, \lambda + \vec e_2}
        =
        \Int_{2}^\lambda
        -
        \projliftoper_{2}^{\lambda, \lambda + \vec e_1}.
    \end{equation*}
    Applying \autoref{projlift-identity} to
    $\projliftoper_{2}^{\lambda, \lambda + \vec e_1}$, we have
    \begin{align*}
        \projliftoper_{2}^{\lambda, \lambda + \vec e_2}
        & =
        \Int_{2}^\lambda
        -
        \sum_{1 \leq b < 2} \frac{\Repl_{b,2}}{\gamma_{1,2}} \circ \projliftoper_{b}^{\lambda, \lambda + \vec e_1}
        \\
        & =
        \Int_{2}^\lambda
        -
        \frac{\Repl_{1,2}}{\gamma_{1,2}} \circ \projliftoper_{1}^{\lambda, \lambda + \vec e_1}
        =
        \Int_{2}^\lambda
        +
        \frac{\Repl_{1,2}}{\gamma_{2,1}} \circ \Int_{1}^\lambda.
    \end{align*}

    \emph{General Case.} Suppose the result holds for
    $\projliftoper_{j}^{\lambda, \lambda + \vec e_j}$
    for $j < i$. Then
    \begin{align*}
        &
        \sum_{1 \leq j < i} \projliftoper_{i}^{\lambda, \lambda + \vec e_j}
        \\
        & =
        \sum_{\substack{
                1 \leq j < i \\
                j = d_0 < d_1 < \cdots < d_t < d_{t+1} = i
              }}
        \frac{\Repl_{d_{t}, d_{t+1}}}{\gamma_{j,d_{t+1}}}
        \circ
        \cdots
        \circ
        \frac{\Repl_{d_{0}, d_{1}}}{\gamma_{j,d_{1}}}
        \circ
        \projliftoper_{j}^{\lambda, \lambda + \vec e_{j}}
        \\
        & =
        \sum_{\substack{
                1 \leq j < i \\
                1 \leq c_1 < \cdots < c_{s} < c_{s+1} = j \\
                j = d_0 < d_1 < \cdots < d_{t} < d_{t+1} = i
              }}
        \frac{\Repl_{d_{t}, d_{t+1}}}{\gamma_{j,d_{t+1}}}
        \circ
        \cdots
        \circ
        \frac{\Repl_{d_{0}, d_{1}}}{\gamma_{j,d_{1}}}
        \circ
        \frac{\Repl_{c_{s}, c_{s+1}}}{\gamma_{j,c_{s}}}
        \circ
        \cdots
        \circ
        \frac{\Repl_{c_{1}, c_{2}}}{\gamma_{j,c_{1}}}
        \circ
        \Int_{c_1}^\lambda.
    \end{align*}
    We now re-write this sum as a sum over all sequences of the form
    \begin{equation*}
        1 \leq b_1 < \cdots < b_{u} < b_{u+1} = i.
    \end{equation*}
    Such a sequence can be expressed as
    \begin{equation*}
        1 \leq c_1 < \cdots < c_{s} < c_{s+1} = j
        = d_0 < d_1 < \cdots < d_{t} < d_{t+1} = i
    \end{equation*}
    in several ways
    (one for each term $b_k$ of the sequence)
    so that the coefficient of
    \begin{equation*}
        \Repl_{b_{u}, b_{u+1}}
        \circ
        \Repl_{b_{u-1}, b_{u}}
        \circ
        \cdots
        \circ
        \Repl_{b_{1}, b_{2}}
        \circ
        \Int_{b_1}^\lambda
    \end{equation*}
    in the above summation is
    \begin{align*}
        &
        \sum_{k}
        \left(
            \frac{1}{\gamma_{b_k, b_{1}}}
            \frac{1}{\gamma_{b_k, b_{2}}}
            \cdots
            \frac{1}{\gamma_{b_k, b_{k-1}}}
        \right)
        \cdot
        \left(
            \frac{1}{\gamma_{b_k, b_{k+1}}}
            \frac{1}{\gamma_{b_k, b_{k+2}}}
            \cdots
            \frac{1}{\gamma_{b_k, b_{u+1}}}
        \right).
    \end{align*}
    Applying \autoref{vandermonde-lemma} with $x_i = \lambda_i - i$,
    we have that this coefficient is equal to
    \begin{equation*}
        -
            \frac{1}{\gamma_{b_{u+1}, b_{1}}}
            \frac{1}{\gamma_{b_{u+1}, b_{2}}}
            \cdots
            \frac{1}{\gamma_{b_{u+1}, b_{u}}}.
            \qedhere
    \end{equation*}
\end{proof}

\begin{Lemma}
    \label{vandermonde-lemma}
    Let $x_0, x_1, x_2, \dots, x_u$ be distinct real numbers.
    Then
    \begin{equation*}
        \sum_{0 \leq k \leq u} \,
            \prod_{\substack{0 \leq j \leq u \\ j \neq k}} \frac{1}{x_k-x_j}
        = 0.
    \end{equation*}
\end{Lemma}

\begin{proof}
    We first prove that the expression is symmetric in $x_0, \dots, x_u$.
    Consider the effect of swapping $x_i$ and $x_{i+1}$.
    If $k \neq i, i+1$ then swapping $x_i$ and $x_{i+1}$ maps
    \begin{equation*}
        \prod_{\substack{0 \leq j \leq u \\ j \neq k}} \frac{1}{x_k-x_j}
        =
        \frac{1}{
        (x_k - x_0)
        (x_k - x_1)
        \cdots
        (x_k - x_{i})
        (x_k - x_{i+1})
        \cdots
        (x_k - x_u)
        }
    \end{equation*}
    to itself.
    If $k$ is $i$ or $i+1$, then swapping $x_i$ and $x_{i+1}$
    exchanges
    \begin{equation*}
        \prod_{\substack{0 \leq j \leq u \\ j \neq i}} \frac{1}{x_i-x_j}
        =
        \frac{1}{
        \cdots
        (x_i - x_{i-1})
        (x_i - x_{i+1})
        (x_i - x_{i+2})
        \cdots
        }
    \end{equation*}
    and
    \begin{equation*}
        \prod_{\substack{0 \leq j \leq u \\ j \neq i+1}} \frac{1}{x_{i+1}-x_j}
        =
        \frac{1}{
        \cdots
        (x_{i+1} - x_{i-1})
        (x_{i+1} - x_{i})
        (x_{i+1} - x_{i+2})
        \cdots
        }.
    \end{equation*}

    Let $V = \prod_{i < j}(x_i - x_j)$ denote the Vandermonde determinant.
    Since $V$ is anti-symmetric and the above expression is symmetric,
    their product is anti-symmetric.
    Moreover, the product is a polynomial of degree less than $\deg(V)$.
    It follows that the product is $0$, since $V$ divides every
    anti-symmetric polynomial in the variables $x_0, \dots, x_u$.
    Since $V$ is nonzero, it follows that the summation is $0$.
\end{proof}

\subsection{Image of $\Rand_n$}
\label{image-of-r2r}

\autoref{liftingeigenvectors2} proves that $\projlift$ maps
eigenvectors of $\Rand^{\lambda}$ to eigenvectors of $\Rand^{\mu}$.
More accurately, the eigenvectors of $\Rand^{\lambda}$ are mapped
to eigenvectors of $\Rand^{\mu}$ that do not lie in its kernel.

We now prove that all \emph{non-kernel} eigenvectors of $\Rand^\mu$ can be
obtained in this way: more generally, that every element in the image of
$\Rand^\mu$ is a linear combination of elements in the images of
$\projlift$.

\begin{Proposition}
    \label{imageR2R}
    For every partition $\lambda = (\lambda_1, \lambda_2, \dots, \lambda_l)$ of $n$,
    \begin{align*}
        \im\left(\Rand|_{S^\lambda}\right)
        \subseteq
        \im\left(\TopR|_{S^\lambda}\right)
        \subseteq
        \sum_{1 \leq a \leq l}
            \im\left(\isoproj_{\lambda} \circ \Int_a |_{S^{\lambda - \vec e_a}}\right).
    \end{align*}
\end{Proposition}

\begin{proof}
    We first show that
    $\im(\Rand|_{S^\lambda}) \subseteq \im(\TopR|_{S^\lambda})$.
    Recall from \autoref{r2r=t2r.r2t} that $\Rand_n = \TopR_n \circ \RTop_n$.
    If $v \in S^\lambda$, then $\RTop_n(v) \in S^\lambda$ since $\RTop_n$ is
    multiplication by an element of the group algebra of $\symm_n$.
    Hence, $\Rand_n(v) = \TopR_n(\RTop_n(v)) \in \TopR_n(S^\lambda)$.

    Also by \autoref{r2r=t2r.r2t}, for any word $w$ we have
    $\TopR_n(w) = \Int_{w_n}(w_1 \cdots w_{n-1})$.
    Rewrite this as
    \begin{align*}
        \TopR_n(w) = \sum_{1 \leq a \leq n} \left(\Int_a \circ \proj_a\right) (w),
    \end{align*}
    where $\proj_a(w_1\cdots w_n)$ is $w_1 \cdots w_{n-1}$
    if $w_n = a$ and is $0$ otherwise.
    Thus, for every $v \in S^\lambda$ we have
    \begin{align*}
        \left(\isoproj_\lambda \circ \TopR_n\right)(v)
        = \sum_{1 \leq a \leq n} \left(\isoproj_\lambda \circ \Int_a \circ \proj_a\right) (v).
    \end{align*}

    We will show that
    $\proj_a(S^\lambda) \subseteq S^{\lambda - \vec e_a}$,
    which implies that $\TopR(v)$
    is a linear combination of elements of the form
    $\left(\isoproj_\lambda \circ \Int_a^{n-1}\right)(v')$
    with $v' \in S^{\lambda - \vec e_a}$.

    Recall that the submodule $S^\lambda$ is spanned by the elements $\w_t$,
    where $t \in \SYT_\lambda$.
    Since $\w_t \cdot \sigma = \w_{\sigma\inv(t)}$ for any permutation $\sigma$, we
    can suppose that $n$ appears at the end of row $a$ of $t$.
    Let $\bar t$ denote the tableau obtained from $t$ by removing the cell
    containing $n$. Then $\word(t) = \word(\bar t) a$.
    Let $c_1, \dots, c_l, n$ be the entries in the column of $t$ containing
    $n$. The identity permutation together with the transpositions $(c_1,
    n)$, \dots, $(c_l, n)$ form a transversal of $\ColStab(\bar t)$ in
    $\ColStab(t)$. Therefore, $N_t$ factors as
    $N_t = \left(\mathid - (c_1, n) - \cdots - (c_l, n)\right)N_{\bar t}.$
    Hence,
    \begin{align*}
        \proj_a(\w_t) &= \proj_a\left(\word(t) \cdot N_t\right) \\
                      &= \proj_a\left(\word(\bar t) a \cdot N_{\bar t}\right)
                       - \sum_{i=1}^l \proj_a\left(\word(\bar t) a \cdot (c_i, n) N_{\bar t}\right) \\
                      &= \word(\bar t) \cdot N_{\bar t} = \w_{\bar t},
    \end{align*}
    where the second last equality follows from the fact that $(c_i, n)$
    swaps the last letter of $\word(\bar t) a$ with a different letter.
    This shows that $\proj_a(\w_t) \in S^{\lambda - \vec e_a}$
    for every tableau $t$ of shape $\lambda$.
\end{proof}

\subsection{Construction of eigenspaces from kernels of $\Rand_m$}
\label{construction-of-eigenspaces-from-kernels}

At this point, due to \autoref{liftingeigenvectors2} and \autoref{imageR2R},
we know that all the eigenvectors of the restriction
of $\Rand_{n}$ to $S^{\lambda}$, except those that lie in its kernel, can be
obtained from eigenvectors of $\Rand_{n-1}$
using the linear transformations
$\isoproj_{\lambda} \circ \Int_a$.
More precisely,
if $v \in S^\lambda$ is an eigenvector of $\Rand_n$ that does
not belong to $\ker \Rand_n$, then $v$ is of the form
\begin{equation*}
    v = \left(\isoproj_{\lambda} \circ \Int_{a_1}\right)\left(v^{(1)}\right)
    + \cdots
    + \left(\isoproj_{\lambda} \circ \Int_{a_r}\right)\left(v^{(r)}\right),
\end{equation*}
where $v^{(j)}$ is an eigenvector of $\Rand_{n-1}$ in some Specht
module $S^\mu$ with $\lambda = \mu + \vec e_{a_j}$.

If $v^{(j)} \notin \ker \Rand_{n-1}$, then we can apply the same reasoning to
write each $v^{(j)}$ in the above form.
Continuing in this way, we can express $v$ as a sum of images of elements in
the kernel of some $\Rand_{n-k}$ under maps of the form:
\begin{equation*}
    \big(\isoproj_\lambda \circ \Int_{r_k}\big)
    \circ \cdots \circ
    \big(\isoproj_{\nu + \vec e_{r_1} + \vec e_{r_2}} \circ \Int_{r_2}\big)
    \circ \big(\isoproj_{\nu + \vec e_{r_1}} \circ \Int_{r_1}\big)
    \Big|_{S^\nu},
\end{equation*}
where $\lambda = \nu + \vec e_{r_1} + \vec e_{r_2} + \cdots + \vec e_{r_k}$.
We study these maps and prove:
\begin{itemize}
    \item
        the intermediate projections are not necessary if the $r_i$ are
        appropriately ordered (\autoref{suffices-to-project-at-end}):
        if $r_1 \leq r_2 \leq \cdots \leq r_k$, then the above simplifies to
        \begin{equation*}
            \isoproj_\lambda \circ \Int_{r_k}
            \circ \cdots \circ \Int_{r_2} \circ \Int_{r_1}
            \Big|_{S^\nu};
        \end{equation*}

    \item
        these maps are zero if $\lambda/\nu$ is not a horizontal strip
        (\autoref{image-of-composition-of-shuffles});

    \item
        the eigenspace decomposition for $\Rand^\lambda$
        is obtained by applying the above maps to $\ker \Rand^\nu$;
        there is one subspace for each horizontal strip $\lambda/\nu$
        (\autoref{directsumdecomposition}).
\end{itemize}

Much of this depends on the following observation relating $\Int_{a_i}$,
$\Theta_{a,{a_i}}$ and the shuffle product, where $a$ is a new letter.

\begin{Lemma}
    \label{CompositionOfLiftings}
    Let $a_1, \dots, a_k$ be letters and let $a$ be a letter such that $a \neq
    a_i$ for all $1 \leq i \leq k$.
    If $w$ does not contain an occurrence of $a$, then
    \begin{align*}
        \left( \Int_{a_k} \circ \cdots \circ \Int_{a_2} \circ \Int_{a_1} \right) (w)
        &= w \shuffle a_1 \shuffle a_2 \shuffle \cdots \shuffle a_k \\
        &= \left( \Repl_{a, a_k} \circ \cdots \circ \Repl_{a, a_2} \circ \Repl_{a, a_1} \right)(w \shuffle aa \cdots a).
    \end{align*}
\end{Lemma}

\begin{proof}
    Proceed by induction on $k$. It holds for $k = 1$.
    Suppose it holds for $k-1$.
    \begin{align*}
        & k! \, \Repl_{a, a_1}\left(w \shuffle \overbrace{aa \cdots a}^k\right)
          = \Repl_{a, a_1}\left(w \shuffle \overbrace{a \shuffle a \shuffle \cdots \shuffle a}^k\right)
        \\
        & = \big(w \shuffle a_1 \shuffle a \shuffle \cdots \shuffle a\big)
           + \big(w \shuffle a \shuffle a_1 \shuffle \cdots \shuffle a\big)
           + \cdots
        \\
        & = k \, \left((w \shuffle a_1) \shuffle \overbrace{a \shuffle \cdots \shuffle a}^{k-1}\right)
          = k! \, \left((w \shuffle a_1) \shuffle \overbrace{a \cdots a}^{k-1}\right).
    \end{align*}
    The result then follows from the induction hypothesis.
\end{proof}

\begin{Proposition}
    \label{suffices-to-project-at-end}
    Suppose that $\lambda$ is a partition of $n$ obtained from $\mu \vdash n-k$
    by adding $k$ cells in rows $r_1 \leq r_2 \leq \dots \leq r_k$ (with
    repetition allowed). Then
    \begin{equation*}
        \begin{multlined}
            \isoproj_\lambda \circ \Int_{r_k} \circ \cdots \circ \Int_{r_2} \circ \Int_{r_1}
            \Big|_{S^\mu}
            \\ \shoveleft[.1in]
            =
            \big(\isoproj_\lambda \circ \Int_{r_k}\big)
            \circ \cdots \circ
            \big(\isoproj_{\mu + \vec e_{r_1} + \vec e_{r_2}} \circ \Int_{r_2}\big)
            \circ \big(\isoproj_{\mu + \vec e_{r_1}} \circ \Int_{r_1}\big)
            \Big|_{S^\mu}.
        \end{multlined}
    \end{equation*}
\end{Proposition}

\begin{proof}
    To simplify the notation in this proof, we temporarily define
    \begin{align*}
        \mu^{(j)} & = \mu + \vec e_{r_1} + \vec e_{r_2} + \cdots + \vec e_{r_j}
        \\
        \pi_j & = \isoproj_{\mu^{(j)}} =
            \isoproj_{\mu + \vec e_{r_1} + \vec e_{r_2} + \cdots + \vec e_{r_j}}.
    \end{align*}

    By \autoref{l:SubmoduleGeneratedByLifts}, the subspace
    $\Int_{r_1}(S^\mu)$ is contained in a submodule $T$
    of $M^{\mu + \vec e_{r_1}}$ that is isomorphic to
    $
        S^{\mu + \vec e_1}
        \oplus S^{\mu + \vec e_2}
        \oplus \cdots
        \oplus S^{\mu + \vec e_{r_1}}.
    $
    Thus,
    \begin{equation*}
        T = T_1 \oplus T_{2} \oplus \cdots \oplus T_{r_1},
    \end{equation*}
    where $T_i = \isoproj_{\mu + \vec e_i}(T)$ is the projection
    of $T$ onto its $(\mu + \vec e_i)$--isotypic component.
    This induces a decomposition of $\Int_{r_1}(S^\mu)$:
    \begin{equation*}
        \Int_{r_1}\left(S^\mu\right)
        = N_1
        \oplus N_2
        \oplus \cdots
        \oplus N_{r_1},
    \end{equation*}
    where $N_i = \Int_{r_1}(S^\mu) \cap T_i = \isoproj_{\mu + \vec e_i}\left(\Int_{r_1}(S^\mu)\right)$.

    For every $u \in S^\mu$, write $\Int_{r_1}(u) = t_1 + t_2 + \cdots + t_{r_1}$ with $t_i \in N_i$.
    So, if we show
    \begin{equation*}
        \pi_j \Big( \left(\Int_{r_j} \circ \cdots \circ \Int_{r_2}\right)(N_i) \Big) = 0
    \end{equation*}
    for every $i < r_1$ and $j \geq 2$, then
    \begin{equation*}
        \begin{aligned}
            \pi_j \Big(
                \left(\Int_{r_j} \circ \cdots \circ \Int_{r_2}\right)\left(\Int_{r_1}(u)\right)
            \Big)
            &=
            \pi_j \Big(
                \left(\Int_{r_j} \circ \cdots \circ \Int_{r_2}\right)\left(t_{r_1}\right)
            \Big)
            \\
            &=
            \left(\pi_j \circ \Int_{r_j} \circ \cdots \circ \Int_{r_2}\right)\left(\pi_1(\Int_{r_1}(u))\right).
        \end{aligned}
    \end{equation*}
    The result then follows by induction on $j$ since
    $\pi_1(\Int_{r_1}(u)) \in S^{\mu + \vec e_{r_1}}$.

    Let $i < r_1$ and $j \geq 2$.
    By \autoref{CompositionOfLiftings},
    \begin{equation*}
        \left(\Int_{r_j} \circ \cdots \circ \Int_{r_2}\right)(N_i)
        \subseteq
        \left(\Theta_{a, r_j} \circ \cdots \circ \Theta_{a, r_2}\right)
        \left(N\right),
    \end{equation*}
    where $N$ is the $\symm_n$--submodule generated by
    \begin{equation*}
        \Big\{ v \cdot \overbrace{aa \cdots a}^{\text{~$j-1$}} : v \in S^{\mu + \vec e_i} \Big\}.
    \end{equation*}
    We will show that $N$ has no submodule
    that is isomorphic to $S^{\mu^{(j)}}$, which
    implies that
    $\left(\Theta_{a, r_j} \circ \cdots \circ \Theta_{a, r_2}\right)
    \left(N\right)$
    has no submodule isomorphic to $S^{\mu^{(j)}}$,
    since each $\Theta_{a, r_k}$ is a $\symm_n$--module morphism.

    Suppose $N$ contains a submodule that is isomorphic to $S^{\mu^{(j)}}$.
    Since
    \begin{equation*}
        N \cong \Ind_{\symm_{n-j+1} \times \symm_{j-1}}^{\symm_{n}}
                    \left(S^{\mu + \vec e_{i}} \otimes S^{(j-1)}\right),
    \end{equation*}
    the branching rule for the symmetric groups implies that ${\mu^{(j)}}$ is
    obtained from $\mu + \vec e_i$ by adding $j-1$ cells.
    Hence,
    ${\mu^{(j)}} = \mu + \vec e_i + \vec e_{i_2} + \vec e_{i_3} + \cdots + \vec e_{i_j}$,
    or equivalently
    \begin{equation*}
        \vec e_{r_1} + \vec e_{r_2} + \vec e_{r_3} + \cdots + \vec e_{r_j}
        =
        \vec e_i + \vec e_{i_2} + \vec e_{i_3} + \cdots + \vec e_{i_j}.
    \end{equation*}
    Since $i < r_1$, it follows that there exists $l \in \{2, 3, \dots, j\}$
    such that $i = r_l$. This leads to the contradiction:
    $i < r_1 \leq r_l = i$.
    Hence, $N$ does not contain a submodule that is isomorphic
    to $S^{\mu^{(j)}}$.
\end{proof}

For a skew partition $\lambda / \nu$ with $\lambda$ obtained from $\nu$ by
adding $k$ cells in rows $r_1 \leq r_2 \leq \dots \leq r_k$, write
\begin{gather*}
    \Int^{\lambda/\nu} = \Int_{r_k} \circ \cdots \circ \Int_{r_2} \circ \Int_{r_1}.
\end{gather*}
The next result proves that $\isoproj_\lambda \circ \Int^{\lambda/\nu} = 0$
if $\lambda/\nu$ is not a horizontal strip.

\begin{Remark}
    The definition of $\Int^{\lambda / \nu}$ does not depend on the ordering
    $r_1, r_2, \dots, r_k$ since $\Int_a \circ \Int_b = \Int_b \circ \Int_a$
    for all $a, b$. However, the identity in the
    \autoref{suffices-to-project-at-end} is not independent of the ordering of
    $r_1, r_2, \dots, r_k$: the identity does not hold in general if we do not
    have $r_1 \leq r_2 \leq \dots \leq r_k$.
\end{Remark}

\begin{Proposition}
    \label{image-of-composition-of-shuffles}
    Suppose that $\lambda$ is a partition of $n$ obtained from $\nu \vdash n-k$
    by adding $k$ cells in rows $r_1, r_2, \dots, r_k$ (with repetition
    allowed). Then
    \begin{align*}
        \left( \Int_{r_k} \circ \cdots \circ \Int_{r_2} \circ \Int_{r_1} \right) (S^\nu)
    \end{align*}
    is contained in a submodule of $M^{\lambda}$
    that is a homomorphic image of the $\symm_n$--module
    $\Ind_{\symm_{n-k} \times \symm_k}^{\symm_n}(S^\nu \otimes S^{(k)})$.
    In particular, if
    two of the cells $r_1, r_2, \dots, r_k$
    lie in the same column, then
    \begin{align*}
        S^\lambda \cap \left( \Int_{r_k} \circ \cdots \circ \Int_{r_2} \circ \Int_{r_1} \right) (S^\nu)
        = 0.
    \end{align*}
    Consequently,
    \begin{align*}
        \isoproj_\lambda \circ \Int_{r_k} \circ \cdots \circ \Int_{r_2} \circ \Int_{r_1} \Big|_{S^\nu} = 0.
    \end{align*}
\end{Proposition}

\begin{proof}
    Let $N$ be the $\symm_n$--submodule generated by the elements
    $\left\{ x \cdot aa \cdots a : x \in S^\nu \right\}$,
    where $a = \ell(\nu) + 1$.
    Then, by \autoref{CompositionOfLiftings},
    $\left( \Int_{r_k} \circ \cdots \circ \Int_{r_1} \right) (S^\nu)$
    is contained in
    $\left( \Repl_{a, r_k} \circ \cdots \circ \Repl_{a, r_1} \right)(N)$.
    Since
    $N \cong \Ind_{\symm_{n-k} \times \symm_k}^{\symm_n}(S^\nu \otimes S^{(k)})$
    \cite{James1978}, this proves the first statement.

    For the second statement, recall that the branching rule for the symmetric
    groups says that $S^\lambda$ is isomorphic to a direct summand of
    $\Ind_{\symm_{n-k} \times \symm_k}^{\symm_n}(S^\nu \otimes S^{(k)})$
    if and only if no two cells of the skew partition $\lambda / \nu$ are contained in the
    same column.
\end{proof}

We are now ready to prove the essential and final part of
\autoref{thm:eigenspace-decomposition}.

\begin{Proposition}
    \label{directsumdecomposition}
    Let $\lambda$ be a partition of $n$. Then
    \begin{gather*}
        S^\lambda
        = \bigoplus_{
            \substack{
                \lambda/\nu \text{~is a} \\
                \text{horizontal strip}
            }
        }
        (\isoproj_\lambda \circ \Int^{\lambda/\nu})\left(\ker \Rand^\nu \right)
    \end{gather*}
    and the action of $\Rand_n$ on
        $(\isoproj_\lambda \circ \Int^{\lambda/\nu})\left(\ker \Rand^\nu \right)$
    is scalar multiplication by
    \begin{equation*}
        \begin{aligned}
            \eig\left(\lambda/\nu\right)
            & =
            \binom{|\lambda|+1}{2} - \binom{|\nu|+1}{2}
                + \diagonalindex\left(\lambda/\nu\right).
        \end{aligned}
    \end{equation*}
    Consequently, the $\varepsilon$--eigenspace of $\Rand_n$ acting on
    $S^\lambda$ is
    \begin{gather*}
        \bigoplus_{
            \substack{
                \eig(\lambda/\nu) = \varepsilon \text{~and~} \\
                \lambda/\nu \text{~is a horizontal strip}
            }
        }
        (\isoproj_\lambda \circ \Int^{\lambda/\nu})\left(\ker \Rand^\nu \right)
    \end{gather*}
    whose dimension is
    \begin{gather*}
        \sum_{
            \substack{
                \eig(\lambda/\nu) = \varepsilon \text{~and~} \\
                \lambda/\nu \text{~is a horizontal strip}
            }
        }
        \ndestab^\nu,
    \end{gather*}
    where $\ndestab^\nu$ is the number of desarrangement tableaux of shape $\nu$.
\end{Proposition}

\begin{proof}
    The second statement follows from
    \autoref{suffices-to-project-at-end} and \autoref{liftingeigenvectors2}.

    The argument preceding \autoref{suffices-to-project-at-end} shows that we can
    write
    \begin{gather*}
        S^\lambda
        = \sum_{\nu \subseteq \lambda} (\isoproj_\lambda \circ \Int^{\lambda/\nu}) \left(\ker \Rand^\nu \right).
    \end{gather*}
    By \autoref{image-of-composition-of-shuffles},
    the restriction of $\isoproj_\lambda \circ \Int^{\lambda/\nu}$ to $S^\nu$
    is zero if $\lambda/\nu$ is not a horizontal strip.
    We prove this sum is direct via a dimension counting argument.

    Let $f^\lambda$ denote the number of standard tableaux of shape $\lambda$.
    Then
    \begin{align*}
        f^\lambda
        = \dim \left(S^\lambda\right)
        = \sum_{\lambda / \nu} \dim\left(\isoproj_\lambda \circ \Int^{\lambda/\nu} (\ker \Rand^\nu)\right)
        \leq \sum_{\lambda / \nu} \dim \left(\ker \Rand^\nu\right).
    \end{align*}
    By \autoref{lemma:kernel-dimensions}, the dimension of $\ker \Rand^\nu$ is
    $\ndestab^\nu$, the number of desarrangement tableaux of shape $\nu$.
    Hence, the right hand side above is the number of desarrangement tableaux
    $Q$ such that $\lambda / \shape(Q)$ is a horizontal strip.
    By \cite[Proposition~VI.9.4]{RSW2014}, there is a bijection between these tableaux and
    the set of standard tableaux of shape $\lambda$.
    Hence, the above inequality is an equality.
    \qedhere

\end{proof}

\newpage
\section{Appendix: Tables of random-to-random eigenvalues}
\label{appendix:figures}

This section contains tables with the eigenvalues (scaled by $n^2$) for the
random-to-random shuffle acting on permutations of size $n \leq 6$
(\cf \autoref{thm:eigenvalues-of-R2R}).

\begin{center}
    \resizebox*{\textwidth}{!}{
    \begin{tabular}{c|ccc|cccc}
        \toprule
        \eigenvaluestableheader \\ \midrule
        $\tikztableauscript{{\null,\null},}$ &   1 &   1 &   1 &   3 &   0 &   1 &   4 \\ \midrule
        $      \tikztableauscript{{X},{X},}$ &   1 &   1 &   1 &   3 &   3 &   0 &   0 \\ \bottomrule
    \end{tabular}
    }

    \bigskip

    \resizebox*{\textwidth}{!}{
    \begin{tabular}{c|ccc|cccc}
        \toprule
        \eigenvaluestableheader \\ \midrule
        $\tikztableauscript{{\null,\null,\null},}$ &   1 &   1 &   1 &   6 &   0 &   0 &   9 \\\midrule
        $      \tikztableauscript{{X,\null},{X},}$ &   1 &   2 &   2 &   6 &   3 &   0 &   4 \\
        $          \tikztableauscript{{X,X},{X},}$ &   1 &   2 &   2 &   6 &   6 &   0 &   0 \\\midrule
        $    \tikztableauscript{{X},{X},{\null},}$ &   1 &   1 &   1 &   6 &   3 &   0 &   1 \\\bottomrule
    \end{tabular}
    }

    \bigskip

    \resizebox*{\textwidth}{!}{
    \begin{tabular}{c|ccc|cccc}
        \toprule
        \eigenvaluestableheader \\ \midrule
        $\tikztableauscript{{\null,\null,\null,\null},}$ &   1 &   1 &   1 &  10 &   0 &   0 &  16 \\\midrule
        $      \tikztableauscript{{X,\null,\null},{X},}$ &   1 &   3 &   3 &  10 &   3 &   0 &  10 \\
        $          \tikztableauscript{{X,X,\null},{X},}$ &   1 &   3 &   3 &  10 &   6 &   0 &   6 \\
        $              \tikztableauscript{{X,X,X},{X},}$ &   1 &   3 &   3 &  10 &  10 &   0 &   0 \\\midrule
        $          \tikztableauscript{{X,X},{X,\null},}$ &   1 &   2 &   2 &  10 &   6 &   0 &   4 \\
        $              \tikztableauscript{{X,X},{X,X},}$ &   1 &   2 &   2 &  10 &  10 &   0 &   0 \\\midrule
        $    \tikztableauscript{{X,\null},{X},{\null},}$ &   1 &   3 &   3 &  10 &   3 &   0 &   6 \\
        $        \tikztableauscript{{X,X},{X},{\null},}$ &   1 &   3 &   3 &  10 &   6 &   0 &   2 \\
        $            \tikztableauscript{{X,X},{X},{X},}$ &   1 &   3 &   3 &  10 &  10 &   0 &   0 \\\midrule
        $          \tikztableauscript{{X},{X},{X},{X},}$ &   1 &   1 &   1 &  10 &  10 &   0 &   0 \\\bottomrule
    \end{tabular}
    }
\end{center}

\begin{figure}[p]
    \resizebox*{!}{0.9\textheight}{
        \begin{tabular}{c|ccc|cccc}
            \toprule
            \eigenvaluestableheader \\ \midrule
            $\tikztableauscript{{\null,\null,\null,\null,\null},}$ &   1 &   1 &   1 &  15 &   0 &   0 &  25 \\\midrule
            $      \tikztableauscript{{X,\null,\null,\null},{X},}$ &   1 &   4 &   4 &  15 &   3 &   0 &  18 \\
            $          \tikztableauscript{{X,X,\null,\null},{X},}$ &   1 &   4 &   4 &  15 &   6 &   0 &  14 \\
            $              \tikztableauscript{{X,X,X,\null},{X},}$ &   1 &   4 &   4 &  15 &  10 &   0 &   8 \\
            $                  \tikztableauscript{{X,X,X,X},{X},}$ &   1 &   4 &   4 &  15 &  15 &   0 &   0 \\\midrule
            $          \tikztableauscript{{X,X,\null},{X,\null},}$ &   1 &   5 &   5 &  15 &   6 &   0 &  11 \\
            $              \tikztableauscript{{X,X,\null},{X,X},}$ &   1 &   5 &   5 &  15 &  10 &   0 &   7 \\
            $              \tikztableauscript{{X,X,X},{X,\null},}$ &   1 &   5 &   5 &  15 &  10 &   0 &   5 \\
            $                  \tikztableauscript{{X,X,X},{X,X},}$ &   2 &   5 &  10 &  15 &  15 &   0 &   0 \\\midrule
            $    \tikztableauscript{{X,\null,\null},{X},{\null},}$ &   1 &   6 &   6 &  15 &   3 &   0 &  13 \\
            $        \tikztableauscript{{X,X,\null},{X},{\null},}$ &   1 &   6 &   6 &  15 &   6 &   0 &   9 \\
            $            \tikztableauscript{{X,X,\null},{X},{X},}$ &   1 &   6 &   6 &  15 &  10 &   0 &   7 \\
            $            \tikztableauscript{{X,X,X},{X},{\null},}$ &   1 &   6 &   6 &  15 &  10 &   0 &   3 \\
            $                \tikztableauscript{{X,X,X},{X},{X},}$ &   2 &   6 &  12 &  15 &  15 &   0 &   0 \\\midrule
            $        \tikztableauscript{{X,X},{X,\null},{\null},}$ &   1 &   5 &   5 &  15 &   6 &   0 &   7 \\
            $            \tikztableauscript{{X,X},{X,\null},{X},}$ &   1 &   5 &   5 &  15 &  10 &   0 &   5 \\
            $            \tikztableauscript{{X,X},{X,X},{\null},}$ &   1 &   5 &   5 &  15 &  10 &   0 &   3 \\
            $                \tikztableauscript{{X,X},{X,X},{X},}$ &   2 &   5 &  10 &  15 &  15 &   0 &   0 \\\midrule
            $          \tikztableauscript{{X,\null},{X},{X},{X},}$ &   1 &   4 &   4 &  15 &  10 &   0 &   6 \\
            $          \tikztableauscript{{X,X},{X},{X},{\null},}$ &   1 &   4 &   4 &  15 &  10 &   0 &   2 \\
            $              \tikztableauscript{{X,X},{X},{X},{X},}$ &   2 &   4 &   8 &  15 &  15 &   0 &   0 \\\midrule
            $        \tikztableauscript{{X},{X},{X},{X},{\null},}$ &   1 &   1 &   1 &  15 &  10 &   0 &   1 \\\bottomrule
        \end{tabular}
    }
    \caption{Eigenvalues (scaled by $n^2$) for the random-to-random shuffle
        acting on permutations of size $n = 5$.}
    \label{fig:eigenvalues-of-R2R5}
\end{figure}

\begin{figure}[p]
    \resizebox*{!}{0.9\textheight}{
        \begin{tabular}{c|ccc|cccc}
            \toprule
            \eigenvaluestableheader \\ \midrule
            $\tikztableauscript{{\null,\null,\null,\null,\null,\null},}$ &   1 &   1 &   1 &  21 &   0 &   0 &  36 \\\midrule
            $      \tikztableauscript{{X,\null,\null,\null,\null},{X},}$ &   1 &   5 &   5 &  21 &   3 &   0 &  28 \\
            $          \tikztableauscript{{X,X,\null,\null,\null},{X},}$ &   1 &   5 &   5 &  21 &   6 &   0 &  24 \\
            $              \tikztableauscript{{X,X,X,\null,\null},{X},}$ &   1 &   5 &   5 &  21 &  10 &   0 &  18 \\
            $                  \tikztableauscript{{X,X,X,X,\null},{X},}$ &   1 &   5 &   5 &  21 &  15 &   0 &  10 \\
            $                      \tikztableauscript{{X,X,X,X,X},{X},}$ &   1 &   5 &   5 &  21 &  21 &   0 &   0 \\\midrule
            $          \tikztableauscript{{X,X,\null,\null},{X,\null},}$ &   1 &   9 &   9 &  21 &   6 &   0 &  20 \\
            $              \tikztableauscript{{X,X,\null,\null},{X,X},}$ &   1 &   9 &   9 &  21 &  10 &   0 &  16 \\
            $              \tikztableauscript{{X,X,X,\null},{X,\null},}$ &   1 &   9 &   9 &  21 &  10 &   0 &  14 \\
            $                  \tikztableauscript{{X,X,X,\null},{X,X},}$ &   2 &   9 &  18 &  21 &  15 &   0 &   9 \\
            $                  \tikztableauscript{{X,X,X,X},{X,\null},}$ &   1 &   9 &   9 &  21 &  15 &   0 &   6 \\
            $                      \tikztableauscript{{X,X,X,X},{X,X},}$ &   3 &   9 &  27 &  21 &  21 &   0 &   0 \\\midrule
            $    \tikztableauscript{{X,\null,\null,\null},{X},{\null},}$ &   1 &  10 &  10 &  21 &   3 &   0 &  22 \\
            $        \tikztableauscript{{X,X,\null,\null},{X},{\null},}$ &   1 &  10 &  10 &  21 &   6 &   0 &  18 \\
            $            \tikztableauscript{{X,X,\null,\null},{X},{X},}$ &   1 &  10 &  10 &  21 &  10 &   0 &  16 \\
            $            \tikztableauscript{{X,X,X,\null},{X},{\null},}$ &   1 &  10 &  10 &  21 &  10 &   0 &  12 \\
            $                \tikztableauscript{{X,X,X,\null},{X},{X},}$ &   2 &  10 &  20 &  21 &  15 &   0 &   9 \\
            $                \tikztableauscript{{X,X,X,X},{X},{\null},}$ &   1 &  10 &  10 &  21 &  15 &   0 &   4 \\
            $                    \tikztableauscript{{X,X,X,X},{X},{X},}$ &   3 &  10 &  30 &  21 &  21 &   0 &   0 \\\midrule
            $              \tikztableauscript{{X,X,X},{X,\null,\null},}$ &   1 &   5 &   5 &  21 &  10 &   0 &  12 \\
            $                  \tikztableauscript{{X,X,X},{X,X,\null},}$ &   2 &   5 &  10 &  21 &  15 &   0 &   7 \\
            $                      \tikztableauscript{{X,X,X},{X,X,X},}$ &   2 &   5 &  10 &  21 &  21 &   0 &   0 \\\bottomrule
        \end{tabular}
    }
    \caption{Eigenvalues (scaled by $n^2$) for the random-to-random shuffle
        acting on permutations of size $n = 6$; continues in
        \autoref{fig:eigenvalues-of-R2R6b} and \autoref{fig:eigenvalues-of-R2R6c}.}
    \label{fig:eigenvalues-of-R2R6a}
\end{figure}

\begin{figure}[p]
    \resizebox*{!}{0.9\textheight}{
        \begin{tabular}{c|ccc|cccc}
            \toprule
            \eigenvaluestableheader \\ \midrule
            $\tikztableauscript{{X,X,\null},{X,\null},{\null},}$ &   1 &  16 &  16 &  21 &   6 &   0 &  15 \\
            $    \tikztableauscript{{X,X,\null},{X,\null},{X},}$ &   1 &  16 &  16 &  21 &  10 &   0 &  13 \\
            $    \tikztableauscript{{X,X,\null},{X,X},{\null},}$ &   1 &  16 &  16 &  21 &  10 &   0 &  11 \\
            $    \tikztableauscript{{X,X,X},{X,\null},{\null},}$ &   1 &  16 &  16 &  21 &  10 &   0 &   9 \\
            $        \tikztableauscript{{X,X,\null},{X,X},{X},}$ &   2 &  16 &  32 &  21 &  15 &   0 &   8 \\
            $        \tikztableauscript{{X,X,X},{X,\null},{X},}$ &   2 &  16 &  32 &  21 &  15 &   0 &   6 \\
            $        \tikztableauscript{{X,X,X},{X,X},{\null},}$ &   2 &  16 &  32 &  21 &  15 &   0 &   4 \\
            $            \tikztableauscript{{X,X,X},{X,X},{X},}$ &   6 &  16 &  96 &  21 &  21 &   0 &   0 \\\midrule
            $  \tikztableauscript{{X,\null,\null},{X},{X},{X},}$ &   1 &  10 &  10 &  21 &  10 &   0 &  14 \\
            $  \tikztableauscript{{X,X,\null},{X},{X},{\null},}$ &   1 &  10 &  10 &  21 &  10 &   0 &  10 \\
            $      \tikztableauscript{{X,X,\null},{X},{X},{X},}$ &   2 &  10 &  20 &  21 &  15 &   0 &   8 \\
            $      \tikztableauscript{{X,X,X},{X},{X},{\null},}$ &   2 &  10 &  20 &  21 &  15 &   0 &   3 \\
            $          \tikztableauscript{{X,X,X},{X},{X},{X},}$ &   4 &  10 &  40 &  21 &  21 &   0 &   0 \\\bottomrule
        \end{tabular}
    }
    \caption{Eigenvalues (scaled by $n^2$) for the random-to-random shuffle
        acting on permutations of size $n = 6$; continues in
        \autoref{fig:eigenvalues-of-R2R6a} and \autoref{fig:eigenvalues-of-R2R6c}.}
    \label{fig:eigenvalues-of-R2R6b}
\end{figure}

\begin{figure}[p]
    \resizebox*{!}{0.9\textheight}{
        \begin{tabular}{c|ccc|cccc}
            \toprule
            \eigenvaluestableheader \\ \midrule
            $    \tikztableauscript{{X,X},{X,X},{\null,\null},}$ &   1 &   5 &   5 &  21 &  10 &   0 &   8 \\
            $        \tikztableauscript{{X,X},{X,X},{X,\null},}$ &   2 &   5 &  10 &  21 &  15 &   0 &   5 \\
            $            \tikztableauscript{{X,X},{X,X},{X,X},}$ &   2 &   5 &  10 &  21 &  21 &   0 &   0 \\\midrule
            $  \tikztableauscript{{X,X},{X,\null},{X},{\null},}$ &   1 &   9 &   9 &  21 &  10 &   0 &   8 \\
            $      \tikztableauscript{{X,X},{X,\null},{X},{X},}$ &   2 &   9 &  18 &  21 &  15 &   0 &   6 \\
            $      \tikztableauscript{{X,X},{X,X},{X},{\null},}$ &   2 &   9 &  18 &  21 &  15 &   0 &   3 \\
            $          \tikztableauscript{{X,X},{X,X},{X},{X},}$ &   4 &   9 &  36 &  21 &  21 &   0 &   0 \\\midrule
            $\tikztableauscript{{X,\null},{X},{X},{X},{\null},}$ &   1 &   5 &   5 &  21 &  10 &   0 &   8 \\
            $    \tikztableauscript{{X,X},{X},{X},{X},{\null},}$ &   2 &   5 &  10 &  21 &  15 &   0 &   2 \\
            $        \tikztableauscript{{X,X},{X},{X},{X},{X},}$ &   2 &   5 &  10 &  21 &  21 &   0 &   0 \\\midrule
            $      \tikztableauscript{{X},{X},{X},{X},{X},{X},}$ &   1 &   1 &   1 &  21 &  21 &   0 &   0 \\\bottomrule
        \end{tabular}
    }
    \caption{Eigenvalues (scaled by $n^2$) for the random-to-random shuffle
        acting on permutations of size $n = 6$; continues in
        \autoref{fig:eigenvalues-of-R2R6a} and \autoref{fig:eigenvalues-of-R2R6b}.}
    \label{fig:eigenvalues-of-R2R6c}
\end{figure}

\newpage

\bibliographystyle{amsalpha}
\bibliography{references}

\end{document}